\documentclass{amsart}
\usepackage{all2017} 
\usepackage{tree-commands2} 
\usepackage[colorlinks, citecolor={black}, linkcolor={{black}}, urlcolor={blue}]{hyperref} 

\usepackage[T1]{fontenc}
\usepackage[utf8]{inputenc}
\usepackage[greek,english]{babel}
\usepackage{caption}

\usepackage[
backend=biber, 
natbib=true,
sorting=ynt,
bibencoding=utf8,
style=alphabetic,
labelalpha=true,
maxalphanames=5, 
minalphanames=2,	
]{biblatex} 

\calclayout

\title[Catalan-many tropical morphisms to trees]{Catalan-many tropical morphisms to trees;
	\\ \small{Part I: Constructions}}
\author{Jan Draisma and Alejandro Vargas}

\begin{document}  		
	\begin{abstract}   
		We investigate the \emph{tree gonality} of a
		genus-$g$ metric graph, defined as the minimum
		degree of a tropical morphism from any
		tropical modification of the metric graph to a
		metric tree. We give a combinatorial
		constructive proof that this number is at most
		$\lceil g/2 \rceil + 1$, a fact whose proofs
		so far required an algebro-geometric detour
		via special divisors on curves. For even
		genus, the tropical morphism which realizes the
		bound belongs to a family of tropical
		morphisms that is pure of dimension $3g-3$ and
		that has a generically finite-to-one map onto
		the moduli space of genus-$g$ metric graphs. Our methods focus on the study of such families. 
		This is part I in a series of two
		papers: in part I we fix the
		combinatorial type of the metric graph to show a bound on tree-gonality, while
		in part II we vary the combinatorial type and
		show that the number of tropical morphisms,
		counted with suitable multiplicities, is the
		same Catalan number that counts morphisms
		from a general genus-$g$ curve to the projective line.
	\end{abstract}
	
	\maketitle  
	
	\subsection*{Introduction} 
	We investigate the \emph{tree gonality} of metric graphs. This concept has its roots in algebraic geometry. The \emph{gonality} of a genus-$g$ curve $X$ is the minimum degree of a non-constant morphism from $X$ to $\mathbb P^1$. The non-constant morphisms $\phi: X \to \mathbb P^1$ are in one-to-one correspondence with linear series of degree $d$ and dimension~$1$ on $X$, so equivalently the gonality of $X$ is the minimum degree of such linear series. Brill-Noether theory for algebraic curves establishes that the gonality of $X$ is at most $\lceil g/2 \rceil + 1 $. The set of curves for which this bound is sharp is dense in the moduli space of genus-$g$ curves~$\mathcal M_g$.

	A striking interplay between
	graphs and algebraic curves has been discovered over the last two
	decades. In tropical geometry the objects that play the role of algebraic curves and of the projective line are metric graphs and metric trees, respectively. This goes beyond a
	formal analogy of concepts: specialization
	lemmas and correspondence theorems between both
	settings are an active field of research at the
	interface between tropical and non-archimedean geometry 
	(see, e.g., \cite{bak08, abbr15, abbr15a}). 
	
	There are two definitions of gonality for (metric) graphs which are relevant to our story.
	A chip-firing game on a (metric) graph yields a divisor theory (see
	\cite{bn07, hkn13}). The \emph{divisorial gonality} of a (metric) graph is the minimum degree of a rank-1 divisor. 
	On the other hand, certain maps between metric graphs with suitable balancing
	conditions bahave similarly to morphisms between
	curves (see \cite{bn09, bbm11, cha13, cap14, mik07}). We call these maps \emph{tropical morphisms}. The
	\emph{tree gonality} of a metric graph is the minimum
	degree of a tropical morphism from any
	\emph{tropical modification} of the metric graph to a
	metric tree. Tropical modification is an operation 
	which retracts or attaches edges ending in a valency-1 vertex, and is useful to satisfy balancing conditions. 
	
	Divisorial gonality mimics the definition of gonality of $X$ via linear series, while tree gonality mimics the one via non-constant morphisms. Unlike the algebro-geometric setting, in the tropical world divisorial gonality and tree-gonality do not coincide. If $\tmor : \mG \to \mT$ is a tropical morphism to a tree, then the divisor $\sum_{x' \in \inv \tmor(x) } m_\dtmor(x') x'$ has rank at least~1 for any choice of $x$ in $\mT$ (see Subsection~\ref{subsection-tropical-morphisms} for notation). So tree gonality is an upper bound of divisorial gonality. On the other hand, there are families of graphs with divisorial gonality equal to 2, and tree gonality equal to $\lceil g/2 \rceil + 1 $, where $g$ is the first Betti number of the graph. At the time of writing, perhaps the most famous open question on divisorial gonality is the following:
	
	
	\begin{conj}[Baker; Conjecture 3.10 in \cite{bak08}] \label{conjecture-baker}
		The divisorial gonality of a finite connected graph $G$ is at most $\lceil g/2  \rceil + 1 $, where $g$ is the first Betti number of $G$.
	\end{conj}
	
	This conjecture has attracted a sizeable amount of attention from the community. If one replaces $G$ with a metric graph $\mG$, then the result is proven in \cite[Theorem 3.12]{bak08}. The proof requires an algebro-geometric detour via special divisors on curves. There is wide interest for a purely combinatorial proof, as it is hoped it may yield methods better suited to tackle Conjecture~\ref{conjecture-baker}.
	Our main contribution is such a purely combinatorial proof:
	
	\begin{thm} \label{theorem-gonality}
		The tree gonality of a genus-$g$ metric graph $\mG$ is at most $\lceil g/2\rceil  + 1$.
	\end{thm}
	
	In fact, we manage to construct tropical morphisms that witness the bound $\lceil g/2\rceil + 1$ on tree-gonality, so beyond being combinatorial, the proof is also \emph{effective}. There is some computer code and visualizations available at \cite{dra}.

	
	\subsection*{Proof idea} The main idea behind our proof is short, elegant, and ubiquitous: we first analyse a favourable case and then continuously deform the result. We need an initial family of graphs $\mG$ for which we have a construction of a tropical morphism $\tmor: \mG \to \mT$ attaining the gonality bound; and a procedure to deform $\tmor$ into $\bar \tmor : \bar \mG \to \bar \mT$ for any given~$\bar \mG$ via a deformation path in $\MTrop g$, the moduli space of metric graphs. We argue that for this strategy to work we need that $\tmor$, and all but finitely many of the points in the deformation path to~$\bar \tmor$, belong to cones of tropical morphisms that we call \emph{full-dimensional}.  Intuitively, these morphisms \emph{depend} on the right number of parameters to cover $\MTrop g$. 
	
	We use this idea for both parts of this series of articles, only the ingredients change. In Part~I we use constructions from \cite{cd18} as initial families, which gives us a tropical morphism for each possible trivalent graph $H$ (this $H$ is the \emph{combinatorial type}, defined in Subsection~\ref{subsection-tropical-modification}). Thus, the deformation procedure only has to handle paths where the combinatorial type is the same throughout the deformation path in~$\MTrop g$. Setting up the deformation requires analysing 16 cases and a {\em balancing relation} between certain determinants associated to
	the tropical morphisms. This balancing relation, closely
	related to the theory of morphisms between tropical fans
	\cite{gkm09}, ensures
	that we can always proceed further in our desired path from
	$\mG$ to $\bar \mG$. 
	
	In Part~II the initial point of the deformation is our own family of graphs, which we call \emph{caterpillars of loops}. We understand them better than the constructions of \cite{cd18} in that we are able to prove that if $\mG$ is a caterpillar of loops, then there are exactly Catalan-many tropical morphisms realizing the gonality bound. We also strengthen the deformation procedure. By analyzing 8 further cases we can deform past changes of combinatorial type. We show that with suitable multiplicity the count of tropical morphisms realizing the gonality bound remains constant thorough deformation. 
	
	To carry out this simple idea we introduce a technical framework to describe what is a full-dimensional cone and its properties, set up the deformation procedure, and support the case work required to prove the correctness of deformation. The road through Parts I and II is long, but it comes with two rewards: we achieve a combinatorial proof of the gonality bound; at the finish line we arrive at a tropical space that parametrizes tropical morphisms attaining the gonality bound, which surjects onto $\MTrop g$, is connected through codimension~1, and has fibres with Catalan-many points. This mirrors the beautiful classical story due Eisenbud and Harris~\cite{eh87}. It would be interesting to understand how
	our work relates to the tropicalization of the
	space of admissible covers from \cite{cmr16}, but this
	is beyond the scope of these combinatorial papers. 
	
	\subsection*{Organization of article}
	Section~\ref{sec-background} recalls several concepts on
	graphs, metric graphs, and $\MTrop g$.
	Section~\ref{sec-trop-glue} develops a formalism for
	tropical morphisms that highlights the combinatorics. Here we introduce full-dimensional cones of tropical morphisms.
	Section~\ref{sec-gluing-datum} introduces the main
	combinatorial gadget of this article, the gluing datum. We use it to study the properties of full-dimensional cones. Section~\ref{sec-examples} is a brief
	pause from developing the theory, with a fully worked out
	example suggesting the deformation procedure---in fact, we
	encourage the reader to go through this example before going
	through the more technical parts of the article! Section~\ref{sec-limits} resumes the theory with a study of the deformation operation, and a proof of the main result. This proof depends on many constructions that we carry out in Section~\ref{sec-constructions}.

	\section*{Acknowledgements} 
	We thank Erwan Brugall\'e, Melody Chan, Hannah Markwig, and Sam Payne for enlightening discussions, pointers to literature, and insights into future research directions. We thank the hospitality and the marvellous tropical environment at the Mittag-Leffler institute during the semester on ``Tropical Geometry, Amoebas and Polytopes'', where part of this research was done. We thank the general and specific remarks of three anonymous referees, which helped us improve greatly the exposition of the paper.
	
	\section{Background} \label{sec-background}
	
	\subsection{Graphs} \label{subsection-graphs}
	A graph $G$ is a pair $(V(G),E(G))$ of disjoint sets, the vertex and edge sets respectively, and a map $\iota_G$ defined on $E(G)$ encoding incidences,  whose values are multisets of two elements of~$V(G)$. 
	Given an edge $e$ with $\iota_G(e) = \{A, B\}$ we call $A, B$ the \emph{ends} of $e$, and say that $A$ and $e$ are \emph{incident}. A \emph{loop} is an edge $e$ such that its two ends are equal. A \emph{subgraph} of $G$ is a graph $G'$ such that $V(G') \subset V(G)$, $E(G') \subset E(G)$, and $\iota_{G'}$ equals the restriction $\iota_G|_{E(G')}$. A subset $S$ of $V(G)$ induces a subgraph by taking for edge set all the edges of $G$ with both ends in $S$. Likewise, a subset $S$ of $E(G)$ induces a subgraph by taking for vertex set the ends of all the edges in $S$. A graph $G$ is \emph{finite} if both $V(G)$ and $E(G)$ are finite sets; and \emph{loopless} if it has no loops.
	
	A \emph{path} of $G$ is a sequence $\langle A_0, e_1, A_1, \dots, e_\mu, A_\mu \rangle $ of alternately vertices and edges, where no element is repeated, such that consecutive elements are incident and $\mu \ge 1$. The \emph{ends} of a path are the vertices $A_0$, $A_\mu$. The remaining vertices $A_1, \dots, A_{\mu - 1}$ are \emph{interior}.  We call $\mu$ the \emph{length} of $P$. We write $x \in P$ to mean that $x$ is an element in the sequence~$P$. If $P$ is a path with ends $A$, $B$, and $e$ in $E(G) \setminus P$ is an edge that also has ends $A$, $B$, then we call the subgraph consisting of $e$ together with $P$ a \emph{cycle}. Given a vertex $A$, its \emph{connected component} is the subgraph induced by the set of ends of all paths with $A \in P$.  The vertex set $V(G)$ is partitioned into connected components. A graph is \emph{connected} if it has a single connected component.
	
	Let $G$ be a finite, connected graph. Set $g(G)=|E(G)|-|V(G)|+1$. Following a convention due to~\cite{bn07}, we call this number the {\em genus} of $G$. 
	If $G$ has genus 0, then we call $G$ a \emph{tree}.
	Write \emph{$x \in G$} for $x \in V(G) \cup E(G)$. Let $A$ be a vertex. Write $\NeighG G A$ for the subset of $E(G)$ incident to $A$. The \emph{valency} $\valeG G A$ of $A$ is the number of edges in $\NeighG G A$, with loops counting twice. The \emph{minimum valency} of $G$ is the number $\min_{A \in V(G)} \valeG G A$. We write $\Neigh A$ and $\vale A$ when $G$ is clear from the context. A vertex is \emph{monovalent}, \emph{divalent}, \emph{trivalent}, or \emph{$n$-valent} if its valency equals 1, 2, 3, or $n$, respectively. We call a monovalent vertex of a tree a \emph{leaf}. 
	
	Let $S \subset E(G)$. The \emph{deletion of $S$}, denoted by $G \setminus S$, is the graph with edge set $E(G) \setminus S$ and incidence map $\iota_G|_{E(G) \setminus S}$. The \emph{contraction of $S$}, denoted by $G / S$, is the graph $(V(G)/\sim, E(G) \setminus S)$, 
	where $\sim$ identifies the vertices $A$ and $B$ if they are the ends of some $e$ in $S$,  with incidence map $\pi \circ \iota_G|_{E(G) \setminus S}$ with $\pi : V(G) \to V(G)/ \sim$ the canonical projection.  
	The \emph{subdivision} of an edge $e$ yields the graph $(V(G) \cup \{C\}, (E(G) \setminus \{e\}) \cup \{e_1, e_2\})$, where the ends of $e_1$ are $A, C$, and of $e_2$ are $C, B$.
	
	Let $G$, $G'$ be two graphs. A \emph{graph morphism} is a map $\gamma: V(G) \cup E(G) \to V(G') \cup E(G')$ such that: $\gamma(V(G)) \subseteq V(G')$; if $\gamma(e) \in V(G')$ for an edge $e$ in $E(G)$ with ends $A$ and~$B$, then $\gamma(e) = \gamma(A) = \gamma(B)$; if $\gamma(e) \in E(G')$ for an edge $e$ in $E(G)$ with ends $A$ and~$B$, then the ends of $\gamma(e)$ are $\gamma(A)$ and $\gamma(B)$. In essence, a morphism is an incidence preserving map that can contract edges. If $\gamma(E(G)) \subseteq E(G')$, then we call $\gamma$ a \emph{homomorphism}. An \emph{isomorphism} is a bijective homomorphism; its inverse is then a homomorphism as well.
	
	An edge (resp. vertex) labelling is an injective map $\lambda$ from $E(G)$ (resp. $V(G)$) to a set~$S$. A total order $\le_S$ on $S$ induces a total order on $E(G)$ by letting $e \le_{E(G)} e'$ when~$\lambda(e) \le_S \lambda(e')$. We use $\mathbb N$ as labelling set, with its natural order. In the next section we consider elements from the vector space $\RR^{E(G)}$ (real valued functions on $E(G)$). By choosing an edge labelling $\lambda$ this space is identified with $\RR^{|E(G)|}$: we write $e_j \in E(G)$ for $\inv \lambda(j)$ (we use this notation whenever a $\lambda$ has been chosen); we take as ordered basis of $\mathbb R^{E(G)}$ the functions $\ell_i$ that map $e_j$ to $1$ if $i = j$, and zero otherwise.
	
	\subsection{Metric graphs} \label{section-metric-graphs} A \emph{length function} for $G$ is a map $\ell:E(G) \to \RRo$. The pair $(G, \ell)$ gives rise to a one-dimensional CW-complex $\mG$ by gluing the disjoint union of closed real intervals $\mG_e =[0,\ell(e)]$, one for each $e$ in $E(G)$, along their endpoints in the manner prescribed by~$G$.  
	We call the first Betti number of the CW-complex $\mG$ the \emph{genus} of~$\mG$, denoted by $g(\mG)$. It coincides with~$g(G)$.
	This CW-complex is equipped with the shortest-path metric in which the intervals $\mG_e$ have lengths $\ell(e)$. This metric space is called a {\em metric graph}. We call $G$ a \emph{model}, and the pair $(G, \ell)$ a \emph{realization}, of $\mG$.
	
	Let $\mG = (G, \ell_G)$. There are infinitely many realizations of~$\mG$. For example, we often consider the model $G'$ obtained by subdividing $e$ in $E(G)$ into $e_1, e_2$, and taking as length function some $\ell'$ with $\ell'(e_1) + \ell'(e_2) = \ell(e)$, so $(G', \ell')$ also realizes $\mG$. A \emph{vertex set} of $\mG$ is a finite subset $S \subset \mG$ of points such that $\mG \setminus S$ is homeomorphic to a disjoint union of open real intervals. Every realization $(G, \ell_G)$ of $\mG$ induces a vertex set: the set of endpoints of the intervals $\mG_{e}$. We refer to the intervals $\mG_e$ as the edges of $\mG$.  Every vertex set $S$ induces a realization $(G_S, \ell_S)$ of $\mG$: the vertices of $G_S$ are the points in $S$; the edges are the intervals $\mG \setminus S$; the length function $\ell_S$ is given by the lengths of $\mG \setminus S$. Thus, a choice of model $G$ for $\mG$ amounts to a choice of vertex set, and therefore edge set as well.
	
	
	
	
	We single out a particular realization for $\mG$. 
	Consider points $x$ in $\mG$ such that for all $\varepsilon > 0$ the open ball $B(x,\varepsilon)$ with centre $x$ and radius $\varepsilon$ is not isometric to the interval~$(-\varepsilon, \varepsilon)$. We call these \emph{essential vertices}. The naming is due to the fact that if $x$ is an essential vertex, then $x$ cannot be in the topological interior of an edge $\mG_e$. Thus, the set of essential vertices is contained in the set of endpoints of edges $\mG_e$, and so there are finitely many. On the other hand, we call the metric graphs with an empty set of essential vertices \emph{metric loops}. They have genus~1 and admit the graph with one vertex and one loop as a model.
	
	\begin{lm} \label{lemma-model-metric-graph}
		Let $\mG$ be a metric graph distinct from a metric loop, $\mathcal E$ the essential vertices of $\mG$, and $S \subset \mG$ a finite set of points. Then $S$ is a vertex set if and only if $\mathcal E \subset S$.
	\end{lm}
	
	The valency of a point $x$ in $\mG$ is the valency of $x$ in $G_S$ for $S$ a vertex set with $x$ in~$S$.
	If $\mG$ is not a metric loop, then we call the realization $(G_{\mathcal E}, \ell_{\mathcal E})$ the \emph{essential realization}. Thus, by Lemma \ref{lemma-model-metric-graph} the \emph{essential model} $G_{\mathcal E}$ is minimal in the sense that every other model arises from a sequence of edge subdivisions of $G_{\mathcal E}$. Hence, all the realizations are refinements of the essential realisation. This shows $\vale x$ to be well defined.

	\subsection{Tropical modification} \label{subsection-tropical-modification}
	In the literature, \emph{tropical modification} of a metric graph $\mG = (G, \ell)$ is an operation that deletes or attaches a segment ending in a monovalent point. For investigating tropical morphisms we consider metric graphs to be equivalent under a sequence of these operations. This allows to meet certain \emph{balancing conditions}, and does not alter $\mG$ in a significant manner. Our approach to tropical modification is based on the notion of dangling elements in the model $G$:
	
	\begin{de} \label{de-dangling-char}
		An edge $e$ of $G$ is \emph{dangling} if deleting $e$ produces two connected components and one of those is a tree. A vertex $A$ of $G$ is dangling if all $e$ in $E(A)$ are dangling.
	\end{de}
	
	This definition detects points contained in a segment that has been attached to $\mG$. Fix a model $G$ for $\mG$. Let $x$ be in $\mG$. If $x$ is in the vertex set induced by $G$, then we say that $x$ is dangling when its corresponding vertex in $G$ is dangling. Otherwise $x$ is in the interior of an edge $\mG_e$ and we say that $x$ is dangling if $e$ is dangling. It is straighforward to see that this definition is independent of the chosen realization $(G, \ell)$.
	
	Let $\tG$ be the graph obtained by deleting all the dangling elements of $G$. By Definition~\ref{de-dangling-char} we have that $g(G) = g(\tG)$. 
	Let $\mH = (\tG, \ell|_{\tG})$, where $\ell|_{\tG}$ is the restriction of $\ell$ to $E(\tG)$. We identify $G$ with $\tG$ and $\mG$ with $\mH$ under tropical modification. 
	Let $H$ be the essential model of $\mH$. The construction of $\mH$ and Lemma~\ref{lemma-model-metric-graph} together imply that:
	\begin{lm} \label{lemma-combinatorial-type}
		Let $\mG$ be a metric graph of genus at least~2, and $\mG'$ a tropical modification of $\mG$, such that $\mG'$ has a model $G'$ with minimum valency at least three. Then $\mG'$ is isometric to $\mH$ and $G'$ is isomorphic to $H$.  
	\end{lm}
	
	Thus, $\mH$ is a canonical representative of the equivalence class under isomorphism and tropical modification of $\mG$. We call a graph with minimum valency at least three a \emph{combinatorial type}. This is an important invariant under tropical modification, so we say that $H$ is the combinatorial type of $\mG$. 
	
	
	\subsection{Moduli space of metric graphs}  \label{section-moduli-space-of-metric-graphs}
	We give a barebones account of the moduli space of genus-$g$ metric graphs $\MTrop g$.  If all the vertices of a graph are trivalent, we call it a \emph{trivalent combinatorial type}, or simply \emph{trivalent}. We denote by $\mathbb G_g$ the family of genus-$g$ trivalent combinatorial types. For the construction of $\MTrop g$ we use the following straightforward lemma:
	
	\begin{lm} \label{lemma-trivalent}
		Let $g \ge 2$ and $H_0$ be a genus-$g$ combinatorial type. There exists $H$ in $\mathbb G_g$ such that there is a sequence of edge contractions of $H$ that yield a graph isomorphic to $H_0$.
	\end{lm}
	
	Given a graph $G$, we denote by $C_G$ the set of length functions defined on $G$. This set has a cone structure, as it equals the positive orthant $\RRo^{E(G)}$. We identify each point $\ell$ of $C_G$ with the metric graph~$(G, \ell)$ to obtain the (rational polyhedral) cone $C_G$ of metric graphs with model $G$.
	
	Motivated by Lemma~\ref{lemma-trivalent} we add some of the boundary points of $C_G$ with the following convention: 
	given a map $\ell: E(G) \to \RRgo$, contract all the edges of~$G$ for which $\ell(e) = 0$, to get $G'$;  
	the pair $(G, \ell)$ refers to the metric graph~$(G', \ell|_{G'})$. Let $\overline{C}_G \subset \RRgo^{E(G)}$ be the set of maps $\ell$ for which the metric graph $(G, \ell)$ has genus $g$, that is, all cycles of $G$ have positive length. gluing together these completed cones gives rise to a space where points bijectively correspond to classes under tropical modification of genus-$g$ metric graphs:
	
	\begin{de}The moduli space of genus-$g$ metric graphs is the space
		\[\MTrop g = \left(\bigsqcup_{\substack{H \in \mathbb G_g}} \overline{C}_H \right) \big / \cong, \]
		where $\cong$ identifies points $\ell_H$ and $\ell_{H'}$ for which $(H, \ell_H)$ and $(H', \ell_{H'})$ are isometric. It is given the quotient topology. 
	\end{de}
	
	Given a metric graph $\mG$, by Lemma~\ref{lemma-combinatorial-type} and \ref{lemma-trivalent} there is a point in $\MTrop g$ corresponding to a graph $\mH$ that is equivalent to $\mG$ under tropical modification. By construction this point is unique.
	
	\begin{re}
		The reader might find in the literature a different use for the symbol $\MTrop g$, albeit for constructions related to ours. They differ by imposing additional structure on the metric graph: marked points, legs of infinite length, weights on the vertices, and more. These are not required for our purposes, as they typically arise to encode more algebro-geometric information in a graph. See \cite{koz09} and \cite{acp15} for in-depth studies.
	\end{re}

	\begin{re}
		Identifying isometric points complicates the geometry of $\MTrop g$. We do not obtain a polyhedral cone complex, as one might expect. This issue is addressed in \cite{acp15} with the introduction of \emph{abstract polyhedral cone complexes}. See their Section~4.3 for a proof that $\MTrop g$ is an abstract polyhedral cone complex.
	\end{re}
	
	Let $H$ be a combinatorial type. From $2|E(H)| = \sum_{A \in V(H)} \vale A$ one can conclude that $H$ is trivalent if and only if $|E(H)| = 3g-3$. Hence, $\MTrop g$ is pure of dimension $3g-3$.
	
	
	\subsection{Notation} \label{sec-notation} 
	In the next four sections the following letters, and small variations thereof, are used to denote specific objects. Metric graphs: $\mG$ for metric graphs with $g(\mG) \ge 2$; $\mT$ for metric trees (due to tree being \textgreek{δέντρο} in Greek), 
	that is genus 0 metric graphs; and $\mH$~for the deletion of the dangling elements of $\mG$. Finite graphs: $G$ for a finite connected graph with $g(G) \ge 2$,  $A$ for its vertices and $e$ for its edges; $T$~for a finite tree, $v$ for its vertices and~$t$ for its edges; and $H$ for the combinatorial type of $G$. Morphisms: $\dtmor$ for maps between finite graphs; and $\tmor$ for maps between metric graphs. We write $[d]$ for the set  $\{1, 2, \dots, d \}$, and $x \in G$ for $x \in V(G) \cup E(G)$.

	
	
	\section{Tropical morphisms} \label{sec-trop-glue}

	\subsection{\DTmors} \label{sec-trop-morph}
	
	We begin with a map encapsulating the combinatorial information of tropical morphisms. These are a special class of graph morphisms, augmented with an index map. 
	
	\begin{de}[discrete tropical morphism] \label{def-discrete-tropical-morphism}
		Let $G$, $G'$ be loopless graphs, $\dtmor \!:\! G \to G'$ a morphism, and $m_\dtmor: G \to \ZZgo$ a map.
		\begin{enumerate}[(a)] 
			\item $m_\dtmor$ is an \emph{index map} for $\dtmor$ if for every $e$ in~$E(G)$ we have that $\dtmor(e) \in V(G')$ if and only if $m_\dtmor(e) = 0$. That is, only index zero edges are contracted by $\dtmor$.
			\item $\dtmor$ is \emph{harmonic} with index map $m_\dtmor$ if it satisfies the \emph{balancing condition}: for every $A$ in $V(G)$ and $e'$ incident to $\dtmor(A)$ we have that
			\[ m_\dtmor(A) = \sum_{e \in E(A) \cap \inv \varphi(e')} m_\varphi(e).\]
			Note that this makes the sum independent of the choice of $e'$.
			\item $\dtmor$ is \emph{non-degenerate} if $m_\dtmor(x) \ge 1$ for all $x$ in $G$.  
			\item $\varphi$ satisfies the \emph{Riemann-Hurwitz condition} (RH-condition) if for every $A$ in $V(G)$:
			\[r_\dtmor(A) = (\vale A - 2) - (\vale \varphi(A) - 2) \cdot m_\varphi (A) \ge 0. \]
			\item A \emph{discrete tropical morphism} is a pair $(\dtmor, m_\dtmor)$ consisting of a non-degenerate harmonic morphism $\varphi$ with index map $m_\dtmor$ that satisfies the RH-condition. We write \DTmor to shorten. 
		\end{enumerate}
	\end{de}
	
	\begin{re}The non-degeneracy condition ensures that no edge is collapsed. \end{re}

	\begin{re} \label{re-difference-caporaso}
		Definition~\ref{def-discrete-tropical-morphism} is similar to Definition 2.1 of \cite{cap14}, but there are differences:  condition (c) of Definition 2.1 in the cited work is required only for vertices of $G$ (see discussion in Section 1.3 of \cite{cd18}); an alternative formula for $r_\dtmor$ is given on  Equation~(10) of the cited work (see Lemma~\ref{lemma-formula-rphi}); and we do not have a vertex labelling representing extra~genus. 
	\end{re} 
	
	\begin{lm}[formula for $r_\varphi$] \label{lemma-formula-rphi} 
		Let $\dtmor$ be a \DTmor with index map $m_\dtmor$. Then
		\[r_\dtmor(A) = 2(m_\dtmor(A) - 1) - \sum_{e \in E(A)} (m_\dtmor(e) - 1).\]
	\end{lm}
	
	\begin{proof}
		By the balancing condition $m_\dtmor(A) \cdot \vale \dtmor(A) = \sum_{e \in \Neigh A} m_\dtmor(e)$. Thus,
		\begin{align*} 
		r_\varphi(A) &= \vale A - 2 - m_\varphi (A) \cdot (\vale \dtmor(A) - 2)\\
		&= 2(m_\dtmor(A) - 1) + \vale A - \sum_{e \in \Neigh A} m_\dtmor(e)\\
		&= 2(m_\dtmor(A) - 1) - \sum_{e \in \Neigh A} (m_\dtmor(e) - 1). \qedhere
		\end{align*}
		
	\end{proof}
	
	We always denote the index map of $\dtmor$ by $m_\dtmor$. So we write ``\emph{$\dtmor$ is a \DTmorp}'' instead of ``\emph{$\varphi$ is a \DTmor with index map $m_\dtmor$}''. We call $G$ the \emph{source}, and $G'$ the \emph{target}, of $\dtmor$. The index $m_\dtmor$ can be understood as a local degree: for $A$ in $V(G)$ restricting $\dtmor$ to the subgraph $G_A$ induced by $E(A)$ gives a \DTmor onto the subgraph induced by $\dtmor(E(A))$; any fibre of $\dtmor |_ {G_A}$ has $m_\dtmor(A)$ elements, each element counted with multiplicity~$m_\dtmor$. Since $G$ is connected this local degree gives rise to a global degree, where the constant obtained in the count of preimages with multiplicities is called the \emph{degree} of~$\dtmor$: 
	\[\deg \dtmor = \sum_{\dtmor(A) = A'} m_\dtmor(A);\]
	this is independent of $A'$ by the connectedness of $G$ (see proof of Lemma 2.4 of \cite{bn09}).
	
	\begin{ex} \label{example-discrete-tropical-morphism}	
		Figure~\ref{figure-1} illustrates two \DTmors $\dtmor_A$, $\dtmor_B$ of degrees 3 and~4, respectively. The targets are drawn below, the sources above, and the fibres with dotted lines. The edges have index~1, unless noted otherwise with a number above the edge.
		\vspace{1em}
		
		\noindent	
		\begin{minipage}{0.95\textwidth}
			\begin{minipage}[t]{0.6\textwidth}
				\centering
				\begin{overpic}[scale=1.1]{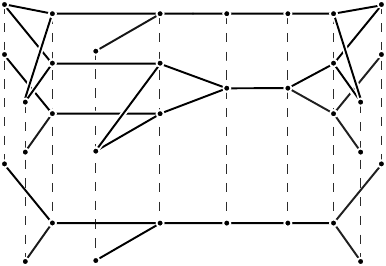}    
					\put (66,47) {\scalebox{0.6}{$2$}}
				\end{overpic} 
			\end{minipage}
			\begin{minipage}[b]{.35\textwidth} 
				
				\begin{overpic}[scale=1.1]{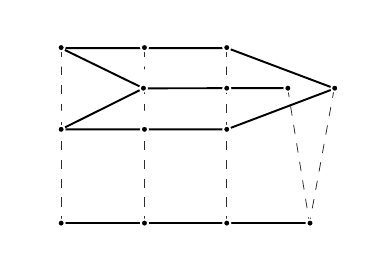}    
					\put (46,47) {\scalebox{0.6}{$2$}}
					\put (65,47) {\scalebox{0.6}{$2$}}
				\end{overpic} 
			\end{minipage}
			\\
			\noindent 
			\begin{minipage}[t]{.6\textwidth} 
				\hspace{11em}
				$\dtmor_A$			
			\end{minipage}\hspace{1em}
			\begin{minipage}[t]{.35\textwidth} 
				\hspace{5.5em}
				$\dtmor_B$
			\end{minipage}
			
			\captionof{figure}{} \label{figure-1}
		\end{minipage}
	\end{ex}

	This setup is parallel to the classical theory of maps between Riemann surfaces. There is a Riemann-Hurwitz formula relating the genera of the source and the target.
	
	\begin{lm}[Riemann-Hurwitz formula] \label{lemma-riemann-hurwitz} 
		Let $\dtmor \!:\! G \to G'$ be a \DTmorp. Then
		\[
		2g(G) - 2 = \deg \dtmor \cdot (2g(G')-2) + \sum_{A \in V(G)} r_\dtmor(A).
		\] 
	\end{lm}
	
	\begin{proof}
		Recall that in a graph the sum of all the valencies is twice the number of edges:
		\begin{align*} 
		\sum_{A \in V(G)} r_\dtmor(A) &= \sum_{A \in V(G)} (\vale A - 2 - m_\varphi (A) \cdot (\vale \dtmor(A) - 2) )\\
		&= 2|E(G)|-2|V(G)|- \deg \dtmor \cdot (2|E(G')|-2|V(G')|)\\
		&= 2g(G) - 2 - \deg \dtmor \cdot (2\cdot g(G')-2). 
		\end{align*}
		Rearrange terms to obtain the result.
	\end{proof}
	The RH-condition implies that the sum $\sum_{A \in V(G)} r_\dtmor(A)$ is positive, thus $g(G) \ge g(G')$ by Lemma~\ref{lemma-riemann-hurwitz}.  As a side remark, consider Equation~(15) of \cite{cap14}. By taking degrees of the divisors we recover Lemma~\ref{lemma-riemann-hurwitz}. Thus, one of the motivations for the RH-condition is a necessary and sufficient condition to have the ramification divisor be effective, as in the classical setting.
	
	\subsection{Tropical morphisms} \label{subsection-tropical-morphisms} By choosing a length function for the target $G'$ we obtain a \emph{tropical morphism} of metric graphs, in the sense of \cite{cd18}.
	
	\begin{de}[tropical morphism] \label{definition-tropical-morphism}  Let $\dtmor  \!:\! G \to G'$ be a \DTmor and $\ell_{G'}$ in~$C_{G'}$. The corresponding \emph{tropical morphism} is the unique map $\tmor$ from the source metric graph $\mG = (G, \,\, e \mapsto \ell_{G'} (\varphi(e))\,/\, m_\varphi(e) \, )$ to the target $\mG' = (G', \ell_{G'})$, whose restriction to $\mG_e$ is a bijective linear function to $\mG'_{\dtmor(e)}$, and that agrees with $\dtmor$ on vertex points.
	\end{de}
	
	Call $\dtmor$ and $(\dtmor, \ell_{G'})$ a \emph{model} and a \emph{realization} of $\tmor$, respectively. From the connectedness of $G$ and $G'$ we get that $\tmor$ is continuous. The length function for the source is constructed so the linear map $\tmor|_{\mG_e}$ has integral slope $m_\dtmor(e)$. For a point $x$ in the interior of an edge the slope is the multiplicity, and if $x$ is a vertex then the multiplicity is computed with the balancing condition. With this, the count with multiplicities of points in the fibre $\inv \tmor (x')$ is independent of  $x'$ in $\mG'$, and moreover agrees with the degree of any model of $\tmor$. 
	
	Definition \ref{definition-tropical-morphism} differs from \cite{cd18} because here our focus is on combinatorial constructions. In the previous paragraph we have argued that the map $\tmor = (\dtmor, \ell_{G'})$ is a tropical morphism in the sense of \cite{cd18}. Conversely, given a tropical morphism $\tmor \!:\! \mG \to \mG'$, we choose vertex sets $\mathcal G$, $\mathcal T$ on $\mG, \mG'$ such that $\tmor$ induces a \DTmor $\dtmor_{\text{ess}}$. 
	
	\begin{cons}[essential model] \label{construction-canonical-model}
		Let $\tmor \!:\! \mG \to \mG' $ be a tropical morphism, and $\mathcal E$, $\mathcal E'$ the essential vertices of $\mG$, $\mG'$ respectively. Set $\mathcal T = \mathcal E' \cup \tmor(\mathcal E)$ and $\mathcal G =  \inv \tmor (\mathcal T)$. By Lemma~\ref{lemma-model-metric-graph} these are vertex sets, so let $(G, \ell), (G', \ell')$ be the induced realizations of $\mG$, $\mG'$. Define a map 
		$\dtmor_{\text {ess	}} : G \to G'$ by sending a vertex of $\mG$ to the corresponding vertex in $\mG'$, 
		and for each edge $e$ in $E(G)$ choose a point $x$ in $\mG_e$ and map $e \mapsto e'$, where $e'$ is such that $\tmor(x)$ is in~$\mG'_{e'}$. As index map take $m_{\dtmor_{\text {ess}}} = \ell' \circ \dtmor(e)/\ell(e)$.
	\end{cons}
	
	\begin{lm} \label{lemma-dtmorph-minimal-model}
		The pair $(\dtmor_{\text {ess}}, \ell')$ realizes $\tmor$, and is minimal in the sense that for any other model of $\tmor$, the induced vertex sets on $\mG, \mG'$ arise as edge subdivisions of $\mathcal G$, $\mathcal T$.
	\end{lm}
	
	\begin{proof}
		We claim that $\dtmor_{\text {ess}}$ is well defined and is a \DTmorp. Let $x$ in $\mG_e$ and $e'$ in $E(G')$ such that $\tmor(x)$ is in $\mG'_{e'}$. We show that $e'$ only depends on $e$. To see this, take $x_1$, $x_2$ in $\mG_e$. There is a topological path $P$ (the image of a continuous map $\omega: [0,1] \to \mG$) from $x_1$ to $x_2$ contained in the interior of $\mG_e$, so in particular it does not pass through any vertex of~$\mathcal G$. Thus, $\tmor(P)$ does not pass through any vertex $x'$ of $\mathcal T$ (otherwise $P$ would pass through some vertex in $\inv \tmor(x') \subset \mathcal G$). Hence, $\tmor(x_1)$ and $\tmor(x_2)$ are in the same $\mG'_{e'}$. The index map $m_{\dtmor_{\text {ess}}}$ records the slope of $\tmor|_{\mG_{e}}$, thus the balancing condition, non-degeneracy, and the RH-condition are satisfied. It is then clear that $(\dtmor_{\text {ess}}, \ell')$ realizes $\tmor$.
		
		Regarding minimality, as \DTmors are graph morphisms we get that $\dtmor(V(G)) \subset V(G')$. Thus, in every model $\dtmor$ of $\tmor$ the vertex set of $\mG'$ has to contain~$\tmor(\mathcal E)$ (and $\mathcal E'$). Furthermore, the non-degeneracy condition of $\dtmor$ implies that the fibre of a vertex of $G'$ is a subset of $V(G)$. So a vertex set of $\mG$ has to contain~$\inv \tmor(\mathcal T)$. 
	\end{proof}

	\subsection{Isomorphisms, tropical modification, combinatorial type}  \label{subsection-isomorphisms-tropicalmodification-combinatorialtype} Recall that in the space $\MTrop g$ we consider metric graphs in an equivalence class under isometry and tropical modification. This setup extends naturally to tropical morphisms.

	\begin{de}[isomorphism of graph morphisms] \label{definition-tropical-iso} 
		
		Let $\dtmor_1: G_1 \to G'_1$ and $\dtmor_2: G_2 \to G'_2$ be graph morphisms. An isomorphism from $\dtmor_1$ to $\dtmor_2$ is given by graph isomorphisms $(\gamma, \gamma')$ for which Diagram~\ref{diagram-iso}.a commutes. If additionally $\dtmor_1$ and $\dtmor_2$ are \DTmorsp, then $(\gamma,\gamma')$ is an isomorphism of \DTmors if the index map $m_{\dtmor_2}$ pulls back to~$m_{\dtmor_1}$, namely Diagram~\ref{diagram-iso}.b commutes. Let $\tmor_1$ and $\tmor_2$ be tropical morphisms. An isomorphism from $\tmor_1$ to $\tmor_2$ is given by isometries $\Psi$ and $\Psi'$ for which Diagram~\ref{diagram-iso}.c commutes.

		\begin{minipage}[b]{0.9\textwidth}
			\centering
			\begin{minipage}{0.3\textwidth}
				\centering
				\[
				\begin{tikzcd}
				G_1 \arrow[r,"\dtmor_1"] \arrow[d,swap,"\gamma"] &
				G'_1 \arrow[d,"\gamma'"] \\
				G_2 \arrow[r,"\dtmor_2"] & G_2'
				\end{tikzcd}	
				\]
				
				(a)
			\end{minipage}
			\begin{minipage}{0.3\textwidth}
				\centering
				\[ 
				\begin{tikzcd}
				G_1 \arrow[d,swap,"\gamma"] \arrow{dr}{m_{\dtmor_1}} \\
				G_2 \arrow[r,swap,"m_{\dtmor_2}"] & \mathbb Z_{\ge 0}
				\end{tikzcd} 
				\]
				
				(b)
			\end{minipage}
			\begin{minipage}{0.3\textwidth}
				\centering
				\[ 
				\begin{tikzcd}
				\mG_1 \arrow[r,"\tmor_1"] \arrow[d,swap,"\Psi"] &
				\mG'_1 \arrow[d,"\Psi'"] \\
				\mG_2 \arrow[r,"\tmor_2"] & \mG_2'
				\end{tikzcd} 
				\]
				
				(c)
			\end{minipage}
			
			\newcounter{diagram}
			\refstepcounter{thm}
			\label{diagram-iso}
			Diagram \ref{diagram-iso}
		\end{minipage}
	\end{de}

	We now define equivalence under tropical modification for a \DTmorp{} $\dtmor \!:\! G \to G'$. We say that the fibre $\inv \dtmor(x')$, for $x'$ in $G'$, is \emph{dangling} if all its elements are dangling in~$G$. Let $D \subset E(G')$ be the set of edges with dangling fibres. Consider the subgraph $\hat G'$  of $G'$ induced by $E(G') \setminus D$. The preimage $\inv \dtmor(\hat G')$ is not necessarily a connected graph (see Example~\ref{example-combinatorial-type-of-DTmors}), but the definition of dangling implies that there is only one connected component with non-zero genus. Call it $\hat G$. Restricting $\dtmor$ to $\hat G$ gives a \DTmor $\hat \dtmor: \hat G \to \hat G'$ that we call the \emph{deletion of dangling fibres} of $\dtmor$. See the Section \emph{Modifications} in \cite{cd18} for some examples. Tropical modification identifies $\dtmor$ with $\hat \dtmor$.
	
	The same definitions and constructions can be carried out on the metric side for a tropical morphism $\tmor = (\dtmor, \ell)$. A fibre $\inv \tmor(x)$ is dangling if all its points are dangling, and the deletion of dangling fibres $\hat \tmor$ is the restriction of $\tmor$ to the connected component with non-zero genus which remains after deleting all points belonging to dangling fibres. One can check that $(\hat \dtmor, \ell|_{\hat G'})$ is a realization of $\hat \tmor$. Tropical modification identifies $\tmor$ with $\hat \tmor$. It is straightforward to see that if $\tmor_1$ and $\tmor_2$ are equivalent under tropical modification, then $\hat \tmor_1$ and $\hat \tmor_2$ are isomorphic, thus $\hat \tmor$ is a canonical representative for the equivalence class of $\tmor$ under isomorphism and tropical modification. Thus, it is natural to use a model of $\hat \tmor$ as the \emph{combinatorial type} of the class of $\tmor$:
	
	\begin{de} \label{definition-combinatorial-type-tmor}
		Let $\tmor$ be a tropical morphism and $\hat \tmor$ its deletion of dangling fibres. The \emph{combinatorial type} of $\tmor$ is the model $\hat \dtmor_{\text {ess}}$ of $\hat \tmor$ constructed using Construction~\ref{construction-canonical-model}.
	\end{de}
	
	Alternatively, one can use Construction~\ref{construction-canonical-model} on $\tmor$ to get $(\dtmor_{\text {ess} }, \ell_{\text {ess} })$ and delete all dangling fibres to get $\widehat{\dtmor_{\text {ess} }}$. It is straightforward to see that $\widehat{\dtmor_{\text {ess} }}$ is isomorphic to $\hat \dtmor_{\text {ess}}$. So $(\hat \dtmor_{\text {ess} }, \ell_{\text {ess} }|_{\hat G_{ess}'})$ realizes $\hat \tmor$. We call the values of $\ell_{\text {ess} }|_{\hat G_{ess}'}$ the \emph{parameters} of $\tmor$. Thus, we say that $\tmor$ moves in the cone $C_\dtmor$ with dimension $\dim C_\dtmor$ and depends on $|E(\hat G_{ess}')|$ many parameters.

	\subsection{Cone of sources} 
	We associate to a \DTmor $\dtmor$ the set $C_\dtmor$ of sources (modulo tropical modification) obtained by varying the parameters. By Section~\ref{subsection-isomorphisms-tropicalmodification-combinatorialtype} this is an invariant under isomorphism and tropical modification.  We say that $\dtmor$ \emph{realizes} $\mG$ if $\mH \in C_\dtmor$ (recall that $\mH$ is obtained by deletion of dangling points). We introduce an important linear map to see that $C_\dtmor$ has the structure of a polyhedral cone. Another important invariant under tropical modification of $\dtmor$  is the combinatorial type of its source, which we denote by~$H(\dtmor)$.

	Let $\dtmor \!:\! G \to G'$ be a \DTmorp, $H = H(\dtmor)$, and $g = g(H)$. 
	Consider a length function $\ell_{G'}$ in~$C_{G'}$; the tropical morphism $\tmor = (\dtmor,\, \ell_{G'})$; and $\mH = (H, \ell_H)$ where $(H, \ell_H)$ is the deletion of dangling elements of the source of $\tmor$. Extend the map sending $\ell_{G'}$ to $\ell_H$ (defined from $\RRo^{E(G')}$ to $\RRo^{E(H)}$) to a linear map $A_\dtmor$ from 
	$\RR^{E(G')}$ to $\RR^{E(H)}$. We call $A_\dtmor$ the \emph{edge-length map}. Recall from Section~\ref{subsection-graphs} that by choosing labellings $\lambda_{H}$ and $\lambda_{G'}$ we identify $\RR^{E(H)}$ and $\RR^{E(G')}$ with $\RR^{|E(H)|}$ and $\RR^{|E(G')|}$. This gives an $|E(H)| \times |E(G')|$-matrix for $A_\dtmor$.
	
	We define the \emph{cone of sources} $C_\dtmor$ to be $A_\dtmor(\RRo^{E(G')})$. It
	parametrizes, up to tropical modification, the sources of tropical morphisms with model~$\dtmor$. Let $\tmor$ be a tropical morphism with model $\dtmor$. We say that $\tmor$ \emph{moves} in the cone $C_\dtmor$. 
	
	Note that $\dim C_\dtmor = \rk A_\dtmor \le 3g - 3$. We are interested in tropical morphisms $\tmor$ that move in a cone whose dimension equals the number of parameters defining $\tmor$. Namely, those for which $\rk A_\dtmor = |E(G')|$. We say that these cones, and \DTmorsp, have \emph{full-rank}.
	If $\dtmor$ has full-rank then there is a unique preimage $\inv A_\dtmor(\ell_H)$ for $\ell_H$. In other words, a unique tropical morphism $(\dtmor, \inv A_\dtmor(\ell_H))$ whose source retracts to $\mH = (H, \ell_H)$. 
	
	\subsection{Combinatorial types} \label{subsection-combinatorial-type-of-dtmors} Let $\dtmor \!:\! G \to G'$ be a \DTmorp, and $\tmor$ a tropical morphism with model $\dtmor$. Recall that the combinatorial type of $\tmor$ is the essential model $\hat \dtmor_{\text{ess}}$ of the deletion of dangling points $\hat \tmor$. Now we give an analogue of Lemma~\ref{lemma-combinatorial-type} for tropical morphisms; namely we give combinatorial conditions that characterize the \DTmors that show up as combinatorial types. We also show that \DTmors with full-rank satisfy these conditions.
	
	\begin{de} \label{definition-combinatorial-type}
		A \emph{combinatorial type of \DTmorsp} is a \DTmorp{} $\dtmor: G \to G'$ without dangling fibres and such that $\sum_{A \in \inv \dtmor(v)} r_\dtmor(A) \ge 1$ for every divalent $v$ in $V(G')$.
	\end{de}
	
	\begin{ex} \label{example-combinatorial-type-of-DTmors} Figure~\ref{figure-2} illustrates two \DTmors $\dtmor_A$, $\dtmor_B$. They are equivalent under tropical modification, and $\dtmor_B$ is a combinatorial type of \DTmorsp. Note that the degree is not invariant.

		\vspace{1em}
		
		\noindent	
		\begin{minipage}{0.95\textwidth}
			\begin{minipage}{.45\textwidth}
				\flushright
				\begin{overpic}[scale=1.3]{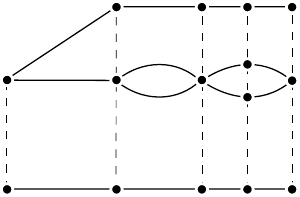}    
					
				\end{overpic} 
			\end{minipage}\hspace{1em}
			\begin{minipage}{.45\textwidth} 
				\centering
				\begin{overpic}[scale=1.3]{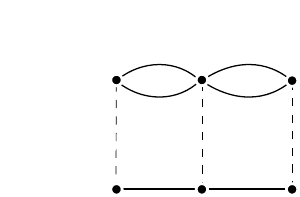}    
				\end{overpic} 
			\end{minipage} 
			\\ ~ \\
			\noindent 
			\begin{minipage}[t]{.45\textwidth} 
				\hspace{11.5em}
				$\dtmor_A$			
			\end{minipage}\hspace{1em}
			\begin{minipage}[t]{.45\textwidth} 
				\hspace{10em}
				$\dtmor_B$
			\end{minipage}
			
			\captionof{figure}{} \label{figure-2}
		\end{minipage}
	\end{ex}

	\begin{lm} \label{lemma-combinatorial-type-tmor}
		Let $\tmor$ be a tropical morphism, and $\bar \tmor: \bar \mG \to \bar \mG '$ be a tropical modification of $\tmor$, such that $\bar \tmor$ has a model $\bar \dtmor: \bar G \to \bar G'$ that is a combinatorial type of \DTmorsp. Then $\bar \tmor$ is isomorphic to $\hat \tmor$ and $\bar \dtmor$ is isomorphic to $\hat \dtmor_{\text{ess}}$.  
	\end{lm}
	
	\begin{proof}
		By definition $\bar \dtmor$ has no dangling fibres. So by Lemma~\ref{lemma-dtmorph-minimal-model} we only need to show that the points of $\bar \mG'$ corresponding to vertices of $\bar G'$ are in the set $\mathcal T$, with $\mathcal T$ as in Construction~\ref{construction-canonical-model} (the union of the set of essential vertices of $\bar \mG'$ with the image of the set of essential vertices of $\bar \mG$). This means proving that for every divalent $v$ in $V(\bar G')$ there is a vertex $A$ in $\inv {\bar \dtmor}(v)$ such that $\vale A \ge 3$. 
		Note that for $A$ in $\inv {\bar \dtmor}(v)$ we have $r_{\bar \dtmor}(A) = \vale A - 2 - m_{\bar \dtmor}(A)(\vale v - 2) = \vale A - 2$. By the balancing condition there are no monovalent vertices in $\inv {\bar \dtmor}(v)$.
		Thus, $\sum_{A \in \inv {\bar \dtmor}(v)} r_{\bar \dtmor}(A) = \sum_{A \in \inv {\bar \dtmor}(v)} \vale A - 2 \ge 1$ if and only if there is at least one vertex $A$ in $\inv {\bar \dtmor}(v)$ with $\vale A \ge 3$.
	\end{proof}

	Combinatorial types of \DTmors are canonical representatives in the equivalence class under tropical modification: the discussion after Definition~\ref{definition-combinatorial-type-tmor} tells us how to construct them, and Lemma~\ref{lemma-combinatorial-type-tmor} ensures their uniqueness.
	
	We call the sum $\sum_{A \in \inv \dtmor(v)} r_\dtmor(A)$ the \emph{change} at $v$, and denote it by $\ch v$. We have seen in the proof of Lemma~\ref{lemma-combinatorial-type-tmor} that divalent $v$ with $\ch v = 0$ arise as edge-subdivisions of the essential model. Thus, for divalent $v$ if $t, t'$ are the edges incident to $v$, then morally the fibres of $t$, $t'$ are equal (each $A$ in $\inv \dtmor(v)$ has valency 2, and the two edges have the same index $m_\dtmor$). This implies the following result, which is our motivation for calling this quantity ``change''.
	
	\begin{lm} \label{lemma-equal-columns}
		Let $v$ in $V(G')$ be divalent, and $\Neigh A = \{t,t'\}$. If $\ch v = 0$, then the columns of $A_\dtmor$ corresponding to $t$, $t'$ are equal.
	\end{lm}
	
	\begin{lm}
		If $\dtmor : G \to G' $ has full-rank, then $\dtmor$ is a combinatorial type of \DTmorsp.
	\end{lm}
	
	\begin{proof}
		Let $H$ be the combinatorial type of $G$, and $A_\dtmor: C_{G'} \to C_H$ be the edge-length map. Suppose $t$ in $E(G')$ has a dangling fibre. Then varying $\ell_{G'}(t)$ does not affect $A_\dtmor(\ell_{G'})$, contradicting that $A_\dtmor$ has full-rank. So $\dtmor$ has no dangling fibres. Moreover, by Lemma~\ref{lemma-equal-columns} we have that $\dtmor$ has no divalent vertices with $\ch v = 0$.
	\end{proof}
	
	Finally, we prove a combinatorial property that puts the condition $\ch v \ge 1$, for divalent vertices, in a wider context.
	
	\begin{lm} \label{lemma-combi-type-greater-than3}
		If $\dtmor$ is a combinatorial type, then $\ch v + \vale v \ge 3$ for all $v$ in $V(G')$.
	\end{lm}
	
	\begin{proof}
		We only need to show that $\ch v \ge 2$ for monovalent vertices $v$ of $G'$. Let $v$ be monovalent. Note that a monovalent vertex of $G$ is dangling. Since $G$ does not have dangling fibres, pick $A$ in $\inv \dtmor (v)$ with $\vale A \ge 2$. For such $A$ we have that $r_\dtmor(A) = \vale A - 2 - m_\dtmor(A) (\vale v - 2) = \vale A - 2 + m_\dtmor(A)$. Since $v$ is monovalent we have that $m_\dtmor(A) \ge \vale A$, thus $r_\dtmor(A) \ge 2$ and by the RH-condition this implies that $\ch v \ge 2$.
	\end{proof}

	\subsection{Dimension Formula}  \label{subsection-dimension-formula} Let $\dtmor: G \to G'$ be a \DTmorp. We prove a formula relating the number $|E(G')|$, the degree $d = \deg \dtmor$, and the genera of the source and the target of $\dtmor$. Note that~$|E(G')|$ is an upper bound for $\dim C_\dtmor$. Hence, we call this result \emph{the dimension formula}.
	
	\begin{lm}[total change] \label{lemma-total-change} 
		Let $\dtmor \!:\! G \to G'$ be a \DTmorp. Then
		\[
		\sum_{v \in V(G')} \ch v = 2g(G) - 2g(G') \cdot d + 2d - 2.
		\] 
	\end{lm}
	
	\begin{proof}
		This is a consequence of Lemma~\ref{lemma-riemann-hurwitz} (Riemann-Hurwitz formula).
	\end{proof}
	
	\begin{prop}[dimension formula] \label{prop-dim-formula} 
		Let $\dtmor \!:\! G \to G'$ be a \DTmorp. Then
		\[|E(G')| + \sum_{v\in V(G')} (\ch v + \vale v - 3) = 2g(G) - g(G')\cdot(2d-3) + 2d - 5. \] 
	\end{prop}
	\begin{proof}
		The ingredients are the fact that the sum of valencies gives twice the number of edges, and Lemma~\ref{lemma-total-change} (total change). We compute:
		\begin{align*}
		|E(G')| + \sum_{v\in V(G')} (\ch v + \vale v - 3) &= 3|E(G')| - 3|V(G')| + \sum_{v\in V(G')} \ch v\\
		&= 2g(G) - g(G') \cdot (2d-3) + 2d - 5. \qedhere
		\end{align*}
		
	\end{proof}
	
	We regard the second term on the left hand side in Proposition~\ref{prop-dim-formula} as a \emph{correction term}. Lemma~\ref{lemma-combi-type-greater-than3} gives us a choice in the equivalence class under tropical modification of $\dtmor$ for which the correction term is non-negative. With this we get an upper bound:
	
	\begin{cor} \label{corollary-upperbound-dim-cone}
		Let $\dtmor \!:\! G \to G'$ be a \DTmorp. Then \[\dim C_\dtmor \le 2g(G) - g(G')\cdot(2d-3) + 2d - 5.\]
	\end{cor}
	
	\begin{proof}
		Let $\hat \dtmor \!:\! \hat G \to \hat G'$ be a combinatorial type of \DTmors that is equivalent under tropical modification to~$\dtmor$. By Lemma~\ref{lemma-combi-type-greater-than3} and the dimension formula we get
		\[\dim C_{\hat \dtmor} \le |E(\hat G')| \le 2g(\hat G) - g(\hat G')\cdot(2d-3) + 2d - 5.\]   
		By the properties of tropical modification we have that $\dim C_{\hat \dtmor} = \dim C_{\dtmor}$, $g(\hat G) = g(G)$, and $g(\hat G') = g(G')$, proving the desired bound. 
	\end{proof}
	
	\subsection{Tree gonality} \label{subsection-gonality} To close this section we use the bound given by the dimension formula to reprove a result on tree gonality that appeared in \cite{cd18}. 
	We go on to outline what is the central problem for the rest of this article, and our strategy to tackle it.
	
	\begin{de}
		The \emph{tree gonality} of a metric graph $\mG$ is the minimum degree of a tropical morphism from some tropical modification of $\mG$ to some metric tree. 
	\end{de}
	
	Let $\dtmor \!:\! G \to T$ be a \DTmor to a tree. Every metric graph $\mG$ in $C_\dtmor$ has tree gonality at most $\deg \dtmor$. From now on we only work with \DTmors $\dtmor : G \to T$ to trees. For such maps Corollary~\ref{corollary-upperbound-dim-cone} gives that $\dim C_\dtmor \le 2g + 2d - 5$. Thus:
	
	\begin{cor} \label{corollary-gonality-bound}
		Let $g$, $d$ be positive integers. The locus of metric graphs in $\MTrop g$ with tree gonality at most $d$ has dimension at most $\min(2g + 2d - 5, 3g-3)$. 
	\end{cor}
	
	\begin{proof}
		Let $\mH$ be in $\MTrop g$, $\mG$ a tropical modification of $\mH$, and $\tmor : \mG \to \mT$ a tropical morphism to a tree with $\deg \tmor \le d$. Then $\mH \in C_\dtmor$, where $\dtmor : G \to T$ is the combinatorial type of $\tmor$. By the preceding discussion and since $\dim \MTrop g = 3g-3$ we get $\dim C_\dtmor \le \min(2g + 2d - 5, 3g-3)$. Note also that $|E(T)| \le \min(2g + 2d - 5, 3g-3)$, and that $|E(G)| \le  \deg \dtmor \cdot |E(T)| \le d \cdot |E(T)|$, so there are finitely many possibilities for $G$, $T$, and consequently also for $\dtmor$. This gives the result.
	\end{proof}
	
	Since $3g-3 \le 2g + 2d - 5$ implies that $\lceil g/2+1 \rceil \le d$, we get that most metric graphs in $\MTrop g$ have tree gonality at least $\lceil g/2+1 \rceil$. This is one of the two inequalities that comprises Theorem~1 in \cite{cd18}. The other inequality, namely that metric graphs have gonality at most $\lceil g/2+1 \rceil$, was shown constructively for certain families of graphs, one family for each combinatorial type:
	
	\begin{thm} \label{theorem-cd}
		Let $H$ be a trivalent combinatorial type and $g = g(H)$. Then there exists a \DTmor $\dtmor$ such that $\deg \dtmor =  \lceil g/2+1 \rceil$, $H(\dtmor) = H$ and $\dim C_\dtmor = 3g - 3$.
	\end{thm}		
	
	Theorem~\ref{theorem-cd} is proven constructively in Section~4 of \cite{cd18}. When $g$ is even, these families have the following property: 
	
	\begin{de} \label{def-full-dim} A \DTmor $\dtmor$, with  $g = g(H(\dtmor))$, is \emph{full dimensional} if \[\dim C_\dtmor = 3g - 3 = 2g + 2d - 5.\]
	\end{de}
	
	The pairwise equalities in the definition of full-dimensional have the following consequences:
	\begin{itemize}
		\item $\dim C_\dtmor = 3g-3$ implies that $H(\dtmor)$ is trivalent and the tropical morphism moves in a space of the right dimension, a necessary condition to hope to realize all sources in $C_{H(\dtmor)}$.
		\item $\dim C_\dtmor = 2g + 2d - 5$, $\dim C_\dtmor = 3g-3$, and the dimension formula together imply that $|E(T)| = 3g-3$, that $\dtmor$ has full-rank, and $\vale v + \ch v = 3$ for all $v$ in $V(T)$.
		\item $3g - 3 = 2g+2d - 5$ implies that $g$ is even and $d = g/2 + 1$.
	\end{itemize}
	
	These observations give a practical criteria to check whether $\dtmor$ is full-dimensional:
	
	\begin{lm} \label{lemma-practical-criteria}
		A \DTmor $\dtmor$ is full-dimensional if and only if $A_\dtmor$ is a nonsingular $(3g-3) \times (3g-3)$ matrix.
	\end{lm}
	
	\begin{ex}
		Lemma~\ref{lemma-practical-criteria} and the calculation done in the fifth case of Example~\ref{ex-three-cycles-loop} shows that the map $\dtmor_A$ in Example~\ref{example-discrete-tropical-morphism} is full-dimensional.
	\end{ex}
	
	Let $H$ be a trivalent combinatorial type with even genus, and $\dtmor$ the full-dimensional \DTmor given by Theorem~\ref{theorem-cd}. Our aim here, in Part~I of this series of two articles is to introduce a deformation procedure that, when applied to $\dtmor$, yields full-dimensional \DTmors $\Aq \dtmor 1$, $\Aq \dtmor 2$, and so on, with cones $\Aq C 1$, $\Aq C 2$, and so on, which together cover $C_H$, showing that all metric graphs with combinatorial type $H$ have gonality at most $g/2 +1$. For the remaining we direct our efforts to the construction of this deformation procedure. We are not limited by the condition that $g$ is even, since the odd genus case follows from the even genus case, see Subsection~\ref{section-proof-main-result}.

	As a sneak peek, we mention that in Part~II of this series of two articles the main problem is the study of the set $\mathbb{FD}_g$ of full-dimensional \DTmors whose source has genus~$g$. A full-dimensional \DTmor has full-rank, so it is a combinatorial type. By the argument given at the end of the proof of Corollary~\ref{corollary-gonality-bound} we have that $\mathbb{FD}_g$ is finite. 
	
	We glue together the cones corresponding to elements in $\mathbb{FD}_g$ in a construction akin to that of $\MTrop g$ and obtain a space $\TM d g$. There is a projection $\Pi$ from $\TM d g$ to $\MTrop g$ where a tropical morphism $\tmor$ is mapped to its source. We show that $\Pi$ is a surjection,  that $\TM d g$ is connected in codimension-1, and that the fibres of $\Pi$ have Catalan-many points when counted with certain multiplicity. This opens the possibility of proving the upper bound on gonality without using Theorem~\ref{theorem-cd}: instead, we only need to study the fibre $\inv \Pi(\mH)$ for one particular kind of metric graph $\mH$, and the rest is deformation.

	\section{gluing datums} \label{sec-gluing-datum} 
	Let $\dtmor \!:\! G \to T$ be a \DTmor to a tree. That the fibre $\inv \dtmor (x)$ has $\deg \dtmor$ elements counted with multiplicity, and that $\dtmor$ is a graph homomorphism, leads us to regard $G$ as the graph resulting from taking $\deg \dtmor$ copies of $T$, and identifying together certain vertices and edges between copies. In this view the index map $m_\dtmor$ records how many copies of $T$ were glued together in a particular place. We make this notion precise with the combinatorial gadget of a \emph{gluing datum}. 
	
	
	\subsection{The gluing datum} \label{subsection-gluing-datum}  gluing datums are a tool to ease the visualization of \DTmorsp; to do book-keeping on edges of $G$ in the process of deformation of \DTmorsp{} (see Sections~\ref{sec-limits} and~\ref{sec-constructions}); and to write computer programs (see \cite{dra} for code). 

	\begin{de}[gluing datum] Let $T$ be a finite tree, $d$ be a positive
		integer, and $\sim$ be an equivalence relation on $ W = T \times [d]$, where $[d] = \{1,\dots,d\}$. Write $\class x k$ for the classes of $W/\sim$, where $x \in V(T) \cup E(T)$. The triple $(T, d, \sim)$ is a  \emph{gluing datum} when~$\sim$ satisfies these properties:
		\begin{enumerate}
			\item Verticality: If $(x, i) \sim (x', i')$, then $x = x'$. Each $x$ in $T$ defines a relation $\sim_x$ on $[d]$ with $i \sim_x j$ if $(x, i) \sim (x, j)$. We say that $\sim_x$ is the \emph{gluing relation above} $x$. By verticality these relations determine~$\sim$.
			
			\item Refinement: If $v$ is in $V(T)$, and $t$ is incident to $v$, then $\sim_t$ is a refinement of $\sim_v$.
			\item Connectedness: For any two classes $\class {x} {k}$ and $\class {x'} {k'}$ there is a sequence $\class {x_0} {i_0} = \class {x} {k}, \class {x_1} {i_1}, \dots, \class {x_r} {i_r} = \class {x'} {k'}$ such that for each $q$ the elements $x_q$, $x_{q+1}$ are incident (so one of them is a vertex and the other an edge), and there is at least one $j$ such that $(x_q, j) \in \class {x_q} {i_q}$ and $(x_{q+1}, j) \in \class {x_{q+1}} {i_{q+1}}$.  
			\item Riemann-Hurwitz condition: For $v$ in $V(T)$ and $i$ in $[d]$ let $A \subset [d]$ be the class of~$i$ under the relation $\sim_v$. Let $l = \vale v$ and $\Neigh v = \{t_1, \ldots, t_l\}$.
			By the refinement property the relation $\sim_{t_q}$ above $t_q$ partitions $A$ into $k_q$ sets.
			Then
			\[(k_1 + k_2 + \dots + k_l) - 2 \ge |A| \cdot (l-2). \]
		\end{enumerate}
	\end{de}

	\begin{figure}
		\includegraphics[scale=1.3]{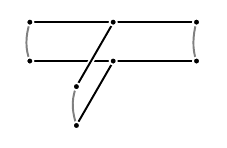} \hspace{1em}
		\includegraphics[scale=1]{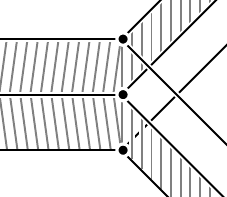} \hspace{1em}
		\includegraphics[scale=1]{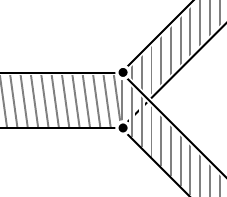} 
		\caption{\label{fig-gluing} On the left, a gluing datum giving rise to the theta graph. In the centre, an allowed set of gluing relations. On the right, a set of gluing relations not allowed because of the RH-condition.}
	\end{figure}
	
	Let $M = (T, d, \sim)$ be a gluing datum. Write $|\class x k|$ for the cardinality of the equivalence class $\class x k$. We abuse notation and say $i$ in $\class x k$ to mean $(x,i)$ in $\class x k$. By verticality this causes no confusion. A consequence is to regard $\class x k$ as a subset of~$[d]$, for the purpose of comparing classes. We do, but to avoid confusion it is pointed out in every such instance of this use that $\class x k$ is momentarily being regarded as a subset of~$[d]$, instead of a class of $W/\sim$.   
	
	We call $T$ the \emph{base tree} of $M$. We visualize gluing datums by drawing $d$ copies of $T$ on top of each other, and imagining all of them above $T$. Curved lines and shaded regions indicate where vertices or edges are identified. A class $\class x k$ is said to be \emph{above} $x$. By the verticality condition, one can only glue elements of $W$ that are above the same element of the base tree, see Figure~\ref{fig-gluing}.

	The quotient $G = W/\sim$ has the following natural structure of a loopless graph: the classes $\class x k$ in $W/\sim$ with $x$ in $V(T)$ are the vertices of $G$; those with $x$ in $E(T)$ are the edges of $G$; an edge $\class t i$ in $E(G)$ is incident to a vertex $\class v k$ in $V(G)$ if and only if $t$ and $v$ are incident in $T$ and $i \in \class v k$. This is a well defined graph, by the verticality and the refinement properties. It is connected, by the connectedness property. It has no loops by the verticality property and the fact that $T$ is a tree. Consider the natural map $\varphi_M: G \to T$ given by $\varphi_M \class x k = x$, and the index map $m_{\dtmor_M} \class x k = |\class x k|$. 
	
	\begin{lm}[Prop. 5, \cite{cd18}] For a gluing datum $M = (T, d, \sim)$, the map $\varphi_M \class x k = x$ with index map $m_{\dtmor_M} \class x k = |\class x k|$ is a \DTmor of degree $d$.
	\end{lm} 
	
	\begin{proof}	
		We prove that $\dtmor$ is a \DTmor by showing the balancing condition; the other requirements are immediate.
		Let $A$ be a vertex of $G$. Set $v = \dtmor(A)$. Fix $t$ in $\Neigh v$. Recall that by the refinement property the relation $\sim_t$ above $t$ partitions $A$. Thus, if $e_1, \dots, e_r$ are the classes in $E(A)$ above $t$, then $|A| = |e_1| + \dots + |e_r|$. Since $m_\dtmor(e) = |e|$ we get that
		\[m_\varphi(A) = |A| =  \sum_{\substack{e \in \Neigh A \\ \varphi(e) = t}} m_\varphi(e). \qedhere \] 
	\end{proof}	
	
	We illustrate the previous concepts in the following example.

	\begin{ex} \label{ex-running} 
		Let $M = (T,3,\sim)$ be a gluing datum with $T$ and $\sim$ as in Figure~\ref{figure-3}.
		
		\noindent	
		\begin{minipage}{1\textwidth}
			\begin{minipage}{.32\textwidth}
				\centering
				\begin{overpic}{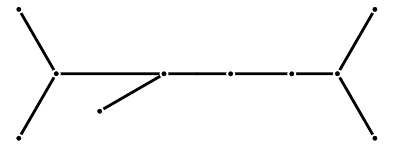}    
					\put (-3,35) {\scalebox{0.8}{$v_1$}}
					\put (-3,-2) {\scalebox{0.8}{$v_2$}}
					\put (14,21) {\scalebox{0.8}{$v_3$}} 
					\put (22,4) {\scalebox{0.8}{$v_4$}} 
					\put (40,21) {\scalebox{0.8}{$v_5$}}
					\put (56,21) {\scalebox{0.8}{$v_6$}}
					\put (71,21) {\scalebox{0.8}{$v_7$}} 
					\put (97,35) {\scalebox{0.8}{$v_8$}} 
					\put (97,-2) {\scalebox{0.8}{$v_9$}}
					
					\put (64,14) {\scalebox{0.8}{$e$}}
				\end{overpic} 
			\end{minipage}\hspace{1em}
			\begin{minipage}{.24\textwidth} 
				\begin{align*}
				&\sim_{v_1} = \sim_{v_2} = \{ \{1,2\},\{3\} \}\\
				&\sim_{v_4} = \sim_e = \{\{1\},\{2,3\}\}, \\
				&\sim_{v_1} = \sim_{v_2} = \{ \{1,3\},\{2\} \},\\
				&\text{else } \sim_x = \{ \{1\},\{2\},\{3\}\}.
				\end{align*}
			\end{minipage} \hspace{1em}
			\begin{minipage}{.30\textwidth} 
				
				\begin{overpic}{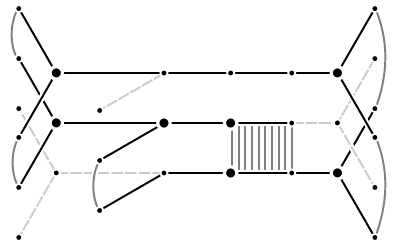}    
					\put (15,33) {\scalebox{0.8}{$A$}}
					\put (40,25) {\scalebox{0.7}{$B$}}
					\put (56,11) {\scalebox{0.7}{$C$}}
					\put (80,13) {\scalebox{0.7}{$D$}}
					\put (80,46) {\scalebox{0.7}{$E$}}
					\put (15,46) {\scalebox{0.7}{$F$}} 
				\end{overpic} 
			\end{minipage}
			\\
			\noindent \begin{minipage}[t]{.32\textwidth}
				\centering
				$T$			
			\end{minipage}\hspace{2em}
			\begin{minipage}[t]{.25\textwidth} 
				\centering
				$\sim$			
			\end{minipage}\hspace{1em}
			\begin{minipage}[t]{.32\textwidth} 
				\begin{center} 
					\hspace{1em} A visualization of $M$\end{center}
			\end{minipage}
			
			\captionof{figure}{} \label{figure-3}
		\end{minipage}
		
		\vspace{1em}
		
		In the visualization we have rendered dangling edges with dotted lines and lighter colour. The copies of $T$ are numbered from top to bottom. We also have labelled with $A, B, \dots, F$ the vertices of $G$ that have non-dangling valency equal to 3. The other non-dangling vertices have non-dangling valency equal to 2. Note that the map $\dtmor_M$ is isomorphic as a \DTmor to the map $\dtmor_A$ of Example~\ref{example-discrete-tropical-morphism} (see also Example~\ref{example-isomorphic-gluing-datums}).
	\end{ex}

	We denote by $\Neighnd{A}$ the set of non-dangling edges of $G$ incident to a vertex $A$ and define the \emph{non-dangling valency} $\nddeg A$ as $|\Neighnd{A}|$. 
	Observe that $\nddeg A = 0$ if and only if $A$ is dangling, otherwise $\nddeg A \ge 2$. In the visualization of $M$ in Example~\ref	{ex-running} the darker coloured edges make up the following graph:
	
	\begin{cons}
		Let $H(M)$ be the following graph: $V(H(M))$ is the set of vertices $A$ of $G$ such that $\nddeg A \ge 3$; and $E(H(M))$ is the set of paths of $G$ whose ends are in $V(H(M))$ and interior vertices have non-dangling valency equal to~2.
	\end{cons}
	
	Recall that $H(\dtmor_M)$ is the combinatorial type of the source of the \DTmor $\dtmor_M$ associated to $M$. It is straightforward to see that $H(M)$ and $H(\dtmor_M)$ are isomorphic. Observe also that the edge $e$ of $G$ is either dangling, or there is a unique $h(e)$ in $E(H(M))$ such that $e \in h(e)$. We say that $h$ in $E(H(M))$ \emph{passes through} a class $\class x k$ of $G$ if $\class x k \in h$ and $\class x k$ is not and end of~$h$. If $h$ passes through $\class x k$, then we say that $h$ \emph{passes above} $x$. This gives a natural way to also visualize $H(M)$ above~$T$, as in Example~\ref{ex-running}.
	
	\begin{ex}	\label{example-gluing-datum}	
		Figure~\ref{figure-5} illustrates $G$ and $H(M)$ for the gluing datum $M$ of Example~\ref{ex-running}. We kept the labelling for vertices with non-dangling valency greater than~2 from the visualization of $M$ in Example~\ref{ex-running}, for ease of comparison. 
		
		\vspace{1em}
		
		\noindent	
		\begin{minipage}{0.95\textwidth}
			\begin{minipage}[t]{0.6\textwidth}
				\centering
				\begin{overpic}[scale=0.6]{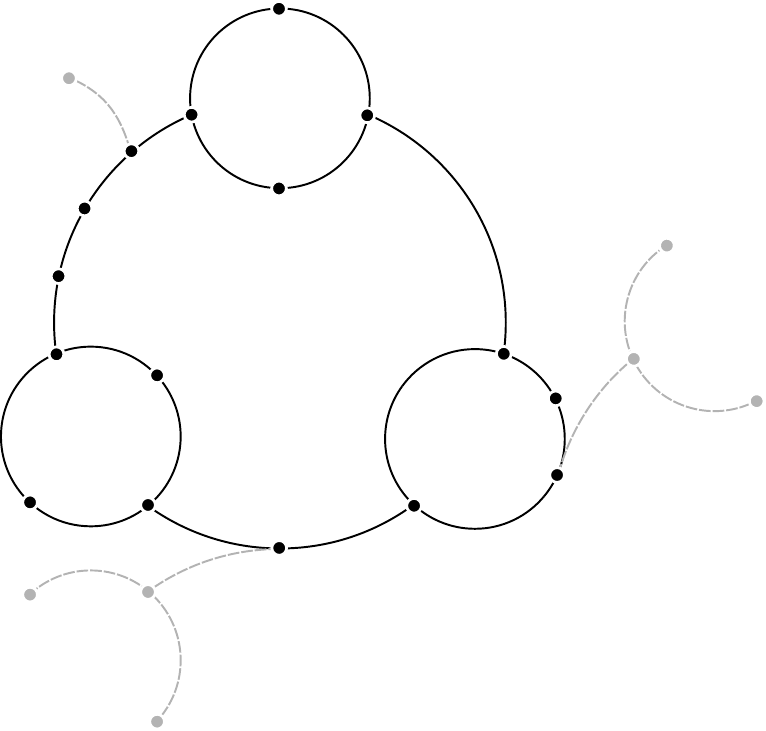}    
					\put (50,82) {\scalebox{0.8}{$A$}}
					\put (68,49) {\scalebox{0.8}{$B$}}
					\put (52,23) {\scalebox{0.8}{$C$}}
					\put (17,23) {\scalebox{0.8}{$D$}}
					\put (0,49) {\scalebox{0.8}{$E$}}
					\put (19,82) {\scalebox{0.8}{$F$}}
				\end{overpic} 
			\end{minipage}
			\begin{minipage}[b]{.35\textwidth} 
				\begin{overpic}[scale=0.6]{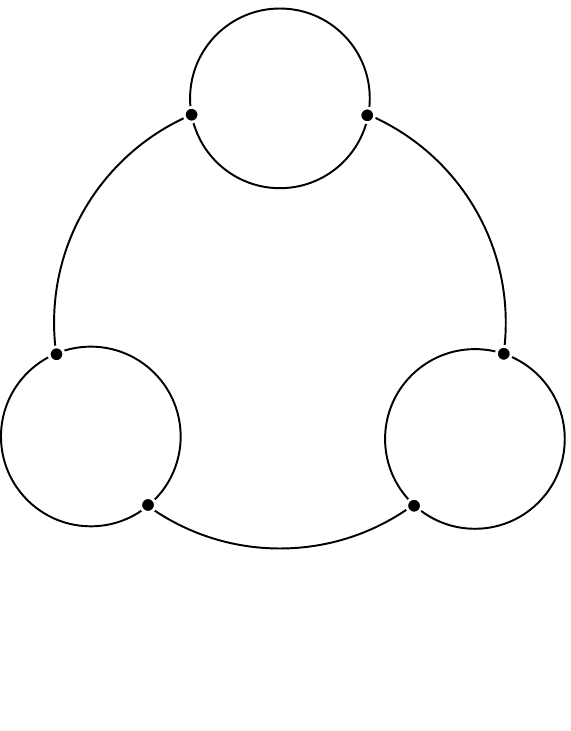}    
					\put (53,84) {\scalebox{0.8}{$A$}}
					\put (71,51) {\scalebox{0.8}{$B$}}
					\put (54,22) {\scalebox{0.8}{$C$}}
					\put (18,22) {\scalebox{0.8}{$D$}}
					\put (0,51) {\scalebox{0.8}{$E$}}
					\put (19,84) {\scalebox{0.8}{$F$}}
				\end{overpic} 
			\end{minipage}
			\\
			\noindent 
			\begin{minipage}[t]{.6\textwidth} 
				\hspace{9.2em}
				$G$			
			\end{minipage}\hspace{1em}
			\begin{minipage}[t]{.35\textwidth} 
				\hspace{3.5em}
				$H(M)$
			\end{minipage}
			\captionof{figure}{} \label{figure-5}
			
		\end{minipage}
		
	\end{ex}
	
	\begin{re} \label{rem-cd-glue}
		Unlike our approach here where a gluing datum gives rise to a \DTmor $\dtmor$, the definition of gluing datum in \cite{cd18} gives rise to a tropical morphism $\tmor : \mG \to \mT$. To go from a gluing datum as in \cite{cd18} to one as in our setting we need to remove the metric information by taking a model $T$ of $\mT$ where for every edge $t$ of $T$ the relation above each interior point $x$ in $\mT_t$ equals some fixed relation $\sim_t$. In the terminology of the Definition~4 of \cite{cd18} a possibility for such a model is the set of points from the essential model of $\mT$ together with the image under $\tmor$ of the set of monovalent points in the metric forest  \[\{u \in \mG \, | \, \exists x \in \mT \text{ with } \psi_i(x) \sim \psi_j (x)  \text{ for distinct } i, j \text{; and } u = \psi_i(x)/\sim \},\] where the map $\psi_i$ sents $\mT$ to its $i$-th copy.
	\end{re}
	
	\begin{cons}  \label{cons-dtmor-to-gd}  Given a degree-$d$ \DTmor $\dtmor:G \to T$ we construct a gluing datum $M$ with base tree $T$ such that $\dtmor$ is isomorphic to $\dtmor_M$. We use four steps:
		\begin{enumerate}[(1)]
			\item Choose a leaf $v$ of $T$ and a map $\chi_v : \inv \dtmor(v) \to \mathcal P([d])$, where $\mathcal P([d])$ is the power set of $\{1,\dots,d\}$, such that $|\chi_v(A)| = m_\varphi(A)$ for $A \in \inv \dtmor(v)$ and $\im(\chi_v)$ is a partition of $[d]$.
			\item Choose incident $x$, $x'$ in $T$ such that a map $\chi_{x}$ has been chosen for $x$ and no map has been chosen for $x'$ yet.
			\item Choose a map $\chi_{x'} : \inv \dtmor(x') \to \mathcal P([d])$ such that $|\chi_{x'}(X')| = m_\varphi(X')$ for $X' \in \inv \dtmor(x')$, $\im(\chi_{x'})$ is a partition of $[d]$; and for every pair of incident elements $(X',X)$ in $\inv \dtmor(x') \times \inv \dtmor(x)$ we have that $\chi_{x'}(X') \subset \chi_{x}(X)$ if $x'$ is an edge, or $\chi_{x'}(X') \supset \chi_{x}(X)$ otherwise.
			\item Repeat step 2 and 3 until every $x$ in $T$ has a $\chi_x$.
		\end{enumerate}
		Take $M = (T,d,\sim)$, where $\sim_x$ is defined by the partition $\im(\chi_{x})$.
	\end{cons}
	
	In Step~(3) of Construction~\ref{cons-dtmor-to-gd} the choice is possible since $\dtmor$ satisfies the balancing condition. It is straightforward to check that the gluing datum $(T,d,\sim)$ has the properties we desire. This construction is essentially the one given in  Proposition 6 of~\cite{cd18}.

	\subsection{Isomorphism classes of \DTmors and gluing datums} \label{subsection-isomorphism-classes-of-dtmors-and-gd} 
	The choices that are crucial in Construction~\ref{cons-dtmor-to-gd} happen at Steps~(1) and (3). We identify all the possibilities using a notion of isomorphism of gluing datums. It turns out that with the right notion of isomorphism there is a bijection between isomorphism classes of \DTmors and isomorphism classes of gluing datums. This completes the setup to study \DTmors via gluing datums. 
	
	Observe that applying a fixed permutation $\pi$ of $[d]$ on all relations $\sim_x$ gives a different gluing datum, but does not change the isomorphism type of~$G$, nor the index map for the edges of~$G$, nor which edges of $G$ lie above a particular edge of $T$. We call this operation a \emph{tree-swap}. It enables us to relate two distinct choices done at Step~(1) of Construction~\ref{cons-dtmor-to-gd}.
	
	Choose a vertex $v$ of $T$, and let $S$ be the vertex set of one of the connected components of $T \setminus \{v\}$. Let $T_{\text{br}}$ be the graph induced by the vertex set $S \cup \{v\}$. We call it a \emph{branch} of~$T$. Choose $i$ and $j$ such that $i \sim_v j$ and swap $i$ with $j$ in all relations above $T_{\text{br}}$. We call this a \emph{branch-swap}. A series of branch-swaps enables us to relate two distinct choices done at Step~(3) of Construction~\ref{cons-dtmor-to-gd}.
	This motivates the notion of isomorphism of gluing datums. In the following definition we write $\class x i_M$ for a class of $M$, to avoid confusions, and in general we regard $\class x i_M$ as a subset of $[d]$.
	
	\begin{de}[gluing datum isomorphism] \label{definition-gd-iso} An isomorphism from $M = (T, d, \sim)$ to $M' = (T', d, \sim')$ consists of a graph isomorphism $\tau:T \to T'$, and for each $x$ in $T$ a permutation $\pi_x$ of $[d]$ satisfying two properties:
		\begin{enumerate}
			\item class-preserving: $\pi_x \left( \class  x i_M \right)$ and $\class {\tau(x)} {\pi_x(i)}_{M'}$ are equal for all $x$ in $T$, $i$ in $[d]$.
			\item incidence-preserving: for all $v$ in~$V(T)$, $t$ in $E(v)$, and $\ i$ in $[d]$ we have that $\class t j_M \subset \class v i_M$ if and only if $\class {\tau(t)} {\pi_t(j)}_{M'} \subset \class {\tau(v)} {\pi_v(i)}_{M'}$.  
		\end{enumerate}	
	\end{de}

	\begin{ex} \label{example-isomorphic-gluing-datums} 
		Figure~\ref{figure-6} presents three isomorphic gluing datums. On the left, $M$ from Example~\ref{ex-running}. Next, in $M'$ we swapped the middle and the bottom tree. Finally, in $M''$ we also swapped the middle and the bottom tree, but only above the branch that is right to the vertex $C$. They all give rise to isomorphic \DTmorsp.
		
		\vspace{0.7em}
		
		\noindent	
		\begin{minipage}{\textwidth}
			\begin{minipage}{.3\textwidth}
				\centering
				\begin{overpic}{\figsdir/263.pdf}    
					\put (15,33) {\scalebox{0.8}{$A$}}
					\put (40,25) {\scalebox{0.7}{$B$}}
					\put (56,11) {\scalebox{0.7}{$C$}}
					\put (80,13) {\scalebox{0.7}{$D$}}
					\put (80,46) {\scalebox{0.7}{$E$}}
					\put (15,46) {\scalebox{0.7}{$F$}} 
				\end{overpic} 
			\end{minipage}\hspace{1em}
			\begin{minipage}{.3\textwidth} 
				\centering
				\begin{overpic}{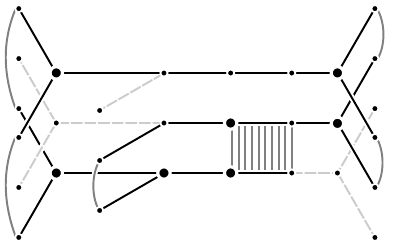}    
					\put (13,11) {\scalebox{0.8}{$A$}}
					\put (40,11) {\scalebox{0.7}{$B$}}
					\put (56,11) {\scalebox{0.7}{$C$}}
					\put (80,25) {\scalebox{0.7}{$D$}}
					\put (80,46) {\scalebox{0.7}{$E$}}
					\put (15,46) {\scalebox{0.7}{$F$}} 
				\end{overpic} 
			\end{minipage} \hspace{1em}
			\begin{minipage}{.3\textwidth} 
				\centering
				\begin{overpic}{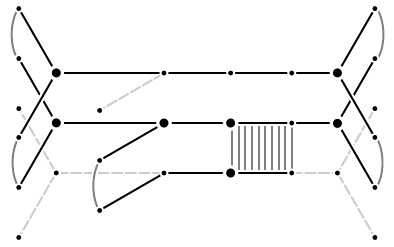}    
					\put (15,33) {\scalebox{0.8}{$A$}}
					\put (40,25) {\scalebox{0.7}{$B$}}
					\put (56,11) {\scalebox{0.7}{$C$}}
					\put (80,25) {\scalebox{0.7}{$D$}}
					\put (80,46) {\scalebox{0.7}{$E$}}
					\put (15,46) {\scalebox{0.7}{$F$}} 
				\end{overpic} 
			\end{minipage}
			\\
			\noindent \begin{minipage}[t]{.3\textwidth}
				\centering
				$M$			
			\end{minipage}\hspace{1em}
			\begin{minipage}[t]{.3\textwidth} 
				\centering
				$M'$			
			\end{minipage}\hspace{1em}
			\begin{minipage}[t]{.3\textwidth} 
				\centering
				$M''$
			\end{minipage}
			
			\captionof{figure}{} \label{figure-6}
		\end{minipage}

	\end{ex}
	\begin{lm} \label{lemma-iso-classes-dtmor-gd} The  isomorphism classes of gluing datums are in bijection with the isomorphism classes of \DTmors from a graph to a tree.
	\end{lm}
	
	\begin{proof}
		Let $M$, $M'$ be gluing datums giving rise to graphs $G$, $G'$ respectively, and $\dtmor = \dtmor_M$, $\dtmor' = \dtmor_{M'}$. First suppose that $\tau$, $\{\pi_x\}_{x \in T}$ is an isomorphism from $M$ to $M'$. Consider the natural map $\gamma \!:\! G \to G'$ given by $\gamma ( \class x k_M ) = \class {\tau(x)} {\pi_x(k)}_{M'} $. By the class-preserving property we have that $\gamma$ is bijective. By the incidence-preserving property we have that both $\gamma$, $\inv \gamma$ are graph homomorphisms. We claim that $(\gamma, \tau)$ is an isomorphism from $\dtmor$ to $\dtmor'$. Diagram~\ref{diagram-iso}.a commutes because
		\begin{align*}
		\dtmor' \circ \gamma \left( \class x k_M \right) &= \dtmor' \left(\class {\tau(x)} {\pi_x(k)}_{M'}\right)  = \tau(x),\\
		\tau \circ \dtmor \class x k &= \tau(x).
		\end{align*}
		Diagram~\ref{diagram-iso}.b commutes because by class-preserving we get that
		\begin{align*}
		m_{\dtmor} \left( \class x k_M \right) &= |\class x k_M | = |\class {\tau(x)} {\pi_x(k)}_{M'}|, \\
		m_{\dtmor'} \circ  \gamma \left( \class x k_M \right) &= m_{\dtmor'} \left( \class {\tau(x)} {\pi_x(k)}_{M'} \right)= |\class {\tau(x)} {\pi_x(k)}_{M'}|. 
		\end{align*}

		Now suppose that $\gamma$, $\tau$ is an isomorphism from $\dtmor$ to $\dtmor'$. Since $m_\dtmor = m_{\dtmor'} \circ \gamma$, we get that $|\class x k_M | = |\gamma ( \class x k_M )|$. Thus, for each $x$ in $T$ we can choose a permutation $\pi_x$ of~$[d]$ such that $\pi_x ( \class x k_M ) = \gamma ( \class x k_M )$. We claim that $\tau$, $\{\pi_x\}_{x \in T}$ is an isomorphism from $M$ to $M'$. We have the class-preserving property by construction. Since $\gamma$, $\inv \gamma$ are graph homomorphisms, we have that $\class t j_M$, $\class v i_M$ are incident if and only if $\gamma (  \class t j_M  )$, $\gamma (  \class v i_M  )$ are incident. Since $\dtmor' \circ \gamma = \tau \circ \dtmor$, 
		we get that $\gamma ( \class x i ) = \class {\tau(x)} {\pi_x(i)}$. The previous two facts and the construction of $G$ and $G'$ together imply that $\class t j_M \subset \class v i_M$ 
		if and only if $\class {\tau(t)} {\pi_t(j)}_{M'} \subset \class {\tau(v)} {\pi_v(i)}_{M'}$.
		
		Thus, we have shown that the map that sends the isomorphism class of $M$ to the isomorphism class of $\dtmor_M$ is well defined and injective. Construction~\ref{cons-dtmor-to-gd} shows that the map is also surjective.
	\end{proof} 
	
	Now that the correspondence between \DTmors and gluing datums is set up, we use the following notation for convenience:
	
	\begin{itemize}
		\item $C_M = C_{\dtmor_M}$, the cone of sources.
		\item $A_M = A_{\dtmor_M}$, the edge-length map.
		\item $(M, \ell_T) = (\dtmor_M, \ell_T)$, a tropical morphism. 
		\item We say that $M$ has full-rank, or is full-dimensional, or is a combinatorial type if $\dtmor_M$ has full-rank, or is full-dimensional, or is a combinatorial type, respectively.  
	\end{itemize}
	
	We also define the \emph{genus of a gluing datum} to be $g(H(M))$. Likewise, the genus of a \DTmor is the genus of its source.

	\subsection{Local properties} \label{sec-local} 
	Let $M = (T, d, \sim)$ be a genus-$g$ gluing datum and $\dtmor$ be $\dtmor_M \!:\! G \to T$. Ideally, a deformation procedure would modify a local part of $M$, namely the gluing relations above a particular edge of $T$ and its endpoints. The resulting $M'$ should be full-dimensional if we began with a full-dimensional $M$. The obstruction to this idea is that a local change in $M$ leads to a change in $A_M$ that may make the rank drop. Controlling the rank of $A_M$ is a global condition, which is harder to check than local conditions.
	
	We defer the hard problem of constructing full-dimensional gluing datums to Section~\ref{sec-constructions}, and explore a different question here: what conditions are necessary for a gluing datum to be full-dimensional. It turns out that full-dimensional gluing datums have some remarkable combinatorial properties, which are easier to check than the rank. This allows us to easily discard many potential constructions as not full-dimensional. We present these combinatorial properties in Definition~\ref{def-conditions}, after giving some auxiliary definitions.

	\begin{de} \label{def-auxiliary-conditions} Let $M = (T, d, \sim)$ be a gluing datum.
		\begin{itemize}
			\item A vertex $v$ of $T$ is \emph{change-minimal} if $\ch v + \vale v = 3$. 
			\item  A vertex $A$ of $G$ satisfies the \emph{no-return} condition if there are at least two non-dangling edges in $E(A)$ above different edges of $T$; in other words $|\dtmor(\Neighnd{A})|$ is at least~2.
			\item Let $h = \langle A_0, e_1, \dots, e_\mu, A_\mu \rangle$ be an edge of $H(M)$ (recall that this means that $h$ is a path in $G$, both ends of $h$ have $\nddeg$ value at least 3, and the inner vertices of $h$ have $\nddeg$ value equal to 2). We say that
			$h$ satisfies the \emph{pass-once condition} if $\dtmor$ restricted to the set $\{e_i \in h \,|\, \dtmor(e_i) \text { is not incident to a leaf} \}$ is injective.
		\end{itemize}
	\end{de} 
	
	\begin{de} \label{def-conditions} Let $M = (T, d, \sim)$ be a gluing datum.
		\begin{itemize}
			\item $M$ is \emph{change-minimal} if all vertices of $T$ are change-minimal.
			\item $M$ satisfies the \emph{dangling-no-glue} condition if $|\class x k| = 1$ for all dangling $|\class x k|$.
			\item $M$ satisfies the \emph{no-return} condition if all non-dangling vertices $A$ of $G$ such that $\dtmor(A)$ is not monovalent satisfy the no-return condition. 				
			\item $M$ satisfies the \emph{pass-once} condition if all edges of $H(M)$ satisfy the pass-once condition.			
		\end{itemize}
	\end{de}

	We proceed to describe informally the intuition behind each property and give an example before beginning with proofs. Being change-minimal is equivalent to saying that the correction term in the dimension formula equals zero, namely that $|E(T)| = 2g + 2d - 5$. By the discussion after Definition~\ref{def-full-dim} we have that full-dimensional implies change-minimal. Dangling-no-glue says that no identification is used in parts of $G$ that do not affect $\ell_{H(M)}$. This is a sort of ``efficiency'' condition. No-return forbids edges of $H(M)$ to change copies of $T$ in any place except above a leaf of $T$. Pass-once eases the calculation of $A_\dtmor$ (see Proposition~\ref{prop-adm-matrix}).
	
	\begin{ex} \label{example-local-properties} The gluing datum of Example~\ref{ex-running} is full-dimensional (see the calculation at 5.a of Example~\ref{ex-three-cycles-loop}). It can be checked that this gluing datum is change-minimal, satisfies dangling-no-glue, no-return, and pass-once; this agrees with our claim. Now we point out instances in the examples that do not fullfill some of the items in Definition~\ref{def-conditions}.

		\begin{itemize}
			\item Consider the map $\dtmor_B$ from Example~\ref{example-discrete-tropical-morphism}. The second vertex from left to right in the target is the only change-minimal vertex, and above it the no-return condition is violated. The pass-once condition is violated by the edge passing above the middle edge of the target. The dangling-no-glue condition is violated by all of the dangling elements.
			\item Consider the map $\dtmor_A$ from  Example~\ref{example-combinatorial-type-of-DTmors}. It does not satisfy dangling-no-glue because of the leftmost vertex in the target, and is not change-minimal. It satisfies no-return and pass-once.
		\end{itemize}
	\end{ex}
	
	Instead of showing that full-dimensional implies dangling-no-glue, no-return and pass-once, we derive these conditions from being change-minimal and having full-rank. This opens the possibility of using these results in future work for studying those special graphs that have gonality less than $\lceil g/2 \rceil + 1$. We begin with a generalization of Lemma~\ref{lemma-equal-columns}.
	
	\begin{lm}[zero change] \label{lemma-change-zero}  Let $v$ be in $V(T)$, $h$ an edge of~$H(M)$ such that $\dtmor(h) \cap E(v)$ contains two edges $t_1$ and $t_2$, and  $A_1,  \dots, A_r$ the vertices in $h \cap \inv \dtmor(v)$. If the following conditions are true:
		\begin{enumerate}[(a)]
			\item $r_\varphi(A_q) = 0 $ for all~$1 \le q \le r$,
			\item the ends of $h$ are not above $v$,
			\item $|e| = 1$ for dangling $e$ in $E(A_q)$ for all~$1 \le q \le r$,
		\end{enumerate}	
		then in the edge-length matrix we have $a_{ht_1} = a_{ht_2}$. 
	\end{lm} 
	
	Before proceeding to the proof, note that 
	the coefficient $a_{ht}$, for $h$ in $E(H(M))$ and $t$ in~$E(T)$, can be calculated as follows:     
	\[\ell_{H(M)}(h) = \sum_{e \in E(G) : \, e \in h } \ell_G(e) = \sum_{t \in E(T)} \left (\sum_{e \in h: \, \dtmor(e) = t} \dfrac 1 {|e|} \right ) \ell_T(t) = \sum_{t \in E(T)} a_{ht} \ell_T(t) ,\]  
	where the sum in the parenthesis equals zero if the index set is empty.
	
	\begin{proof} 
		Since $A_q$ is not an end of $h$ we have that $\nddeg A_q = 2$. Let $A_q = \class {v} {k_q}$, and $e$,$e'$ be the two edges in $\Neighnd {A_q}$. Lemma~\ref{lemma-formula-rphi} and conditions~(a) and~(c) imply that $2|{A_q}| = |e| + |e'|$, and since $|{A_q}| \ge |e|,|e|$ we get that $|{A_q}| = |e| = |e'|$, which makes the classes $\class{t_1}{k_q}$, $A_q$, $\class{t_2}{k_q}$ equal as subsets of $[d]$.  Thus, $a_{ht_1} = 1/\class{t_1}{k_1} + \dots + 1/\class{t_1}{k_r} = 1/\class{t_2}{k_1} + \dots + 1/\class{t_2}{k_r} = a_{ht_2}$.
	\end{proof}
	
	Note that Condition~(c) of Lemma~\ref{rem-leaves-min-change} is implied by dangling-no-glue. Recall that if $\vale v = 2$, then $r_\dtmor(A_q) = 0$ gives that $\vale A_q = 2$. Hence, $A_q$ is not an end of $h$, and there are no dangling edges incident to $A_q$, fulfilling Conditions~(b) and~(c). Now we prove dangling-no-glue.
	
	\begin{re}[change-minimal leaves] \label{rem-leaves-min-change}			
		Let $v$ in $V(T)$ be a leaf that is change-minimal, $A$ in $\inv \dtmor(v)$ non-dangling, and $t$ the edge incident to $v$. Since $A$ is non-dangling $\vale A \ge 2$, and as $\vale v = 1$ we get that $|A| \ge \vale A$. Thus, $r_\dtmor(A) = \vale A - 2 - |A|(\vale v - 2) \ge 2$. Since $\ch v = 2$ we get $r_\dtmor(A) = 2$, $\vale A = 2$, and $|A| = 2$. Thus, there is a unique non-dangling vertex above $v$, $\, \sim_t$ is trivial, and  exactly one edge of $H(M)$ passes above $v$ and~$t$. 
		
		This implies that in the edge-length matrix $A_M$ the column corresponding to $t$ has only one non-zero entry; it is a 2 in the row corresponding to the edge of $H(M)$ passing above~$v$. Thus, if $M$ is change-minimal and has full-rank, then each edge of $H(M)$ passes above at most one leaf; for otherwise $A_M$ would have two equal columns, a contradiction.
	\end{re}

	\begin{lm} \label{lemma-dangling-rphi}
		Let $M$ be change-minimal and full-rank, and $A$ in $V(G)$ dangling. Then $r_\varphi(A) = 0$.
	\end{lm}
	
	\begin{proof}
		Let $v = \dtmor(A)$. By change-minimal $r_\dtmor(A)$ equals  $0$, $1$, or $2$. If $r_\dtmor(A) = 1$, then Remark~\ref{rem-leaves-min-change} (change-minimal leaves) implies that $\vale v = 2$, and so $\ch v =  1$. Thus, if $h$ in $E(H(M))$ passes through $A'$ above $w_0$, then $r_\dtmor(A') = 0$, because $A'$ is non-dangling, so $A' \ne A$ and $r_\dtmor(A') + r_\dtmor(A) \le \ch v = 1$. Hence, Lemma~\ref{lemma-change-zero} (zero change) implies that the columns of $A_M$ corresponding to $t_1$, $t_2$ incident to $v$ are equal, contradicting that $M$ has full-rank. If $r_\dtmor(A) = 2$, Remark~\ref{rem-leaves-min-change} implies that $A$ is non-dangling, a contradiction.
	\end{proof}
	
	\begin{lm} \label{lemma-dangling-no-glue} 
		A change-minimal and full-rank $M$ satisfies the dangling-no-glue condition.
	\end{lm}
	

	
	\begin{proof}
		Observation~I: Let $A$ be a vertex such that $r_\dtmor(A) = 0$ and all elements of $\Neigh A$ have cardinality~1, except possibly one element $e$. Then by Lemma~\ref{lemma-formula-rphi} (formula for $r_\varphi$) we get that $r_\varphi(A) = 2(|A| - 1) - (|e|-1)$. As $r_\varphi(A) = 0$ we get $|e| + 1 = 2|A|$. Hence, $|A|=|e|=1$ since $|A| \ge |e|$.
		
		Now suppose that $e_0$ is dangling and $|e_0| > 1$. Choose an end $A_0$ of $e_0$. If $A_0$ is dangling then Lemma~\ref{lemma-dangling-rphi} allows us to apply Observation~I to choose $e_1$ in $\Neigh{A_0}$, distinct from $e_0$, such that $|e_1|>1$. Repeat the previous step with $e_1$, and continue to construct a sequence of vertices $A_0, A_1$, and so on; since $G$ is finite, either we arrive to a non-dangling $A_r$, or a vertex repeats in the sequence. Either way this means that after deleting $e_0$ the connected component of $A_0$ has a cycle. This argument can be applied to the other end of $e_0$ as well, giving a contradiction that $e_0$ is dangling. Thus, if $e$ is dangling, then $|e| = 1$. If $A$ is dangling, all edges in $\Neigh A$ are dangling, so they have cardinality~1, and the result follows from another application of Lemma~\ref{lemma-dangling-rphi} and Observation~I.
		%
		%
	\end{proof}
	
	Now we prove the no-return condition.
	
	\begin{lm}[nd.~$r_\dtmor$ formula] \label{lem-rphi-nd} 
		Let $M$ satisfy dangling-no-glue, and $A$ be in $V(G)$. If $e_1, e_2, \ldots, e_r$ in $E(G)$ are the non-dangling edges incident to $A$, then \[r_\varphi (A) = \nddeg A - 2 + 2|A| - (|e_1| + \ldots + |e_r|).\]
	\end{lm}
	\begin{proof}
		This is a consequence of Lemma~\ref{lemma-formula-rphi} and the dangling-no-glue condition.
	\end{proof}
	
	\begin{lm}[r1 implies no-return] \label{lem-noreturn-rphi1} 
		Let $M$ satisfy dangling-no-glue, and $A$ in $V(G)$ be non-dangling with $\vale \dtmor(A) \ge 2$. If $r_\varphi(A) \le 1$, then $A$ satisfies no-return.  
	\end{lm}
	\begin{proof}
		If $|A| = 1$ the result is clear. Assume that $|A| \ge 2$. Let $e_1, \ldots, e_r$ be the edges in $\Neighnd A$. Note that $r \ge 2$, since $A$ is non-dangling. Assume that all $e_q$ are above the same edge $t$, so $|e_1| + \ldots + |e_r| \le |A|$. Lemma~\ref{lem-rphi-nd} (nd.~$r_\dtmor$ formula) implies that $r_\varphi (A) = r - 2 + 2|A| - (|e_1| + \ldots + |e_r|) \ge r - 2 + |A| \ge 2$, a contradiction.   
	\end{proof}
	
	\begin{cor} \label{cor-no-return} 
		A change-minimal $M$ that satisfies the dangling-no-glue condition also satisfies the no-return condition.
	\end{cor}
	\begin{proof}
		Since $M$ is change-minimal by Remark~\ref{rem-leaves-min-change} if $r_\dtmor(A) \ge 2$, then $\dtmor(A)$ is a leaf. 
		The corollary follows then from Lemma~\ref{lem-noreturn-rphi1}. 
	\end{proof}
	
	By dangling-no-glue, $G$ looks like $T$ locally around dangling vertices. We now give a local description for non-dangling vertices, assuming a change-minimal~$M$. To make coming references to this result easier, we state all the possible cases, though several are straightforward corollaries of previous results. 
	
	\begin{prop}[local properties] \label{prop-local} 
		Let $M$ be change-minimal and satisfying dangling-no-glue, $A$ in $V(G)$ with $\nddeg A$ either 2 or 3, $v = \varphi(A)$, $r = \nddeg A$, and $\{e_1, \dots, e_r\} = \Neighnd A$. Then exactly one of the following cases happens:  
		\begin{itemize}[leftmargin=*] 
			\item (r0-nd3): $r_\varphi(A) = 0$ and $\nddeg A = 3$. Then $\dtmor$ restricted to $\Neighnd A$ is injective; $\vale v$ is $3$; and $|e_1| + |e_2| + |e_3| = 2|A| + 1$.
			
			\item (r0-nd2): $r_\varphi(A) = 0$ and $\nddeg A = 2$. Then $\dtmor$ restricted to $\Neighnd A$ is injective; $\vale v$ is 2 or 3; and $|e_1| = |A| = |e_2|$. 
			
			\item (r1-nd3): $r_\varphi(A) = 1$ and $\nddeg A = 3$. Then there is a labelling of the $e_i$ where $\dtmor(e_1) = \dtmor(e_2) \ne \dtmor(e_3)$; $\vale v$ is 2; and $|e_1| + |e_2| = |A| = |e_3|$. 
			
			\item (r1-nd2): $r_\varphi(A) = 1$ and $\nddeg A = 2$. Then $\dtmor$ restricted to $\Neighnd A$ is injective; $\vale v$ is 2; and there is a labelling of the $e_i$ where $|e_1| = |A|$ and $|e_2| = |A|-1$. 
			
			\item (r2-nd2): $r_\varphi(A) = 2$. Then $\dtmor(e_1) = \dtmor(e_2)$; $\vale v$ is 1; $|e_1| = |e_2|= 1$; and $|A| = 2$.
		\end{itemize}
	\end{prop}
	\begin{proof}
		Change-minimal implies that $r_\varphi(A)$ equals 0, 1, or 2, and that $\vale v$ equals 1, 2, 3 for $v$ in $V(T)$. By Remark~\ref{rem-leaves-min-change} (change-minimal leaves) $r_\varphi(A) = 2$ and $\nddeg A = 3$ cannot be. So the described cases are all the possibilities. The same remark proves~(r2-nd2). Note that $M$ satisfies the no-return condition by Corollary~\ref{cor-no-return}.
		
		For the case (r0-nd3), no-return implies that either the three non-dangling edges are above different edges of $T$, or (up to labelling) $e_2$ and $e_3$ are above the same edge and $e_1$ is above another. In the latter case Lemma~\ref{lem-rphi-nd} and $|e_1| \le |A|$ imply $0 = 1 + 2|A| - |e_1| - |e_2| - |e_3| \ge 1 + |A|  - |e_2| - |e_3|$, but $|e_2| + |e_3| \le |A|$, a contradiction.  Hence $e_1$, $e_2$, $e_3$ are above distinct edges of $T$, so $\vale v = 3$. Lemma~\ref{lem-rphi-nd} implies the last formula.   
		
		The remaining follows from change-minimal, dangling-no-glue, Lemma~\ref{lem-rphi-nd}, and the gluing datum axioms.
	\end{proof}
	
	\begin{cor} [image of a path] \label{cor-edge-of-h-injection}
		Let $M$ be change-minimal and satisfy dangling-no-glue. Let $P = \langle A_0, e_1, \ldots, e_\mu, A_\mu \rangle$ be a path of $G$ such that $A_i$ is non-dangling and $\vale \dtmor(A_i) \ge 2$ for each~$i$. Then $\langle \dtmor(A_0), \dtmor(e_1), \ldots, \dtmor(e_\mu), \dtmor(A_\mu) \rangle$ is a path of $T$.
	\end{cor}
	
	\begin{proof}
		Since $\dtmor$ is a graph morphism, and by the local properties of Proposition~\ref{prop-local} , we get that consecutive elements of $\dtmor(P) := \langle \dtmor(A_0), \dtmor(e_1), \ldots, \dtmor(e_\mu), \dtmor(A_\mu) \rangle$ are incident, and distinct. Since consecutive elements of $\dtmor(P)$ are distinct, if $\dtmor(P)$ repeats an element, then we get a cycle. As $T$ is a tree we conclude that $\dtmor(P)$ is a path.
	\end{proof}
	
	Now we move on to the pass-once condition.
	
	\begin{lm} \label{lemma-loop-12} Let $M$ be change-minimal and have full-rank, $u$ and $v$ be divalent and monovalent vertices of $T$, respectively, with $\Neigh u = \{t', t\}$, $\Neigh v = \{t\}$. Then there is exactly one edge of $H$ above $t$ and one above $t'$; they are a loop and a bridge, respectively.
	\end{lm}   
	
	\begin{proof}
		Swap trees so that $1 \sim_v 2$. Remark~\ref{rem-leaves-min-change} (change-minimal leaves) yields that $e_1 = \class t 1$ and $e_2 = \class t 2$ are two distinct edges of~$G$, with $|e_1| = |e_2| = 1$, and are the only non-dangling edges above~$t$. This implies, by the no-return condition, that $\class u 1$ and $\class u 2$ are the only non-dangling classes above~$u$ (it is possible that $\class u 1 = \class u 2$). Let $h$ be the edge of $H$ containing $e_1$, $e_2$. If both $\class u 1$ and $\class u 2$ have non-dangling valency equal to~2, then the only edge of $H$ passing above $t$, and $t'$ is $h$. Thus, $a_{ht}$ and $a_{ht'}$ are the only non-zero entries in the columns of $A_M$ corresponding to $t$ and $t'$, respectively, a contradiction. Therefore, assume without loss of generality that $\nddeg \class u 1 = 3$. As $\vale u = 2$ and $|e_1| = 1$, by Case~(r1-nd3) of the local properties we must have that $\Neighnd{\class u 1}$ has two edges above $t$; namely $e_1$, $e_2$, that make the loop; and one edge above $t'$, the bridge.
	\end{proof}  
	
	\begin{lm} \label{lemma-pass-once}A change-minimal and full-rank $M$ satisfies the pass-once condition
	\end{lm}
	\begin{proof}
		Let $h = \langle A_0, e_1, \ldots, e_\mu, A_\mu \rangle$ be an edge of $H(M)$. If $h$ does not pass above a leaf, then we are done by Corollary~\ref{cor-edge-of-h-injection}. Assume $h$ passes above a leaf. By Remark~\ref{rem-leaves-min-change} (change-minimal leaves) there is a single vertex $A_i \in h$ such that $v := \dtmor(A_i)$ is a leaf. Thus, $h$ passes at most twice above any edge of $T$ by Corollary~\ref{cor-edge-of-h-injection}. We also get that $\dtmor(e_i) = \dtmor(e_{i+1})$ from Remark~\ref{rem-leaves-min-change}. Let $t_1 = \dtmor(e_i)$. Suppose that $h$ violates the pass-once condition. That is, for some $j,k\ \not \in \{i,i+1\}$ we have $\dtmor(e_j) = \dtmor(e_{k})$. We get that $\dtmor(e_{i-1}) = \dtmor(e_{i+2})$, because otherwise $\dtmor(\langle A_0, \ldots A_i \rangle)$ and $\dtmor(\langle A_i, \ldots A_\mu \rangle)$ would intersect only at $t_1$ by Corollary~\ref{cor-edge-of-h-injection} and the fact that $T$ is a tree, which contradicts that $\dtmor(e_j) = \dtmor(e_{k})$. Let $t_2 = \dtmor(e_{i-1})$.
		
		Label the copies of $T$ such that $1 \sim_v 2$. Label $h$ with 1. Lemma~\ref{lemma-loop-12} implies that $\vale w = 3$. Let $w$ be the other end of $t_1$, and $t_3$  be the other edge incident to~$w$. We investigate the classes above $t_1, t_2$ and $t_3$. Observe that $\class {t_3} 1$ and $\class {t_3} 2$ are dangling. Change-minimal and Remark~\ref{rem-leaves-min-change} imply that $|\class {t_1} 1| = |\class {t_1} 2| = 1$. All the vertices of $G$ above $w$ have $r_\varphi$-value zero. Thus, case (r0-nd2) of the local properties on $\class w 1$ and $\class w 2$ implies that $|\class {t_2} 1| = |\class {t_2} 2| = 1$.  Hence, $a_{1,1} = 2, a_{1,2} = 2, a_{1,3} = 0$. As $\class {t_1} k$ is dangling for $k \ge 3$, the other non-dangling vertices above $w$ have $\nddeg$ equal to 2. Lemma~\ref{lemma-change-zero} (zero change) implies that $a_{i,2} = a_{i,3}$ for $i \ge 2$.  See Figure~\ref{figure-7}, dangling edges shown dotted.

		\noindent	  
		\begin{minipage}{\textwidth}
			\begin{minipage}[b]{.5\textwidth}
				\centering
				\begin{overpic}[scale=1.5]{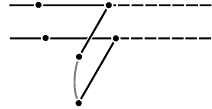}
					\put (35,-6) {\scalebox{1}{$v$}}
					\put (55,24) {\scalebox{1}{$w$}}    
					\put (47,9) {\scalebox{1}{$t_1$}}
					\put (25,24) {\scalebox{1}{$t_2$}}
					\put (77,24) {\scalebox{1}{$t_3$}}
					\put (-2,47) {\scalebox{1}{$e$}}		
				\end{overpic}  
			\end{minipage}%
			\begin{minipage}[b]{0.5\textwidth}
				\centering
				\begin{align*}
				&a_{1,1} = 2 
				\qquad a_{1,2} = 2
				\qquad a_{1,3} = 0\\
				&a_{i,1} = 0, \qquad \qquad 
				a_{i,2} = a_{i,3}, \, \\ &\text{for } i \ge 2.
				\end{align*}
				~
			\end{minipage}
			\captionof{figure}{} \label{figure-7}
		\end{minipage}	
		
		\vspace{0.5em}
		
		Let $\mathbf a_1, \mathbf a_2$, and $\mathbf a_3$ be the first three columns of $A_M$. Note that $\mathbf a_2 = \mathbf a_1 + \mathbf a_3$, contradicting that $M$ has full-rank.
	\end{proof}

	\begin{prop}[edge-length map is local] \label{prop-adm-matrix} 
		Let $M = (T,d,\sim)$ satisfy that the leaves of $T$ are change-minimal. Then $M$ satisfies the pass-once condition if and only if for all $h$ in $E(H)$, and $t$ in $E(T)$ the entry $a_{ht}$ is:
		\begin{enumerate}[(a)]
			\item 2 if $h$ passes above a leaf $v$ in $ V(T)$, and $t$ is incident to $v$.
			\item $1/|\class {t} k|$ if $h$ passes through some edge $\class {t} k$ of $G$, $t$ not incident to a leaf. 
			\item 0 otherwise.
		\end{enumerate}
	\end{prop}
	\begin{proof}
		Part (a) follows from Remark~\ref{rem-leaves-min-change}. Part (b) is true if and only if pass-once is. 
	\end{proof}

	
	
	\section{Examples} \label{sec-examples}  
	
	\begin{ex}[3 loops on a loop] \label{ex-three-cycles-loop}  Let $\mH = (H, \ell_H)$ be a loop of 3 loops (as in \cite{lpp12}), with $H$, vertex and edge labellings shown in Figure~\ref{figure-8}. Set $y_i = \ell_H(h_i)$, $z_i = \ell_T(t_i)$.
		
		\vspace{1em}
		
		\begin{minipage}{\textwidth}
			\centering
			\begin{overpic}[scale=0.6]{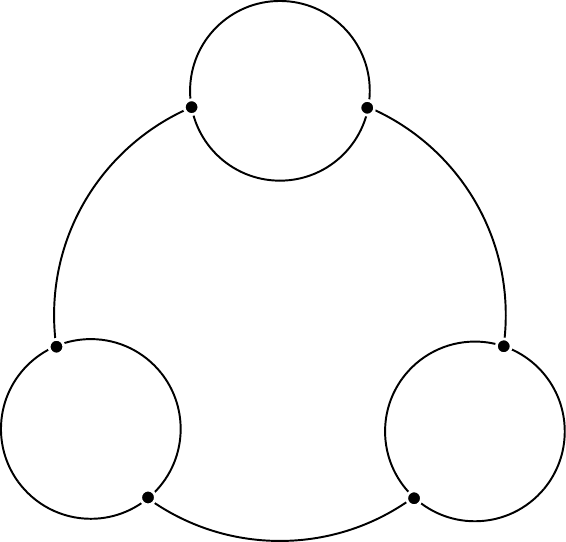}  
				\put (78,53) {\scalebox{0.8}{$y_1$}}
				\put (47,3) {\scalebox{0.8}{$y_2$}}
				\put (15,53) {\scalebox{0.8}{$y_3$}}
				
				\put (72,24) {\scalebox{0.8}{$y_4$}}
				\put (89,11) {\scalebox{0.8}{$y_5$}}
				\put (23,24) {\scalebox{0.8}{$y_6$}}
				
				\put (5,10) {\scalebox{0.8}{$y_7$}}
				\put (47,66) {\scalebox{0.8}{$y_8$}} 
				\put (47,89) {\scalebox{0.8}{$y_9$}}
				
				\put (67,78) {\scalebox{0.8}{$A$}}
				\put (90,35) {\scalebox{0.8}{$B$}}
				\put (70,0) {\scalebox{0.8}{$C$}}
				\put (23,0) {\scalebox{0.8}{$D$}}
				\put (3,35) {\scalebox{0.8}{$E$}}
				\put (27,78) {\scalebox{0.8}{$F$}}
			\end{overpic}
			
			\vspace{0.5em}
			$H$
			\captionof{figure}{} \label{figure-8}
		\end{minipage}
		
		\vspace{0.5em}
		
		Assume without loss of generality that $y_3 = \max(y_1, y_2, y_3)$, $y_4 \le y_5$, $y_6 \le y_7$, and $y_8 \le y_9$. We show 9 cases of glueing datums which realize $\mH$. These depend on how large the length $y_3$ is, compared to several other linear combinations of lengths. 
		Each case depicts: the base tree $T$ with an edge labelling; the glueing datum which realizes $\mH$; and its edge-length matrix using as labelling for $H$ the one shown in Figure~\ref{figure-8}. The odd cases are full-dimensional, the even cases are not. In two consecutive odd cases the edge-length matrices differ only by one column. 
		We compute $\det A_M$ to show these are indeed full-dimensional.
		
		
		\begin{itemize}[leftmargin=*]  
			\item First case: if $y_3 < y_1 + y_2$, then $M_1$ in Figure~\ref{figure-9} realizes $\mH$.
			
			\vspace{0.5em}
			
			\noindent 
			\begin{minipage}{\textwidth}
				\begin{minipage}{.32\textwidth}
					\centering
					\begin{overpic}[scale=1]{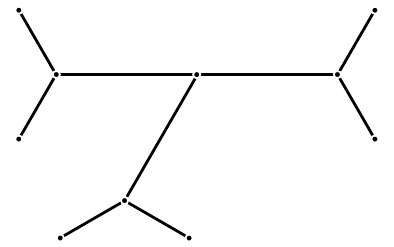}    
						\put (42,23) {\scalebox{0.8}{$z_1$}} 
						\put (65,46) {\scalebox{0.8}{$z_2$}} 
						\put (30,46) {\scalebox{0.8}{$z_3$}} 
						\put (41,9) {\scalebox{0.8}{$z_4$}} 
						\put (18,9) {\scalebox{0.8}{$z_5$}} 
						\put (82,32) {\scalebox{0.8}{$z_6$}} 
						\put (82,50) {\scalebox{0.8}{$z_7$}}
						\put (12,32) {\scalebox{0.8}{$z_8$}} 
						\put (12,50) {\scalebox{0.8}{$z_9$}} 
					\end{overpic} 
					
				\end{minipage}\hspace{1em}
				\begin{minipage}{.32\textwidth} 
					\centering
					\begin{overpic}{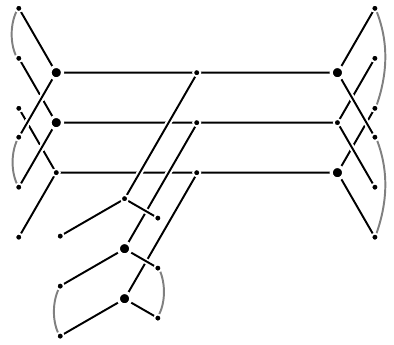}    
						\put (15,58) {\scalebox{0.8}{$A$}}
						\put (25,24) {\scalebox{0.7}{$B$}}
						\put (25,11) {\scalebox{0.7}{$C$}}
						\put (80,45) {\scalebox{0.7}{$D$}}
						\put (80,71) {\scalebox{0.7}{$E$}}
						\put (15,71) {\scalebox{0.7}{$F$}}
					\end{overpic} 
					
				\end{minipage}\hspace{1em}
				\begin{minipage}{.25\textwidth} 
					\centering 
					$\left(\begin{smallmatrix}
					1 & 0 & 1 &\     &   &   &\     &   &   \\
					1 & 1 & 0 &\     &   &   &\     &   &  \\
					0 & 1 & 1 &\     &   &   &\     &   &   \\
					&   &   &\   2 &   &   &\     &   &   \\  
					&   &   &\     & 2 &   &\     &   &   \\ 
					&   &   &\     &   & 2 &\     &   &   \\  
					&   &   &\     &   &   &\   2 &   &   \\ 
					&   &   &\     &   &   &\     & 2 &   \\
					&   &   &\     &   &   &\     &   & 2 \\
					\end{smallmatrix} \right)$ 
					
				\end{minipage}\\
				\noindent \begin{minipage}[t]{.32\textwidth}
					\centering
					$T_1$			
				\end{minipage}\hspace{1em}
				\begin{minipage}[t]{.32\textwidth} 
					\centering
					$M_1$			
				\end{minipage}\hspace{1em}
				\begin{minipage}[t]{.25\textwidth} 
					\centering 	
					$\det A_1 =  128$
				\end{minipage}
				\captionof{figure}{} \label{figure-9}
				\vspace{1em}
			\end{minipage}

			\item Second case: if $y_3 = y_1 + y_2$,  then $M_2$ in Figure~\ref{figure-10} realizes $\mH$.
			Note that edge $t_1$ was contracted.
			
			\noindent	
			\begin{minipage}{\textwidth}
				\begin{minipage}{.32\textwidth}
					\centering
					\begin{overpic}{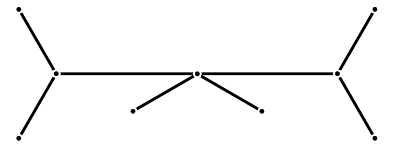}   
						\put (65,21) {\scalebox{0.8}{$z_2$}} 
						\put (30,21) {\scalebox{0.8}{$z_3$}} 
						\put (56,7) {\scalebox{0.8}{$z_4$}} 
						\put (40,7) {\scalebox{0.8}{$z_5$}} 
						\put (82,7) {\scalebox{0.8}{$z_6$}} 
						\put (81,25) {\scalebox{0.8}{$z_7$}}
						\put (12,7) {\scalebox{0.8}{$z_8$}} 
						\put (12,25) {\scalebox{0.8}{$z_9$}} 
					\end{overpic} 
				\end{minipage}\hspace{1em}
				\begin{minipage}{.32\textwidth} 
					\centering
					\begin{overpic}{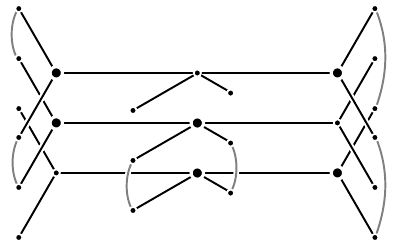}    
						\put (15,33) {\scalebox{0.8}{$A$}}
						\put (47,33.5) {\scalebox{0.7}{$B$}}
						\put (47,20) {\scalebox{0.7}{$C$}}
						\put (80,20) {\scalebox{0.7}{$D$}}
						\put (80,46) {\scalebox{0.7}{$E$}}
						\put (15,46) {\scalebox{0.7}{$F$}} 
					\end{overpic} 
				\end{minipage}\hspace{1em}
				\begin{minipage}{.25\textwidth} 
					\centering 
					$\left(\begin{smallmatrix}
					0 & 1 &\     &   &   &\     &   &   \\
					1 & 0 &\     &   &   &\     &   &  \\
					1 & 1 &\     &   &   &\     &   &   \\
					&   &\   2 &   &   &\     &   &   \\  
					&   &\     & 2 &   &\     &   &   \\ 
					&   &\     &   & 2 &\     &   &   \\  
					&   &\     &   &   &\   2 &   &   \\ 
					&   &\     &   &   &\     & 2 &   \\
					&   &\     &   &   &\     &   & 2 \\
					\end{smallmatrix} \right)$ 
					
				\end{minipage}\\
				\noindent \begin{minipage}[t]{.32\textwidth}
					\centering
					$T_2$			
				\end{minipage}\hspace{1em}
				\begin{minipage}[t]{.32\textwidth} 
					\centering
					$M_2$			
				\end{minipage}\hspace{1em}
				\begin{minipage}[t]{.25\textwidth} 
					\centering 	
					$A_2$ \\
				\end{minipage}
				\captionof{figure}{} \label{figure-10}
			\end{minipage}
			
			\vspace{0.5em}
			
			\item Third case: if $y_1 + y_2 < y_3 < y_1 + y_2 + y_4$, Figures~\ref{figure-11} and~\ref{figure-12} show realizations of~$\mH$.
			\begin{itemize}[leftmargin=*] 
				\item Case 3.a 
				
				\noindent	
				\begin{minipage}{\textwidth}
					\begin{minipage}{.32\textwidth}
						\centering
						\begin{overpic}{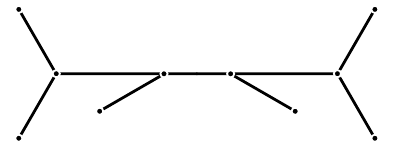}    
							\put (48,21) {\scalebox{0.8}{$z_1$}}
							\put (70,21) {\scalebox{0.8}{$z_2$}} 
							\put (28,21) {\scalebox{0.8}{$z_3$}} 
							\put (62,9) {\scalebox{0.8}{$z_4$}} 
							\put (34,9) {\scalebox{0.8}{$z_5$}} 
							\put (82,7) {\scalebox{0.8}{$z_6$}} 
							\put (82,25) {\scalebox{0.8}{$z_7$}}
							\put (12,7) {\scalebox{0.8}{$z_8$}} 
							\put (12,25) {\scalebox{0.8}{$z_9$}} 
						\end{overpic} 
						
					\end{minipage}\hspace{1em}
					\begin{minipage}{.32\textwidth} 
						\centering
						\begin{overpic}{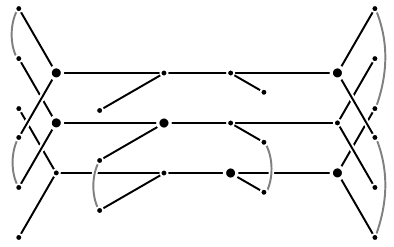}    
							\put (15,33) {\scalebox{0.8}{$A$}}
							\put (40,33) {\scalebox{0.7}{$B$}}
							\put (55,20) {\scalebox{0.7}{$C$}}
							\put (80,20) {\scalebox{0.7}{$D$}}
							\put (80,46) {\scalebox{0.7}{$E$}}
							\put (15,46) {\scalebox{0.7}{$F$}}
						\end{overpic} 
						
					\end{minipage}\hspace{1em}
					\begin{minipage}{.25\textwidth} 
						\centering 
						$\left(\begin{smallmatrix}
						0 & 0 & 1 &\     &   &   &\     &   &   \\
						0 & 1 & 0 &\     &   &   &\     &   &  \\
						1 & 1 & 1 &\     &   &   &\     &   &   \\
						1 &   &   &\   2 &   &   &\     &   &   \\  
						1 &   &   &\     & 2 &   &\     &   &   \\ 
						&   &   &\     &   & 2 &\     &   &   \\  
						&   &   &\     &   &   &\   2 &   &   \\ 
						&   &   &\     &   &   &\     & 2 &   \\
						&   &   &\     &   &   &\     &   & 2 \\
						\end{smallmatrix} \right)$ 
					\end{minipage}
					\\
					\noindent \begin{minipage}[t]{.32\textwidth}
						\centering
						$\Aq {T_3} a$			
					\end{minipage}\hspace{1em}
					\begin{minipage}[t]{.32\textwidth} 
						\centering
						$M_3$			
					\end{minipage}\hspace{1em}
					\begin{minipage}[t]{.25\textwidth} 
						\centering 	
						$\det A_3 =  -64$
					\end{minipage}
					\captionof{figure}{} \label{figure-11}
				\end{minipage}
				
				\item 	Case 3.b
				
				\noindent	
				\begin{minipage}{\textwidth}
					\begin{minipage}{.32\textwidth}
						\centering
						\begin{overpic}{\figsdir/052.pdf}    
							\put (48,21) {\scalebox{0.8}{$z_1$}}
							\put (70,21) {\scalebox{0.8}{$z_2$}} 
							\put (28,21) {\scalebox{0.8}{$z_3$}} 
							\put (62,9) {\scalebox{0.8}{$z_5$}} 
							\put (34,9) {\scalebox{0.8}{$z_4$}} 
							\put (82,7) {\scalebox{0.8}{$z_6$}} 
							\put (82,25) {\scalebox{0.8}{$z_7$}}
							\put (12,7) {\scalebox{0.8}{$z_8$}} 
							\put (12,25) {\scalebox{0.8}{$z_9$}} 
						\end{overpic} 
					\end{minipage}\hspace{1em}
					\begin{minipage}{.32\textwidth} 
						\centering
						\begin{overpic}{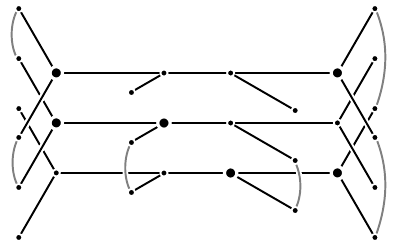}    
							\put (15,33) {\scalebox{0.8}{$A$}}
							\put (40,33) {\scalebox{0.7}{$B$}}
							\put (55,20) {\scalebox{0.7}{$C$}}
							\put (80,20) {\scalebox{0.7}{$D$}}
							\put (80,46) {\scalebox{0.7}{$E$}}
							\put (15,46) {\scalebox{0.7}{$F$}} 
						\end{overpic} 
					\end{minipage}\hspace{1em}
					\begin{minipage}{.25\textwidth} 
						\centering 
						$\left(\begin{smallmatrix}
						0 & 0 & 1 &\     &   &   &\     &   &   \\
						0 & 1 & 0 &\     &   &   &\     &   &  \\
						1 & 1 & 1 &\     &   &   &\     &   &   \\
						1 &   &   &\   2 &   &   &\     &   &   \\  
						1 &   &   &\     & 2 &   &\     &   &   \\ 
						&   &   &\     &   & 2 &\     &   &   \\  
						&   &   &\     &   &   &\   2 &   &   \\ 
						&   &   &\     &   &   &\     & 2 &   \\
						&   &   &\     &   &   &\     &   & 2 \\
						\end{smallmatrix} \right)$ 
					\end{minipage}
					\\
					\noindent \begin{minipage}[t]{.32\textwidth}
						\centering
						$\Aq {T_3} b$			
					\end{minipage}\hspace{1em}
					\begin{minipage}[t]{.32\textwidth} 
						\centering
						$M_3$			
					\end{minipage}\hspace{1em}
					\begin{minipage}[t]{.25\textwidth} 
						\centering 	
						$A_3$ \\
					\end{minipage}
					\captionof{figure}{} \label{figure-12}
				\end{minipage}
				\vspace{0.5em}
				
			\end{itemize}
			
			Here $\mH$ is realized by two tropical morphisms (distinct ones in general, as we argue next), yet at the combinatorial level the gluing datum is the same. We drew $M_3$ twice to illustrate which of $y_4$, $y_5$ is longer; and to show the labelling of the base tree. The two length functions $\Aq z a$, $\Aq z b$ are related by swapping $z_4$ and $z_5$. 
			
			We claim that $\Aq \tmor a = (\dtmor_3,\ \Aq z a)$ and $\Aq \tmor b = (\dtmor_3,\ \Aq z b)$ are distinct tropical morphisms if and only if $z_4$, $z_5$ are distinct lengths. Set $v = \dtmor (B)$. 
			Since $z_4$, $z_5$ are distinct, one of the two edges between $B$, $C$ is longer than the other, and there is a preimage of $v$ under $\Aq \tmor a$ in this longer edge. This is not true for $\Aq \tmor b$. Conversely, it is clear that if  $z_4$, $z_5$ are equal then swapping the lengths has no effect.
			
			This situation persists for the remaining gluing datums in this example.
			
			\item Fourth case: if $y_3 = y_1 + y_2 + y_4$, then Figures~\ref{figure-13} and~\ref{figure-14} show realizations of~$\mH$. The edge $t_4$ was contracted. The drawing shows two different gluing datums (which reflects a bifurcation in the shrinking and regrowing sequence so far), but by swapping some branches one can see them to be isomorphic, so we label both by $M_4$.
			This observation applies as well to the remaining cases.
			
			\begin{itemize}[leftmargin=*] 
				\item Case 4.a
				
				\noindent	
				\begin{minipage}{\textwidth}
					\begin{minipage}{.32\textwidth}
						\centering
						\begin{overpic}{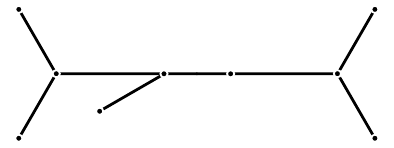}    
							\put (48,21) {\scalebox{0.8}{$z_1$}}
							\put (65,21) {\scalebox{0.8}{$z_2$}} 
							\put (28,21) {\scalebox{0.8}{$z_3$}} 
							
							\put (82,7) {\scalebox{0.8}{$z_6$}} 
							\put (82,25) {\scalebox{0.8}{$z_7$}}
							\put (12,7) {\scalebox{0.8}{$z_8$}} 
							\put (12,25) {\scalebox{0.8}{$z_9$}}
							
							\put (32,9) {\scalebox{0.8}{$z_5$}}   
						\end{overpic}

					\end{minipage}\hspace{1em}
					\begin{minipage}{.32\textwidth} 
						\centering
						\begin{overpic}{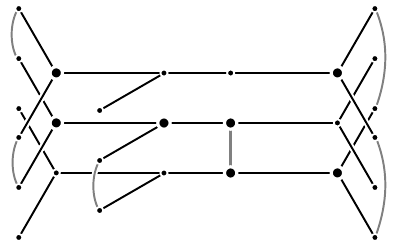}     
							\put (15,33) {\scalebox{0.8}{$A$}}
							\put (40,25) {\scalebox{0.7}{$B$}}
							\put (57,11) {\scalebox{0.7}{$C$}}
							\put (80,13) {\scalebox{0.7}{$D$}}
							\put (80,46) {\scalebox{0.7}{$E$}}
							\put (15,46) {\scalebox{0.7}{$F$}} 
						\end{overpic}

					\end{minipage}\hspace{1em}
					\begin{minipage}{.25\textwidth} 
						\centering  
						$\left(\begin{smallmatrix}
						0 & 0 & 1 &\      &   &\     &   &   \\
						0 & 1 & 0 &\      &   &\     &   &  \\
						1 & 1 & 1 &\      &   &\     &   &   \\
						1 &   &   &\      &   &\     &   &   \\  
						1 &   &   &\    2 &   &\     &   &   \\ 
						&   &   &\      & 2 &\     &   &   \\  
						&   &   &\      &   &\   2 &   &   \\ 
						&   &   &\      &   &\     & 2 &   \\
						&   &   &\      &   &\     &   & 2 \\
						\end{smallmatrix} \right)$ 
					\end{minipage}
					\\
					\noindent \begin{minipage}[t]{.32\textwidth}
						\centering
						$\Aq {T_4} a$			
					\end{minipage}\hspace{1em}
					\begin{minipage}[t]{.32\textwidth} 
						\centering
						$M_4$			
					\end{minipage}\hspace{1em}
					\begin{minipage}[t]{.25\textwidth} 
						\centering 	
						$A_4$ \\
					\end{minipage}
					\captionof{figure}{} \label{figure-13}
				\end{minipage}
				\vspace{0.5em}

				\item Case 4.b
				
				\noindent	
				\begin{minipage}{\textwidth}
					\begin{minipage}{.32\textwidth}
						\centering
						\begin{overpic}{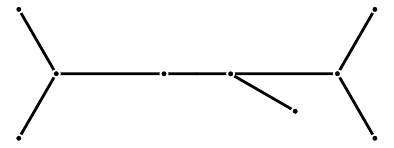}   
							\put (48,21) {\scalebox{0.8}{$z_1$}}
							\put (65,21) {\scalebox{0.8}{$z_2$}} 
							\put (28,21) {\scalebox{0.8}{$z_3$}} 
							
							\put (82,7) {\scalebox{0.8}{$z_6$}} 
							\put (82,25) {\scalebox{0.8}{$z_7$}}
							\put (12,7) {\scalebox{0.8}{$z_8$}} 
							\put (12,25) {\scalebox{0.8}{$z_9$}}

							\put (64,7) {\scalebox{0.8}{$z_5$}} 
						\end{overpic} 
					\end{minipage}\hspace{1em}
					\begin{minipage}{.32\textwidth} 
						\centering
						\begin{overpic}{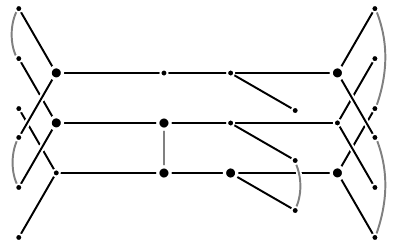}    
							\put (15,33) {\scalebox{0.8}{$A$}}
							\put (39,34) {\scalebox{0.7}{$B$}}
							\put (55,11) {\scalebox{0.7}{$C$}}
							\put (80,13) {\scalebox{0.7}{$D$}}
							\put (80,46) {\scalebox{0.7}{$E$}}
							\put (15,46) {\scalebox{0.7}{$F$}}
						\end{overpic} 
					\end{minipage}\hspace{1em}
					\begin{minipage}{.25\textwidth} 
						\centering 
						$\left(\begin{smallmatrix}
						0 & 0 & 1 &\      &   &\     &   &   \\
						0 & 1 & 0 &\      &   &\     &   &   \\
						1 & 1 & 1 &\      &   &\     &   &   \\
						1 &   &   &\      &   &\     &   &   \\  
						1 &   &   &\    2 &   &\     &   &   \\ 
						&   &   &\      & 2 &\     &   &   \\  
						&   &   &\      &   &\   2 &   &   \\ 
						&   &   &\      &   &\     & 2 &   \\
						&   &   &\      &   &\     &   & 2 \\
						\end{smallmatrix} \right)$ 
					\end{minipage}
					\\
					\noindent \begin{minipage}[t]{.32\textwidth}
						\centering
						$\Aq {T_4} b$			
					\end{minipage}\hspace{1em}
					\begin{minipage}[t]{.32\textwidth} 
						\centering
						$M_4$			
					\end{minipage}\hspace{1em}
					\begin{minipage}[t]{.25\textwidth} 
						\centering 	
						$A_4$ \\
					\end{minipage}
					\captionof{figure}{} \label{figure-14}
				\end{minipage}
				\vspace{0.5em }
			\end{itemize}
			
			\item Fifth case:
			
			\begin{itemize}[leftmargin=*] 
				\item
				Case 5.a: if $y_1 + y_2 + y_4 < y_3 < y_1 + 2y_2 + y_4$, then $M_5$ in Figure~\ref{figure-15} realizes $\mH$.

				\vspace{0.5em}
				
				\noindent	
				\begin{minipage}{\textwidth}
					\begin{minipage}{.32\textwidth}
						\centering
						\begin{overpic}{\figsdir/056.pdf}    
							\put (46,21) {\scalebox{0.8}{$z_1$}}
							\put (76,21) {\scalebox{0.8}{$z_2$}} 
							\put (25,21) {\scalebox{0.8}{$z_3$}} 
							\put (62,21) {\scalebox{0.8}{$z_4$}} 
							\put (94,7) {\scalebox{0.8}{$z_6$}} 
							\put (94,25) {\scalebox{0.8}{$z_7$}}
							\put (-1,7) {\scalebox{0.8}{$z_8$}} 
							\put (1,25) {\scalebox{0.8}{$z_9$}}
							
							\put (31,8) {\scalebox{0.8}{$z_5$}}  
						\end{overpic} 
					\end{minipage}\hspace{1em}
					\begin{minipage}{.32\textwidth} 
						\centering
						\begin{overpic}{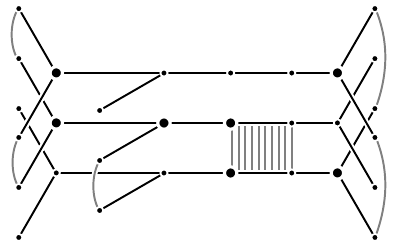}    
							\put (15,33) {\scalebox{0.8}{$A$}}
							\put (40,25) {\scalebox{0.7}{$B$}}
							\put (56,11) {\scalebox{0.7}{$C$}}
							\put (80,13) {\scalebox{0.7}{$D$}}
							\put (80,46) {\scalebox{0.7}{$E$}}
							\put (15,46) {\scalebox{0.7}{$F$}} 
						\end{overpic} 
					\end{minipage}\hspace{1em}
					\begin{minipage}{.25\textwidth} 
						\centering 
						$\left(\begin{smallmatrix}
						0 & 0 & 1 &\   0   &   &   &\     &   &   \\ 
						0 & 1 & 0 &\   1/2 &   &   &\     &   &   \\ 
						1 & 1 & 1 &\   1   &   &   &\     &   &   \\
						1 &   &   &\       &   &   &\     &   &   \\  
						1 &   &   &\       & 2 &   &\     &   &   \\ 
						&   &   &\       &   & 2 &\     &   &   \\  
						&   &   &\       &   &   &\   2 &   &   \\ 
						&   &   &\       &   &   &\     & 2 &   \\
						&   &   &\       &   &   &\     &   & 2 \\
						\end{smallmatrix} \right)$
					\end{minipage}
					\\
					\noindent \begin{minipage}[t]{.32\textwidth}
						\centering
						$\Aq {T_5} a$			
					\end{minipage}\hspace{1em}
					\begin{minipage}[t]{.32\textwidth} 
						\centering
						$M_5$			
					\end{minipage}\hspace{1em}
					\begin{minipage}[t]{.25\textwidth} 
						\centering 	
						$\det A_5 = 16$
					\end{minipage}
					\captionof{figure}{} \label{figure-15}
				\end{minipage}
				\vspace{0.5em}
				
				\item	Case 5.b: if $y_1 + y_2 + y_4 < y_3 < 2y_1 + y_2 + y_4$, then $M_5$ in Figure~\ref{figure-16} realizes $\mH$.
				
				\vspace{0.5em}
				
				\noindent	
				\begin{minipage}{\textwidth}
					\begin{minipage}{.32\textwidth}
						\centering
						\begin{overpic}{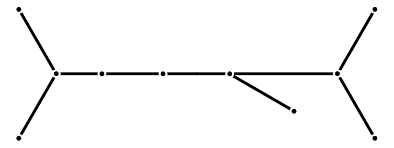}    
							\put (48,21) {\scalebox{0.8}{$z_1$}}
							\put (70,21) {\scalebox{0.8}{$z_2$}} 
							\put (18,21) {\scalebox{0.8}{$z_3$}} 
							\put (32,21) {\scalebox{0.8}{$z_4$}} 
							\put (62,9) {\scalebox{0.8}{$z_5$}} 
							\put (93,7) {\scalebox{0.8}{$z_6$}} 
							\put (92,25) {\scalebox{0.8}{$z_7$}}
							\put (-1,7) {\scalebox{0.8}{$z_8$}} 
							\put (2,25) {\scalebox{0.8}{$z_9$}} 
						\end{overpic} 
					\end{minipage}\hspace{1em}
					\begin{minipage}{.32\textwidth} 
						\centering
						\begin{overpic}{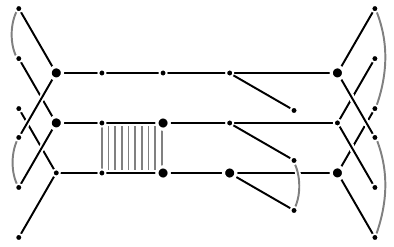}    
							\put (15,33) {\scalebox{0.8}{$A$}}
							\put (39,34) {\scalebox{0.7}{$B$}}
							\put (55,11) {\scalebox{0.7}{$C$}}
							\put (80,13) {\scalebox{0.7}{$D$}}
							\put (80,46) {\scalebox{0.7}{$E$}}
							\put (15,46) {\scalebox{0.7}{$F$}} 
						\end{overpic} 
					\end{minipage}\hspace{1em}
					\begin{minipage}{.25\textwidth} 
						\centering 
						$\left(\begin{smallmatrix}
						0 & 0 & 1 &\ 1/2 &   &   &\     &   &   \\ 
						0 & 1 & 0 &\   0 &   &   &\     &   &   \\ 
						1 & 1 & 1 &\   1 &   &   &\     &   &   \\
						1 &   &   &\     &   &   &\     &   &   \\  
						1 &   &   &\     & 2 &   &\     &   &   \\ 
						&   &   &\     &   & 2 &\     &   &   \\  
						&   &   &\     &   &   &\   2 &   &   \\ 
						&   &   &\     &   &   &\     & 2 &   \\
						&   &   &\     &   &   &\     &   & 2 \\
						\end{smallmatrix} \right)$ 
					\end{minipage}
					\\
					\noindent \begin{minipage}[t]{.32\textwidth}
						\centering
						$\Aq {T_5} b$			
					\end{minipage}\hspace{1em}
					\begin{minipage}[t]{.32\textwidth} 
						\centering
						$M_5$			
					\end{minipage}\hspace{1em}
					\begin{minipage}[t]{.25\textwidth} 
						\centering 	
						$A_5$ \\
					\end{minipage}
					\captionof{figure}{} \label{figure-16}
				\end{minipage}
				\vspace{0.5em}
				
			\end{itemize}
			
			\item Sixth case:
			
			\begin{itemize}[leftmargin=*] 
				\item
				
				Case 6.a: if $y_3 = y_1 + 2y_2 + y_4$, then $M_6$ in Figure~\ref{figure-17} realizes $\mH$. The edge $t_2$ was contracted.
				
				\noindent	
				\begin{minipage}{\textwidth}
					\begin{minipage}{.32\textwidth}
						\centering
						\begin{overpic}{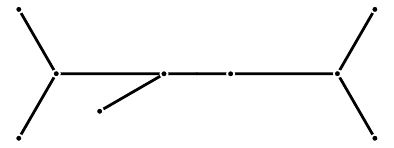}    
							\put (48,21) {\scalebox{0.8}{$z_1$}} 
							\put (23,21) {\scalebox{0.8}{$z_3$}} 
							\put (68,21) {\scalebox{0.8}{$z_4$}} 
							\put (34,7) {\scalebox{0.8}{$z_5$}} 
							\put (94,7) {\scalebox{0.8}{$z_6$}} 
							\put (94,25) {\scalebox{0.8}{$z_7$}}
							\put (-1,7) {\scalebox{0.8}{$z_8$}} 
							\put (2,25) {\scalebox{0.8}{$z_9$}} 
						\end{overpic} 
					\end{minipage}\hspace{1em}
					\begin{minipage}{.32\textwidth} 
						\centering
						\begin{overpic}{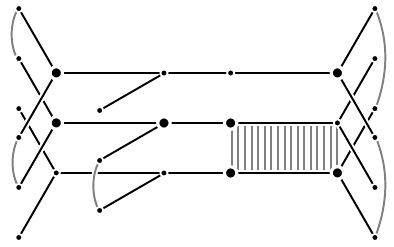}    
							\put (15,33) {\scalebox{0.8}{$A$}}
							\put (40,25) {\scalebox{0.7}{$B$}}
							\put (56,11) {\scalebox{0.7}{$C$}}
							\put (80,13) {\scalebox{0.7}{$D$}}
							\put (80,46) {\scalebox{0.7}{$E$}}
							\put (15,46) {\scalebox{0.7}{$F$}} 
						\end{overpic} 
					\end{minipage}\hspace{1em}
					\begin{minipage}{.25\textwidth} 
						\centering 
						$\left(\begin{smallmatrix}
						0  & 1 &\ 0   &   &   &\     &   &   \\ 
						0  & 0 &\ 1/2 &   &   &\     &   &   \\ 
						1  & 1 &\   1 &   &   &\     &   &   \\
						1  &   &\     &   &   &\     &   &   \\  
						1  &   &\     & 2 &   &\     &   &   \\ 
						&   &\     &   & 2 &\     &   &   \\  
						&   &\     &   &   &\   2 &   &   \\ 
						&   &\     &   &   &\     & 2 &   \\
						&   &\     &   &   &\     &   & 2 \\
						\end{smallmatrix} \right)$ 
					\end{minipage}
					\\
					\noindent \begin{minipage}[t]{.32\textwidth}
						\centering
						$\Aq {T_6} a$			
					\end{minipage}\hspace{1em}
					\begin{minipage}[t]{.32\textwidth} 
						\centering
						$M_6$			
					\end{minipage}\hspace{1em}
					\begin{minipage}[t]{.25\textwidth} 
						\centering 	
						$A_6$ 
					\end{minipage}
					\captionof{figure}{} \label{figure-17}
				\end{minipage}
				\vspace{0.5em}
				
				\item Case 6.b: if $y_3 = 2y_1 + y_2 + y_4$, then $M_6$ in Figure~\ref{figure-18} realizes $\mH$. The edge $t_3$ was contracted.
				
				\noindent	
				\begin{minipage}{\textwidth}
					\begin{minipage}{.32\textwidth}
						\centering
						\begin{overpic}{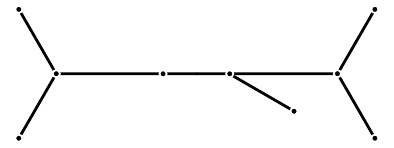}    
							\put (48,21) {\scalebox{0.8}{$z_1$}}
							\put (70,21) {\scalebox{0.8}{$z_2$}} 
							\put (26,21) {\scalebox{0.8}{$z_4$}} 
							\put (64,7) {\scalebox{0.8}{$z_5$}} 
							\put (94,7) {\scalebox{0.8}{$z_6$}} 
							\put (94,25) {\scalebox{0.8}{$z_7$}}
							\put (-1,7) {\scalebox{0.8}{$z_8$}} 
							\put (2,25) {\scalebox{0.8}{$z_9$}} 
						\end{overpic} 
					\end{minipage}\hspace{1em}
					\begin{minipage}{.32\textwidth} 
						\centering
						\begin{overpic}{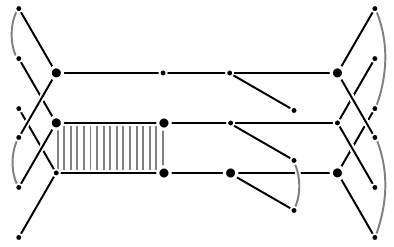}    
							\put (15,33) {\scalebox{0.8}{$A$}}
							\put (39,34) {\scalebox{0.7}{$B$}}
							\put (55,11) {\scalebox{0.7}{$C$}}
							\put (80,13) {\scalebox{0.7}{$D$}}
							\put (80,46) {\scalebox{0.7}{$E$}}
							\put (15,46) {\scalebox{0.7}{$F$}}  
						\end{overpic} 
					\end{minipage}\hspace{1em}
					\begin{minipage}{.25\textwidth} 
						\centering 
						$\left(\begin{smallmatrix}
						0 & 0 &\  1/2 &   &   &\     &   &   \\
						0 & 1 &\    0 &   &   &\     &   &   \\
						1 & 1 &\    1 &   &   &\     &   &   \\
						1 &   &\      &   &   &\     &   &   \\  
						1 &   &\      & 2 &   &\     &   &   \\  
						&   &\      &   & 2 &\     &   &   \\ 
						&   &\      &   &   &\   2 &   &   \\ 
						&   &\      &   &   &\     & 2 &   \\
						&   &\      &   &   &\     &   & 2 \\
						\end{smallmatrix} \right)$ 
					\end{minipage}
					\\
					\noindent \begin{minipage}[t]{.32\textwidth}
						\centering
						$\Aq {T_6} b$			
					\end{minipage}\hspace{1em}
					\begin{minipage}[t]{.32\textwidth} 
						\centering
						$M_6$			
					\end{minipage}\hspace{1em}
					\begin{minipage}[t]{.25\textwidth} 
						\centering 	
						$A_6$
					\end{minipage}
					\captionof{figure}{} \label{figure-18}
				\end{minipage}
				\vspace{0.5em}
				
			\end{itemize}
			
			\item Seventh case:
			
			\begin{itemize}[leftmargin=*] 
				\item
				Case 7.a: if $y_1 + 2y_2 + y_4 < y_3 < y_1 + 2y_2 + y_4 + y_6$, then $M_7$ in Figure~\ref{figure-19} realizes $\mH$. To obtain this gluing datum from 6.a we have swapped branch copies 2 and 3 in the branch above $t_6$, and then regrown $t_2$.
				
				\vspace{0.5em}
				
				\noindent	
				\begin{minipage}{\textwidth}
					\begin{minipage}{.32\textwidth}
						\centering
						\begin{overpic}{\figsdir/056.pdf}    
							\put (46,21) {\scalebox{0.8}{$z_1$}}
							\put (76,21) {\scalebox{0.8}{$z_2$}} 
							\put (25,21) {\scalebox{0.8}{$z_3$}} 
							\put (62,21) {\scalebox{0.8}{$z_4$}} 
							\put (94,7) {\scalebox{0.8}{$z_6$}} 
							\put (94,25) {\scalebox{0.8}{$z_7$}}
							\put (-1,7) {\scalebox{0.8}{$z_8$}} 
							\put (1,25) {\scalebox{0.8}{$z_9$}}
							
							\put (31,8) {\scalebox{0.8}{$z_5$}}  
						\end{overpic} 
					\end{minipage}\hspace{1em}
					\begin{minipage}{.32\textwidth} 
						\centering
						\begin{overpic}{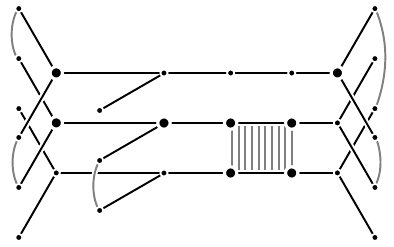}    
							\put (15,33) {\scalebox{0.8}{$A$}}
							\put (40,25) {\scalebox{0.7}{$B$}}
							\put (56,34) {\scalebox{0.7}{$C$}}
							\put (72,34) {\scalebox{0.7}{$D$}}
							\put (80,46) {\scalebox{0.7}{$E$}}
							\put (15,46) {\scalebox{0.7}{$F$}}  
						\end{overpic} 
					\end{minipage}\hspace{1em}
					\begin{minipage}{.25\textwidth} 
						\centering 
						$\left(\begin{smallmatrix}
						0 & 0 & 1 &\   0 &   &   &\     &   &   \\ 
						0 & 0 & 0 &\ 1/2 &   &   &\     &   &   \\ 
						1 & 1 & 1 &\   1 &   &   &\     &   &   \\
						1 &   &   &\     &   &   &\     &   &   \\  
						1 &   &   &\     & 2 &   &\     &   &   \\ 
						& 1 &   &\     &   & 2 &\     &   &   \\  
						& 1 &   &\     &   &   &\   2 &   &   \\ 
						&   &   &\     &   &   &\     & 2 &   \\
						&   &   &\     &   &   &\     &   & 2 \\
						\end{smallmatrix} \right)$ 
					\end{minipage}
					\\
					\noindent \begin{minipage}[t]{.32\textwidth}
						\centering
						$\Aq {T_7} a$			
					\end{minipage}\hspace{1em}
					\begin{minipage}[t]{.32\textwidth} 
						\centering
						$M_7$			
					\end{minipage}\hspace{1em}
					\begin{minipage}[t]{.25\textwidth} 
						\centering 	
						$\det A_7 = -16$
					\end{minipage}
					\captionof{figure}{} \label{figure-19}
				\end{minipage}
				\vspace{0.5em}

				\item Case 7.b: if $2y_1 + y_2 + y_4 < y_3 < 2y_1 + y_2 + y_4 + y_8$, then $M_7$ in Figure~\ref{figure-20} realizes $\mH$. This was obtained from 6.b by swapping branch copies 2 and 3 in the branch above $t_9$, and then regrowing $t_2$. 
				
				\noindent	
				\begin{minipage}{\textwidth}
					\begin{minipage}{.32\textwidth}
						\centering
						\begin{overpic}{\figsdir/058.pdf}    
							\put (48,21) {\scalebox{0.8}{$z_1$}}
							\put (70,21) {\scalebox{0.8}{$z_2$}} 
							\put (18,21) {\scalebox{0.8}{$z_3$}} 
							\put (32,21) {\scalebox{0.8}{$z_4$}} 
							\put (62,9) {\scalebox{0.8}{$z_5$}} 
							\put (93,7) {\scalebox{0.8}{$z_6$}} 
							\put (92,25) {\scalebox{0.8}{$z_7$}}
							\put (-1,7) {\scalebox{0.8}{$z_8$}} 
							\put (2,25) {\scalebox{0.8}{$z_9$}}   
						\end{overpic} 
					\end{minipage}\hspace{1em}
					\begin{minipage}{.32\textwidth} 
						\centering
						\begin{overpic}{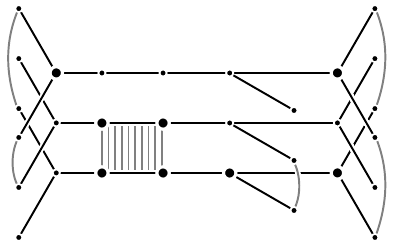}    
							\put (23,34) {\scalebox{0.8}{$A$}}
							\put (39,34) {\scalebox{0.7}{$B$}}
							\put (55,11) {\scalebox{0.7}{$C$}}
							\put (80,13) {\scalebox{0.7}{$D$}}
							\put (80,46) {\scalebox{0.7}{$E$}}
							\put (15,46) {\scalebox{0.7}{$F$}}  
						\end{overpic} 
					\end{minipage}\hspace{1em}
					\begin{minipage}{.25\textwidth} 
						\centering 
						$\left(\begin{smallmatrix}
						0 & 0 & 0 &\ 1/2 &   &   &\     &   &   \\
						0 & 1 & 0 &\   0 &   &   &\     &   &   \\
						1 & 1 & 1 &\   1 &   &   &\     &   &   \\
						1 &   &   &\     &   &   &\     &   &   \\  
						1 &   &   &\     & 2 &   &\     &   &   \\  
						&   &   &\     &   & 2 &\     &   &   \\ 
						&   &   &\     &   &   &\   2 &   &   \\ 
						&   & 1 &\     &   &   &\     & 2 &   \\
						&   & 1 &\     &   &   &\     &   & 2 \\
						\end{smallmatrix} \right)$ 
					\end{minipage}
					\\
					\noindent \begin{minipage}[t]{.32\textwidth}
						\centering
						$\Aq {T_7} b$			
					\end{minipage}\hspace{1em}
					\begin{minipage}[t]{.32\textwidth} 
						\centering
						$M_7$			
					\end{minipage}\hspace{1em}
					\begin{minipage}[t]{.25\textwidth} 
						\centering 	
						$A_7$ 
					\end{minipage}
					\captionof{figure}{} \label{figure-20}
				\end{minipage}
				\vspace{0.5em}
			\end{itemize}

			\item Eighth case:
			
			\begin{itemize}[leftmargin=*] 
				\item
				Case 8.a: if $y_3 = y_1 + 2y_2 + y_4 + y_6$, then $M_8$ in Figure~\ref{figure-21} realizes $\mH$. The edge $t_6$ was contracted.
				
				\noindent	
				\begin{minipage}{\textwidth}
					\begin{minipage}{.32\textwidth}
						\centering
						\begin{overpic}{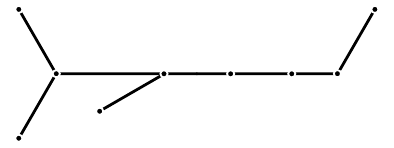}    
							\put (46,21) {\scalebox{0.8}{$z_1$}}
							\put (76,21) {\scalebox{0.8}{$z_2$}} 
							\put (25,21) {\scalebox{0.8}{$z_3$}} 
							\put (62,21) {\scalebox{0.8}{$z_4$}}  
							\put (92,25) {\scalebox{0.8}{$z_7$}}
							\put (-1,7) {\scalebox{0.8}{$z_8$}} 
							\put (1,25) {\scalebox{0.8}{$z_9$}}
							
							\put (31,8) {\scalebox{0.8}{$z_5$}}     
						\end{overpic} 
					\end{minipage}\hspace{1em}
					\begin{minipage}{.32\textwidth} 
						\centering
						\begin{overpic}{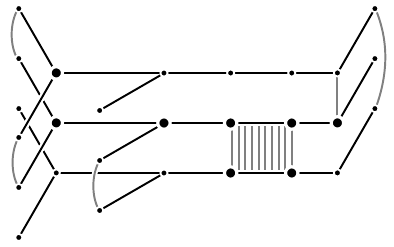}    
							\put (15,33) {\scalebox{0.8}{$A$}}
							\put (40,25) {\scalebox{0.7}{$B$}}
							\put (56,34) {\scalebox{0.7}{$C$}}
							\put (72,34) {\scalebox{0.7}{$D$}}
							\put (81,47) {\scalebox{0.7}{$E$}}
							\put (15,46) {\scalebox{0.7}{$F$}}  
						\end{overpic} 
					\end{minipage}\hspace{1em}
					\begin{minipage}{.25\textwidth} 
						\centering 
						$\left(\begin{smallmatrix}
						0 & 0 & 1 &\   0 &   &\      &   &   \\ 
						0 & 0 & 0 &\ 1/2 &   &\      &   &   \\ 
						1 & 1 & 1 &\   1 &   &\      &   &   \\
						1 &   &   &\     &   &\      &   &   \\  
						1 &   &   &\     & 2 &\      &   &   \\ 
						& 1 &   &\     &   &\      &   &   \\  
						& 1 &   &\     &   &\    2 &   &   \\ 
						&   &   &\     &   &\      & 2 &   \\
						&   &   &\     &   &\      &   & 2 \\
						\end{smallmatrix} \right)$ 
					\end{minipage}
					\\
					\noindent \begin{minipage}[t]{.32\textwidth}
						\centering
						$\Aq {T_8} a$			
					\end{minipage}\hspace{1em}
					\begin{minipage}[t]{.32\textwidth} 
						\centering
						$M_8$			
					\end{minipage}\hspace{1em}
					\begin{minipage}[t]{.25\textwidth} 
						\centering 	
						$A_8$
					\end{minipage}
					\captionof{figure}{} \label{figure-21}
				\end{minipage}
				\vspace{0.5em}
				
				\item
				Case 8.b: if $y_3 = 2y_1 + y_2 + y_4 + y_8$, then $M_8$ in Figure~\ref{figure-22} realizes $\mH$. The edge $t_8$ was contracted.
				
				\noindent	
				\begin{minipage}{\textwidth}
					\begin{minipage}{.32\textwidth}
						\centering
						\begin{overpic}{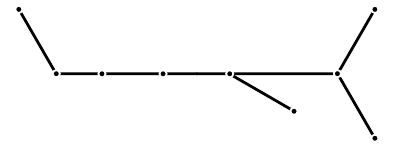}    
							\put (48,21) {\scalebox{0.8}{$z_1$}}
							\put (70,21) {\scalebox{0.8}{$z_2$}} 
							\put (18,21) {\scalebox{0.8}{$z_3$}} 
							\put (32,21) {\scalebox{0.8}{$z_4$}} 
							\put (62,9) {\scalebox{0.8}{$z_5$}} 
							\put (93,7) {\scalebox{0.8}{$z_6$}} 
							\put (92,25) {\scalebox{0.8}{$z_7$}}
							\put (2,25) {\scalebox{0.8}{$z_9$}} 
						\end{overpic} 
					\end{minipage}\hspace{1em}
					\begin{minipage}{.32\textwidth} 
						\centering
						\begin{overpic}{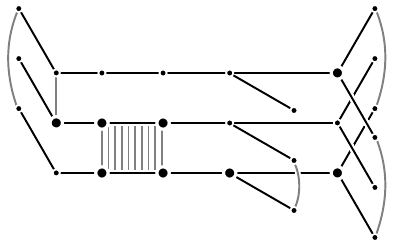}    
							\put (23,34) {\scalebox{0.8}{$A$}}
							\put (39,34) {\scalebox{0.7}{$B$}}
							\put (55,11) {\scalebox{0.7}{$C$}}
							\put (80,13) {\scalebox{0.7}{$D$}}
							\put (80,46) {\scalebox{0.7}{$E$}}
							\put (15,46) {\scalebox{0.7}{$F$}}   
						\end{overpic} 
					\end{minipage}\hspace{1em}
					\begin{minipage}{.25\textwidth} 
						\centering 
						$\left(\begin{smallmatrix}
						0 & 0 & 0 &\ 1/2 &   &   &\     &   \\
						0 & 1 & 0 &\   0 &   &   &\     &   \\
						1 & 1 & 1 &\   1 &   &   &\     &   \\
						1 &   &   &\     &   &   &\     &   \\  
						1 &   &   &\     & 2 &   &\     &   \\  
						&   &   &\     &   & 2 &\     &   \\ 
						&   &   &\     &   &   &\   2 &   \\ 
						&   & 1 &\     &   &   &\     &   \\
						&   & 1 &\     &   &   &\     & 2 \\
						\end{smallmatrix} \right)$ 
					\end{minipage}
					\\
					\noindent \begin{minipage}[t]{.32\textwidth}
						\centering
						$\Aq {T_8} b$			
					\end{minipage}\hspace{1em}
					\begin{minipage}[t]{.32\textwidth} 
						\centering
						$M_8$			
					\end{minipage}\hspace{1em}
					\begin{minipage}[t]{.25\textwidth} 
						\centering 	
						$A_8$ 
					\end{minipage}
					\captionof{figure}{} \label{figure-22}
				\end{minipage}
				\vspace{0.5em}
			\end{itemize}

			\item Ninth case:
			
			\begin{itemize}[leftmargin=*] 
				\item
				Case 9.a: if $y_1 + 2y_2 + y_4 + y_6 < y_3$, then $M_9$ in Figure~\ref{figure-23} realizes $\mH$.
				
				\vspace{0.5em}
				
				\noindent	
				\begin{minipage}{\textwidth}
					\begin{minipage}{.32\textwidth}
						\centering
						\begin{overpic}{\figsdir/056.pdf}    
							\put (46,21) {\scalebox{0.8}{$z_1$}}
							\put (76,21) {\scalebox{0.8}{$z_2$}} 
							\put (25,21) {\scalebox{0.8}{$z_3$}} 
							\put (62,21) {\scalebox{0.8}{$z_4$}} 
							\put (94,7) {\scalebox{0.8}{$z_6$}} 
							\put (94,25) {\scalebox{0.8}{$z_7$}}
							\put (-1,7) {\scalebox{0.8}{$z_8$}} 
							\put (1,25) {\scalebox{0.8}{$z_9$}}
							
							\put (31,8) {\scalebox{0.8}{$z_5$}}      
						\end{overpic} 
					\end{minipage}\hspace{1em}
					\begin{minipage}{.32\textwidth} 
						\centering
						\begin{overpic}{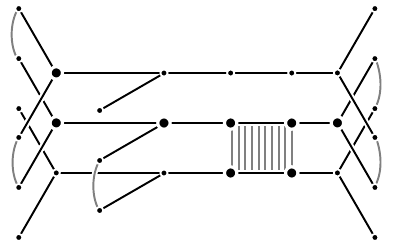}    
							\put (15,33) {\scalebox{0.8}{$A$}}
							\put (40,25) {\scalebox{0.7}{$B$}}
							\put (56,34) {\scalebox{0.7}{$C$}}
							\put (72,34) {\scalebox{0.7}{$D$}}
							\put (80,26) {\scalebox{0.7}{$E$}}
							\put (15,46) {\scalebox{0.7}{$F$}}   
						\end{overpic} 
					\end{minipage}\hspace{1em}
					\begin{minipage}{.25\textwidth} 
						\centering 
						$\left(\begin{smallmatrix}
						0 & 0 & 1 &\   0 &   &   &\     &   &   \\
						0 & 0 & 0 &\ 1/2 &   &   &\     &   &   \\
						1 & 1 & 1 &\   1 &   & 2 &\     &   &   \\
						1 &   &   &\     &   &   &\     &   &   \\  
						1 &   &   &\     & 2 &   &\     &   &   \\  
						& 1 &   &\     &   &   &\     &   &   \\ 
						& 1 &   &\     &   &   &\   2 &   &   \\ 
						&   &   &\     &   &   &\     & 2 &   \\
						&   &   &\     &   &   &\     &   & 2 \\
						\end{smallmatrix} \right)$ 
					\end{minipage}
					\\
					\noindent \begin{minipage}[t]{.32\textwidth}
						\centering
						$\Aq {T_9} a$			
					\end{minipage}\hspace{1em}
					\begin{minipage}[t]{.32\textwidth} 
						\centering
						$M_9$			
					\end{minipage}\hspace{1em}
					\begin{minipage}[t]{.25\textwidth} 
						\centering 	
						$\det A_9 = 16$
					\end{minipage}
					\captionof{figure}{} \label{figure-23}
				\end{minipage}
				\vspace{0.5em}
				
				\item	Case 9.b: if $2y_1 + y_2 + y_4 + y_8 < y_3$, then $M_9$ in Figure~\ref{figure-24} realizes $\mH$.
				
				\vspace{0.5em}
				
				\noindent	
				\begin{minipage}{\textwidth}
					\begin{minipage}{.32\textwidth}
						\centering
						\begin{overpic}{\figsdir/058.pdf}    
							\put (48,21) {\scalebox{0.8}{$z_1$}}
							\put (70,21) {\scalebox{0.8}{$z_2$}} 
							\put (18,21) {\scalebox{0.8}{$z_3$}} 
							\put (32,21) {\scalebox{0.8}{$z_4$}} 
							\put (62,9) {\scalebox{0.8}{$z_5$}} 
							\put (93,7) {\scalebox{0.8}{$z_6$}} 
							\put (92,25) {\scalebox{0.8}{$z_7$}}
							\put (-1,7) {\scalebox{0.8}{$z_8$}} 
							\put (2,25) {\scalebox{0.8}{$z_9$}} 
						\end{overpic} 
					\end{minipage}\hspace{1em}
					\begin{minipage}{.32\textwidth} 
						\centering
						\begin{overpic}{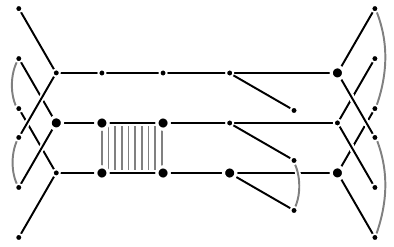}    
							\put (23,34) {\scalebox{0.8}{$A$}}
							\put (39,34) {\scalebox{0.7}{$B$}}
							\put (55,11) {\scalebox{0.7}{$C$}}
							\put (80,13) {\scalebox{0.7}{$D$}}
							\put (80,46) {\scalebox{0.7}{$E$}}
							\put (15,26) {\scalebox{0.7}{$F$}}   
						\end{overpic} 
					\end{minipage}\hspace{1em}
					\begin{minipage}{.25\textwidth} 
						\centering 
						$\left(\begin{smallmatrix}
						0 & 0 & 0 &\ 1/2 &   &   &\     &   &   \\
						0 & 1 & 0 &\   0 &   &   &\     &   &  \\
						1 & 1 & 1 &\   1 &   &   &\     & 2 &   \\
						1 &   &   &\     &   &   &\     &   &   \\  
						1 &   &   &\     & 2 &   &\     &   &   \\  
						&   &   &\     &   & 2 &\     &   &   \\ 
						&   &   &\     &   &   &\   2 &   &   \\ 
						&   & 1 &\     &   &   &\     &   &   \\
						&   & 1 &\     &   &   &\     &   & 2 \\
						\end{smallmatrix} \right)$ 
					\end{minipage}
					\\
					\noindent \begin{minipage}[t]{.32\textwidth}
						\centering
						$\Aq {T_9} b$			
					\end{minipage}\hspace{1em}
					\begin{minipage}[t]{.32\textwidth} 
						\centering
						$M_9$			
					\end{minipage}\hspace{1em}
					\begin{minipage}[t]{.25\textwidth} 
						\centering 	
						$A_9$ 
					\end{minipage}
					\captionof{figure}{} \label{figure-24}
				\end{minipage}
				\vspace{0.5em}
				
			\end{itemize}
			
		\end{itemize}
		
	\end{ex}

	
	
	
	\section{Deformation} \label{sec-limits}
	We visualize Example~\ref{ex-three-cycles-loop} as a \emph{movie} featuring a continuous deformation of the length function~$\ell_H$ (for an actual movie go here\footnote{https://mathsites.unibe.ch/jdraisma/MovieGenus4.mp4}). 
	The deformation path grows the length $y_3$ while leaving fixed the remaining lengths. The path moves in and through nine different
	cones of sources $C_M$. We use the full-rank property of $A_M$ to calculate $\ell_T = \inv A_M \ell_H$, which defines the tropical morphism $(M,\ell_T)$. As we grow the length~$y_3$, some lengths of $\ell_T$ shrink down to zero. We call the cases where a length of $\ell_T$ is zero a \emph{limit} in the deformation. Since we do not allow zero lengths in an edge-length function, we contract an edge of $T$ instead. 
	
	If we were to walk further past a limit while using the same gluing datum we would get negative entries in~$\inv A_M \ell_H$, which is not allowed in a length function. In order to go beyond we pass to a different gluing datum. In Example~\ref{ex-three-cycles-loop} the even cases are \emph{limit gluing datums}, as an edge is contracted. Around these limits there is a change in the gluing relations. In this section we show that understanding this change of gluing relations sets up a deformation procedure and gives a proof of Theorem~\ref{theorem-gonality}.

	\subsection{Limits} \label{sec-results-on-limits}
	We set up a framework to study limits. 
	Let $M = (T, d, \sim)$ be a gluing datum, and $t$ an edge of $T$.
	
	\begin{de}[Limit gluing datum] \label{def-limit-gluing}  The \emph{limit} of $M$ by contracting $t$, in short \emph{the limit at $t$}, is the gluing datum $M_0 = (T_0, d, \sim_0)$ given by the following data: 
		\begin{itemize} 
			\item $T_0$ is obtained by contracting $t$ in $T$. The ends $u, v$ of $t$ get identified with a vertex $w_0$ of $T_0$. Edges and vertices of $T$ different from $u, v$, and~$t$ correspond in a natural manner with edges and vertices of $T_0$.  
			\item $\sim_{0,{w_0}}$ equals the finest common coarsening of $\sim_u$ and $\sim_v$; and $\sim_{0,x}$ equals $\sim_x$ for $x \ne w_0$.
		\end{itemize}
		
	\end{de}
	
	For the remainder of this section, let $M_0 = (T_0, d, \sim_0)$ be the limit at $t$ of~$M$, giving rise to $G_0$, and $u$, $v$ be the ends of $t$ that contract to $w_0$ in~$V(T_0)$.
	Fix an edge $t$ of $T$. Let $\dtmor_0 \!:\! G_0 \to T_0$ be the \DTmor that arises from $\dtmor_M \!:\! G \to T$ by contracting $t$ in $T$ to obtain $T_0$, and contracting all the edges in $\inv \dtmor(t)$ to obtain~$G_0$. It is straightforward to see that $\dtmor_0$ is canonically isomorphic to $\dtmor_{M_0}$, where $M_0$ is the limit of $M$ at $t$. Namely, the classes representing vertices and edges of $G$ \emph{contract} to the classes representing vertices and edges of $G_0$.
	Note that a class $\class x i$ of $G$ contracts to $\class {x_0} j$ of $G_0$ if and only if $\class x i \subseteq \class {x_0} j$  and $x$ contracts to $x_0$.
	
	To prove that $M_0$ is indeed a gluing datum, we observe that verticality, refinement and connectedness are inherited.  We write $r_0$ for $r_{\dtmor_0}$. The RH-condition amounts to proving that $r_0(A_0) \ge 0$ for all $A_0$ in $V(G_0)$. This is true because	 $r_0$ is additive under contraction: 
	
	\begin{prop}[$r_\varphi$ under contraction] \label{prop-rphi-under-contraction}
		Let $A_0$ be in $\inv \dtmor_0(w_0)$, and $A_1, \dots, A_r$ be the vertices of $G$ that contract to~$A_0$. If $g(H(M)) = g(H(M_0))$, then: 
		\begin{align*}
		r_{0}(A_0) &= r_\varphi(A_1) + \dots + r_\varphi(A_r), \\
		\ch_{M_0} w_0 &= \ch_M u + \ch_M v. 
		\end{align*}
	\end{prop}
	\begin{proof} 
		Let $G_A$ be the subgraph defined by the vertices and edges of $G$ that contract to~$A_0$. Namely, $V(G_A) = \{A_1, \dots, A_r\}$. It is a forest; otherwise we would have $g(H(M_0)) < g(H(M))$; and also connected, hence a tree. Restrict $\dtmor$ to $G_A$ to get a map $\dtmor_A$ onto the graph on $\{u, v\}$, joined by the single edge~$t$. This map is a \DTmor of degree~$|A_0|$. 
		By Lemma~\ref{lemma-total-change} the total change of $\dtmor_A$ is $\sum_{B \in V(G_A)} r_{\dtmor_A} (B) = 2|A_0| - 2$. Using Lemma~\ref{lemma-formula-rphi} we calculate:
		\begin{align*}
		\sum_{i=1}^r r_\varphi(A_i) &= \sum_{i = 1}^r (2|A_i| - 2)
		- \sum_{i=1}^r \sum_{\substack{e \in E(A_i) \\ \dtmor(e) \ne t}} (|e| - 1) 
		- 2 \cdot \sum_{\substack{e \in E(G_A)}} (|e| - 1)\\
		&= \sum_{B \in V(G_A)} r_{\dtmor_A} (B)
		- \sum_{i=1}^r \sum_{\substack{e \in E(A_i) \\ \dtmor(e) \ne t}} (|e| - 1) \\
		&= 2|A_0| - 2 - \sum_{\substack{e \in E(A_0)}} (|e| - 1) = r_{0}(A).
		\end{align*}
		Apply this formula over all vertices above $w_0$ to get the second result.
	\end{proof}
	\subsection{Inherited properties} \label{section-inherited-properties} 
	Now we study which properties of Definition~\ref{def-conditions} are inherited by limits. Let $M_0$, $G_0$, $u$, $v$, $w_0$, $\dtmor_0$ and $r_0$ be as in Subsection~\ref{sec-results-on-limits}.
	
	\begin{lm}[dangling in the limit] \label{lemma-dangling-limit} 
		If $g(H(M)) = g(H(M_0))$, then a vertex $A_0$ in $\inv \dtmor_0(w_0)$ is dangling if and only if all the $A_1, \dots, A_r$ in $V(G)$ that contract to $A_0$ are dangling. 	
	\end{lm}
	
	\begin{proof}
		Since $g(H(M)) = g(H(M_0))$, no cycle of $G$ is contracted in $G_0$. This implies that $A_0$ is not contained in a cycle of $G_0$ if and only if none of $A_1, \dots, A_r$ are contained in a cycle of~$G$.
	\end{proof}

	\begin{lm} \label{lemma-limit-dangling-no-glue}
		If $g(H(M))=g(H(M_0))$ and $M$ satisfies the dangling-no-glue condition, then $M_0$ satisfies the dangling-no-glue condition.
	\end{lm}
	
	\begin{proof}
		We only need to check that a dangling vertex $A_0$ above $w_0$ satisfies $|A_0| = 1$. Since $\sim_{w_0}$ is the finest common coarsening of $\sim_u$ and~$\sim_v$, we have $|A_0| > 1$ only if for some $A$ in $V(G)$ contracting to $A_0$ we have $|A|>1$. By Lemma~\ref{lemma-dangling-limit} there is no such $A$.
	\end{proof}
	
	\begin{lm}[non-dangling union] \label{lemma-class-union} Assume that $g(H(M))=g(H(M_0))$. Let $M$ satisfy the dangling-no-glue condition, let $A_0$ in $\inv \dtmor_0(w_0)$ be non-dangling, and let $A_1, \dots, A_{r}$ in $V(G)$ be the non-dangling vertices that contract to $A$. Then $A_0 = A_1 \cup \dots \cup A_{r}$, as subsets of~$[d]$.
	\end{lm}
	
	\begin{proof}
		This is equivalent to proving that for every $k$ in $A_0$ at least one of $B_u = \class u k$, $B_v = \class v k$ in $V(G)$ is non-dangling. Assume not, so by dangling-no-glue $|B_u|=|B_v|=1$, thus $B_u$, $B_v$ are not incident to other vertices above $u$ or $v$. Hence these two dangling classes are the only ones that contract to $A_0$. By Lemma~\ref{lemma-dangling-limit} $A_0$ is dangling, a contradiction.
	\end{proof}
	
	Note that in the previous lemmas we have used the fact that the gluing relations of $M_0$ agree with those of $M$ outside of $w_0$. From this fact we also get that if $M$ has an edge labelling $\lambda = (\lambda_T, \lambda_H)$, with $\lambda_T: E(T) \to [|E(T)|]$ and $\lambda_H: E(H(M)) \to [|E(H(M))|]$ (recall that $[m]$ means $\{1,\dots,m\}$), then they induce edge labellings $\lambda_{T,0}$ and $\lambda_{H,0}$ on $M_0$. It is clear how to do this for $T_0$. For calculating $\lambda_{H,0}(h_0)$, with $h_0$ in $E(H(M_0))$, choose any $e_0 \in h_0$. As edges of $G$ not above $u$, $v$, or $t$ correspond to the edges of $G_0$, there is an edge $e$ in $G$ that corresponds to $e_0$. Recall that $h(e)$ is the unique edge of $H(M)$ containing $e$. Set $\lambda_{H,0}(h_0) = \lambda_{H}(h(e))$. It is straightforward to see that this is well defined.
	
	Going one step further, if two gluing datums share a limit, then loosely speaking  their combinatorial information agrees outside the contracted edge and its ends. We define a condition on the edge orders that allows us to relate the edge-length matrices. Recall from the proof of Lemma~\ref{lemma-iso-classes-dtmor-gd} that an isomorphism $(\tau, \pi_x)$ from $M$ to some gluing datum $M'$ gives rise to a graph isomorphism $\gamma: G \to G'$. This $\gamma$ descends to a graph isomorphism $\wgamma: H(M) \to H(M')$.
	
	\begin{de}
		Let $\Aq M q = (\Aq T q, d, \Aq \sim q)$ be gluing datums with edge labellings $\Aq {\lambda_T} q : E(\Aq T q) \to [|E(\Aq T q)|]$ and $\Aq {\lambda_H} q : E(H(\Aq M q)) \to [|E(H(\Aq M q))|]$, for $q=1,2$. Let $\Aq {M_0} q$ be the limit at $t_k = \inv{(\Aq {\lambda_T} q)}(k)$ of $\Aq M q$, with $\Aq {\lambda_{T,0}} q$ and $\Aq {\lambda_{H,0}} q$ the induced labellings. If $|E(H(M_0))| \ge |E(H(\Aq M q))| - 1$, then we say that the labellings of $\Aq M 1$ and $\Aq M 2$ are \emph{compatible at $t_k$} if there is an isomorphism $(\tau, \pi_x)$ from $\Aq M 1 _0$ to~$\Aq M 2 _0$ such that $\Aq {\lambda_{T,0}} 2 = \Aq {\lambda_{T,0}} 1 \circ \inv \tau$ and $\Aq {\lambda_{H,0}} 2 = \Aq {\lambda_{H,0}} 1 \circ \inv \wgamma$.
	\end{de}
	
	The condition $|E(H(M_0))| \ge |E(H(\Aq M q))| - 1$ ensures that in $M_0$ at most one edge of $H(\Aq M q)$ is contracted. This is needed to get unicity at the end of Remark~\ref{remark-choose-one-labelling}.
	
	\begin{lm}[limit matrix] \label{lemma-limit-matrix-change} 
		Let $\Aq M 1$, $\Aq M 2$ be gluing datums with labellings compatible at $t_k$, and edge-length matrices $(a_{ij})$, $(b_{ij})$ respectively. Then $a_{i j} = b_{ij}$ for $j \ne k$. 
	\end{lm}
	
	
	\begin{proof}
		Restrict the edge-length map of $\Aq M 1$ to the subspace $\{x_k = 0\}$ to get the edge-length map of $\Aq {M_0} 1$. This amounts to deleting the $k$-th column of $(a_{ij})$. The matrix of the edge-length map of $\Aq {M_0} 1$ coincides with that of the edge length-map of $\Aq {M_0} 2$ because of the isomorphism and compatibility of labellings. Since the matrix of the edge length map of $\Aq {M_0} 2$ is obtained by deleting the $k$-th column of $(b_{ij})$, we get the result.   
	\end{proof}
	
	\begin{lm} \label{lemma-limit-full-rank} If $M$ has full-rank, then every limit of $M$ has full-rank. 
	\end{lm}
	
	\begin{lm} \label{lemma-limit-pass-once} If $M$ satisfies pass-once, then every limit of $M$ satisfies pass-once. 
	\end{lm}
	
	\begin{proof}
		This is a consequence of Lemma~\ref{lemma-limit-matrix-change} and Proposition~\ref{prop-adm-matrix}.
	\end{proof}
	
	\begin{re} \label{remark-choose-one-labelling}
		Note that if the limit of $M$ at $t$, and the limit of $M'$ at $t'$ are isomorphic, then a labelling $\lambda$ on $M$ induces a unique labelling $\lambda'$ on $M'$ such that the two labellings are compatible at $t_k$ with $k = \lambda_T(t)$. Indeed, induce $\lambda_0$ on $M_0$ from $\lambda$; the isomorphism from $M_0$ to $M'_0$ gives a labelling $\lambda'_0$ from $\lambda_0$; and there is a unique $\lambda'$ that induces $\lambda'_0$ on $M'_0$ such that $\lambda'_{T'}(t') = \lambda_T(t)$; and if edges $h'$, $h$ of $H(M')$, $H(M)$, respectively, were contracted, then $\lambda'_{H'}(h') = \lambda_H(h)$
	\end{re}
	
	\begin{lm} \label{lemma-limit-no-return} If $M$ satisfies no-return and pass-once, then a limit $M_0$ satisfies no-return if either $w_0$ is a leaf or $\ch w_0 \le 2$. 
	\end{lm}
	
	\begin{proof}
		Suppose that $A_0$ in $V(G_0)$, with $r=\nddeg A \ge 2$, violates no-return. Since $M$ satisfies no-return, then $A_0$ is above $w_0$. If $\vale w_0 = 1$ there is nothing to check, so assume $w_0$ is not a leaf. If $r_0(A_0) \le 1$, we get a contradiction with Lemma~\ref{lem-noreturn-rphi1} (r1 implies no-return). Since $\ch w_0 \le 2$, we have that $r_0(A_0) = 2$. As $A_0$ violates no-return, all the non-dangling edges $e_1, \dots, e_r$ incident to $A_0$ are above a single edge of $T$, thus $|e_1| + \dots + |e_r| \le |A|$, so $|A| \ge 2$. Hence, $r_0(A) = 2 = r - 2 + 2|A| - (|e_1| + \dots + |e_r|) \ge r- 2 + |A|$ gives $|A|=2$ and $r = 2$. So there is an edge of $H_0$ passing above $w_0$ by going through $A_0$, and twice above~$t$. Since $w_0$ is not a leaf, $M_0$ does not satisfy pass-once. On the other hand, Lemma~\ref{lemma-limit-pass-once} implies that $M_0$ satisfies pass-once, a contradiction.
	\end{proof}
	
	The condition of having $\ch w_0 \le 2$ when $w_0$ is not a leaf is satisfied when $M$ is change-minimal and has full-rank; in particular it is satisfied by full-dimensional $M$. Indeed, Proposition~\ref{prop-rphi-under-contraction} and the equality $\vale w_0 = \vale u + \vale v - 2$ together imply that
	\begin{align} 
	\label{eq-C} \tag{C}\ch w_0 + \vale w_0 - 3 = (\ch u + \vale u - 3) + (\ch v + \vale v - 3) + 1 = 1.
	\end{align}
	This means that $\ch w_0$ is 0,1,2 or 3. The last value implies that $w_0$ is a leaf. Thus:
	
	\begin{lm} \label{lemma-limit-full-dimensional}  Let $M$ be change-minimal with full-rank, and $M_0$ be a limit of $M$ with $g(H(M)) = g(H(M_0))$. Then $M_0$ has full-rank and satisfies dangling-no-glue, no-return, and pass-once. 
	\end{lm}
	
	\begin{proof}
		This is a consequence of Lemmas~\ref{lemma-limit-dangling-no-glue}, \ref{lemma-limit-full-rank}, \ref{lemma-limit-pass-once}, \ref{lemma-limit-no-return} and Equation~\eqref{eq-C}.
	\end{proof}

	\subsection{Deformation} \label{sec-walking-through-I}
	We introduce the deformation procedure. Let $\omega$ be a deformation path in the interior of some cone $C_H$. The deformation procedure proves that upon encountering a limit while deforming a tropical morphism $\tmor$ along $\omega$, as in Example~\ref{ex-three-cycles-loop}, it is possible to continue further by changing the combinatorial structure of the tropical morphism. This means passing to another gluing datum. The majority of the proof of this fact is deferred to Section~\ref{sec-constructions}. Here we use Lemma~\ref{lemma-limit-matrix-change} to reduce the proof to the problem of constructing full-dimensional gluing datums with a given limit. 
	
	Recall that $\mathbb{FD}_g$ is the set of full-dimensional \DTmors whose sources have genus~$g$. We abuse notation and write $M$ in $\mathbb{FD}_g$ for $\dtmor_M$ in $\mathbb{FD}_g$.
	Let $M_0$ be the limit at $t_1$ of some $M$ in $\mathbb{FD}_g$.  Denote by $\cstar {M_0}$ the set of isomorphism classes of gluing datums in $\mathbb{FD}_g$ that contract to a limit in the isomorphism class of $M_0$. By Remark~\ref{remark-choose-one-labelling} a choice of an edge labelling on $M$ induces an edge labelling compatible at $t_1$ on each element of~$\cstar {M_0}$. 
	
	Since $\omega$ is contained in a cone $C_H$, namely our walk does not change the combinatorial type of the source,
	assume that $H(M_0)$ and $H(M)$ are isomorphic. Thus, $C_{M_0}$ is a subset of $C_{H(M)}$.
	Denote by $|\ell_1 - \ell_2|$ the euclidean distance in~$\RR^{E(H)}$. We wish to prove:
	
	\begin{prop}[deformation I]  \label{proposition-walking-through-I}
		Let $M_0$ be the limit at $t_1$ of some full-dimensional gluing datum $M$, with $H(M_0)$ isomorphic to $H = H(M)$. Let $y_0 \in C_{M_0} \subset C_H$. 
		There exists a constant $\varepsilon(M_0, y_0)>0$ such that if $\hy$ is an element of $C_H \setminus C_{M_0}$ with $|y_0 - \hy| < \varepsilon(M_0, y_0)$, then there is $\hM$ in $\cstar {M_0}$ with $H(\hM)$ isomorphic to $H$ and~$\hy \in C_{\hM}$.
	\end{prop}
	
	Proposition~\ref{proposition-walking-through-I} is called \emph{deformation I} because in the part II of this series of articles we enhance it to show that the number of distinct ways one can walk into a limit coincides with the number of ways one can walk out of it, under a suitable notion of multiplicity. This notion of multiplicity is suggested by the \emph{balancing condition} in Section~\ref{sec-constructions} of this article.
	
	We compute $\varepsilon(M_0, y_0)$ to have the property that if $y \in \RR^{E(H)}$ is in a ball $B$ centered at $y_0$ with radius $\varepsilon(M_0, y_0)$ and the labellings of the elements of $\cstar {M_0}$ are compatible at $t_1$, then given $\hM = (\hT, d, \hat \sim)$ in $\cstar {M_0}$ the only obstruction for $\hy \in B$ to also be in $C_{\hM}$ is that $\hy(t_1)$ is positive.
	For $\hM$ in $\cstar{M_0}$ write $\hA$ for $A_{\hM}$, and let $\hz_0 = \inv \hA(y_0) \in \RR^{E(\hT)}$. Denote by $\min^+ y$ the smallest positive entry of~$y$.  Every element of the open ball $B(\hz_0, \min^+ \hz_0) \subset \RR^{E(\hT)}$ has positive entries, except possibly the one corresponding to $t_1$. By continuity of $\inv \hA$ we can choose a number $\varepsilon(\hM, M_0, y_0) > 0$ such that the image under $\inv \hA$  of the open ball with centre $y_0$ and radius $\varepsilon(\hM, M_0, y_0) > 0$ lies in $B(\hz_0, \min^+ \hz_0)$. 
	We define
	\[\varepsilon(M_0, y_0) = \min_{\hM \in \cstar {M_0}}  \varepsilon(\hM, M_0, y_0). \] This is well defined and positive since $\mathbb{FD}_g$ is finite. By construction:
	
	\begin{lm} \label{lemma-yinM-iff-zPositive}
		Let $\hM$ be in $\cstar{M_0}$ and $\hy$ in $C_H$. Set $\hz = \inv \hA (\hy)$. If $|y_0 - \hy| < \varepsilon(M_0, \ell_0)$, then $\hy \in C_{\hM}$ if and only if $\hz(t_1) > 0$.
	\end{lm}
	
	Now we compute a formula for $\hz(t_1)$, showing that its sign depends only on  $M_0$, $\hy$, and crucially only on~$\det \hat A$. Let $\hA = (\hat a_{ij})$. Append a column of zeroes at the beginning of the edge-length matrix of~$M_0$ to obtain $A_0 = (a_{ij})$.  From Lemma \ref{lemma-limit-matrix-change} (limit matrix) it follows that $\hat a_{ij} = a_{ij}$ for all $i$ and $j \ge 2$. Hence, the first row of the adjugate matrix $\adj(\hat A)$ depends only on~$M_0$. Let $(c_1, c_2, \dots, c_{3g-3})$ be that first row.  Recall that the adjugate matrix satisfies:
	\[\adj(\hat A) \cdot \hat A = \det( \hat A) \cdot  I = \hat c \cdot I,\]
	where $\hat c = \det(\hat A)$.  This formula holds in general for gluing datums with square edge-length matrices. This gives that:
	\begin{align} 
	\label{eq-ast} \tag{$\ast$}
	\hat c &= \sum_{i = 1}^{3g-3} c_i \hat a_{i1},\\ 
	\label{eq-astast} \tag{$\ast \ast$}
	0 &= \sum_{i = 1}^{3g-3} c_i \hat a_{ij} = \sum_{i = 1}^{3g-3} c_i a_{ij}, \text{ for } j \ge 2.
	\end{align}
	
	Moreover, since $\hM$ is full-dimensional we have $\hc \ne 0$, so $\inv \hA = \frac 1 \hc \adj(\hA)$. Thus, 
	\begin{align*}  \label{eq-L} \tag{L1}
	\hz(t_1) 
	&= \dfrac 1 {\hat c} 
	\sum_{i = 1}^{3g-3} 
	c_i \, \hy(h_i).
	\end{align*}
	Hence, the sum depends only on $\hat c$, $\hy$ and on the $c_i$ (which depend on $M_0$). Write $c = \det A_M$ as well. Applying Equation~\eqref{eq-L} to $z(t_1) = (\inv A_M \hy)(t_1)$ we get that 
	\begin{align*}  
	z(t_1) 
	&= \dfrac 1 {c} 
	\sum_{i = 1}^{3g-3} 
	c_i \, \hy(h_i).
	\end{align*}
	Using Lemma~\ref{lemma-yinM-iff-zPositive} and comparing signs we obtain:

	\begin{lm} \label{lemma-opposite-sign}
		Assume that $|y_0 - \hy| < \varepsilon(M_0, \ell_0)$ and $\hy \not \in C_{M_0}$. If $\sgn \hc = \sgn c$, then $\hy \in C_{\hM}$ if and only if $\hy \in C_{M}$. If $\sgn \hc \not = \sgn c$, then $\hy \in C_{\hM}$ if and only if $\hy \not \in C_{M}$. 
	\end{lm}
	
	
	
	So we would like to construct full-dimensional gluing datums with prescribed limit at~$t_1$ and a prescribed sign for $\det A_M$.
	
	\begin{lm} \label{lemma-construct-next-cone}
		If $M$ is in $\cstar{M_0}$, the edge labelling on $\cstar{M_0}$ is compatible at $t_1$, and $H(M_0)$ is isomorphic to $H(M)$, then there is $\hM$ in $\cstar{M_0}$ with $\sgn \det A_M \not = \sgn \det A_{\hM}$.
		
	\end{lm}
	
	Our constructive proof of Lemma~\ref{lemma-construct-next-cone} is the main achievement of this article. This is carried out in Section~\ref{sec-constructions}. Proposition~\ref{proposition-walking-through-I} (deformation I) follows easily from Lemma~\ref{lemma-construct-next-cone}:
	
	\begin{proof}[Proof of Proposition~\ref{proposition-walking-through-I} (deformation I)] If $\hy \in C_M$, then we are done, just take $\hM = M$. Otherwise $\hy \not \in C_M$. By Lemma~\ref{lemma-construct-next-cone} there is a gluing datum $\hM$ in $\cstar{M_0}$ with $\sgn \hc \not = \sgn c$. By Lemma~\ref{lemma-opposite-sign} we have that~$\hy \in C_{\hM}$. 
	\end{proof}

	\subsection{Proof of main result} \label{section-proof-main-result} 
	To close this section, we put together the deformation procedure of Proposition~\ref{proposition-walking-through-I} and the initial families of Theorem~\ref{theorem-cd} to derive a proof of Theorem~\ref{theorem-gonality}.
	%
	%
	%
	%
	Our deformation machinery is designed to work with full-dimensional gluing datums, which implies even genus. The odd genus case follows from the usual trick of attaching a lollipop. For executing the trick we need a converse to Lemma~\ref{lemma-loop-12}.
	
	\begin{lm}(loop and bridge) \label{lemma-loop-bridge} 
		Let $M$ be change-minimal and full-rank. Let $A$ in $V(H(M))$ be trivalent, incident to a bridge $h_b$ and a loop $h_l$. Then $h_b$ and $h_l$ are above a path of $T$ of length 2 that leads to a leaf. Moreover, they are the only edges of $H$ above this path. 
		%
		%
	\end{lm}
	
	\begin{proof}
		Regard $A$ as a vertex of $G$. Let $v=\dtmor(A)$. Suppose that $r_\dtmor(A) = 0$. By Lemma~\ref{lemma-dangling-no-glue} we can apply Proposition~\ref{prop-local} (local properties). By the case~(r0-nd3) the three non-dangling edges incident to $A$ are above three distinct edges of $T$. Since two of the edges of $G$ incident to $A$ are in $h_l$, we have that $h_l$ passes above at least two leaves of~$T$, contradicting Remark~\ref{rem-leaves-min-change} (change-minimal leaves). The possibility $r_\dtmor(A) = 2$ is ruled out since $\nddeg A = 3$. Thus, $r_\dtmor(A) = 1$. The case (r1-nd3) implies that $v$ is divalent and there is an edge $t$ in $E(T)$ incident to $v$ such that
		there are non-dangling edges $e_1$ and $e_2$ in $\Neigh A$ above $t$. Both $e_1,e_2$ are in $h_l$, for otherwise there is a contradiction again. Hence, $h_l$ passes twice above $t$, so by pass-once $t$ leads to a leaf. The path is made by $t$ and the other edge incident to~$v$, and now we can apply Lemma~\ref{lemma-loop-12} to conclude uniqueness.
	\end{proof}
	
	Let $M = (T, d, \sim)$ be a genus-$g$ gluing datum. We add some boundary points to $C_M \subset C_{H(M)}$, in a process akin to the one done in Section~\ref{section-moduli-space-of-metric-graphs} that added some boundary points to $C_H$, giving rise to $\overline C_H$. For a given $\ell_T$ in $\RRgo^{E(T)}$ contract all the edges of $T$ for which $\ell_T(e) = 0$ to get $M_0$ (this is independent of contraction order); 
	the pair $(M, \ell_T)$ refers to the tropical morphism~$(M_{0}, \ell|_{T_0})$. Denote by $\overline C_M$ the image under $A_M$ of the points $\ell_T$ in $\RRgo^{E(T)}$ such that the source of $(M, \ell_T)$ has genus $g$; namely points where the genus does not drop. 
	
	\begin{re}
		The cone $\overline C_M \subset \overline C_{H(M)}$ is closed in the topology that $\RR^{E(H(M))}$ induces on~$\overline C_{H(M)}$.
	\end{re}
	
	Let $H$ be a genus-$g$ combinatorial type. We denote by $\FD H$ the subset of gluing datums $M$ in $\FD g$ such that $H(M)$ is isomorphic to $H$. To prove Theorem~\ref{theorem-gonality} we show that it is possible to cover $\overline C_H$ with the cones $\overline C_M$ of the $M$ in $\FD H$, by identifying $\overline C_H$ with points in $\RR_{\ge 0}^{3g-3}$. Recall that an edge labelling $(\lambda_T, \lambda_H)$ on $M$ identifies $\overline C_M$ with a cone in $\RR_{\ge0}^{3g-3}$. Moreover, recall from Subsection~\ref{section-inherited-properties} that if $M_0$ has trivalent $H(M_0)$ and is the limit at $t_1$ of some $M$ in $\FD H$, then a labelling compatible at $t_1$ has the property of embedding in $\RR_{\ge0}^{3g-3}$ the cones of two $M, \hat M$ in $\cstar {M_0}$ in such a way that the intersection is the embedding of the cone of $M_0$. Thus, embed $\FD H$ into $\RR^{3g-3}$ by choosing labellings for the elements of $\FD H$ such that for every $M_0$ that has trivalent $H(M_0)$ and is a limit at $t_1$ of some $M$ in $\FD H$, we have labellings compatible at $t_1$ for $\cstar{M_0}$.
	
	We now present the proof of Theorem~\ref{theorem-gonality}. To realize a point $y_1$ in~$\overline C_H$ we use Theorem~\ref{theorem-cd} to choose an initial point $y^*$ in~$C_H$ realizable by a tropical morphism, and draw a straight line in $\RR_{\ge 0}^{3g-3}$ between $y_1$ and $y^*$ to use as a deformation path. We then walk along this line, applying Proposition~\ref{proposition-walking-through-I} as needed.  To ensure that the hypotheses of Proposition~\ref{proposition-walking-through-I} are met, we choose $y^*$ to be in general position using the following standard argument.
	
	\begin{lm} \label{lemma-wiggle-point}
		Let $S = \bigcup_{L \in \mathcal F} L \subset \mathbb R^n$ be a countable union of affine linear spaces $L$ of codimension at least~2, and $y, y_1$ be in $\mathbb R^n$. For every $\varepsilon > 0$ there is a point $y^*$ in the radius-$\varepsilon$ ball centred at  $y$ such that the half-open line segment $[y^*, y_1)$ does not intersect~$S$. 
	\end{lm}
	
	\begin{proof}
		Let $\bar S = \bigcup_{L \in \mathcal F} \operatorname{aff-span} (y_1,L) \subset \mathbb R^n$. The codimension of  $\operatorname{aff-span} (y_1,L)$ is at least~1; as $\mathcal F$ is countable, this implies that the ball $B(y, \epsilon)$ is not contained in $\bar S$. So choose $y^*$ in $B(y, \varepsilon) \setminus \bar S$.
	\end{proof}
	
	\begin{proof}[Proof of Theorem~\ref{theorem-gonality}.] Assume that $g = 2g'$. To realize a genus-$g$ metric graph $\mG$, delete the dangling trees of $\mG$ to obtain the metric graph $\mH$, and let $H$ be a trivalent combinatorial type such that for some $y_1 \in \overline C_H$ we have that $\mH = (H, y_1)$. By Theorem~\ref{theorem-cd} we have that $\FD {H}$ is non-empty. Moreover it is a finite set. So embed $\FD {H}$ into $\overline C_H$ via labellings compatible at $t_1$.
		
		For a gluing datum $M$ in $\FD H$, denote by $\mathbb L_k M$ the set of affine spaces in $\mathbb R^{3g-3}$ associated to limits of $M$ where at least $k$ edges have been contracted. This means, images under $A_M$ of sets in $\mathbb R^ {|E(T)|}$ defined by fixing $k$ coordinates to be zero; the members of $\mathbb L_k M$ have codimension at least~$k$.
		Let \[Z_k = \bigcup_{M \in \FD H} \mathbb L_k M.\] 
		
		Since $\FD H$ is finite, the set $Z_k$ is closed in the topology that $\mathbb R^{3g-3}$ induces on $\overline C_H$. Choose $M$ in $\FD H$, and using Lemma~\ref{lemma-wiggle-point} choose $y^*$ in $C_M \subset \RR_{\ge 0}^{3g-3}$ such that the half open line segment $\sigma = [y^*, y_1)$ does not intersect $Z_2$. Note that $y^*$ is not in $Z_1$ since $C_M \cap Z_1 = \varnothing$ and $y^* \in C_M$. Thus, $\sigma$ intersects with $Z_1$ finitely many times, and these intersections are not in $Z_2$. Now simply walk from $y^*$ to $y_1$, always staying inside of the embedding of $\FD H$ into $\RR_{\ge 0}^{3g-3}$ by applying Proposition~\ref{proposition-walking-through-I}, if necessary, when arriving at one of the finitely many points $\sigma \cap Z_1$.
		
		Now assume that $g$ is odd. Choose a point $x$ in $\mG$ and attach to it a loop via a bridge. The lengths of the bridge and loop are not important, hence the construction depends on one parameter (the placement of $x$). The resulting graph $\mG_\text b = (G_\text b, \ell_\text b)$ has even genus $g+1$. As $g$ is odd, $\lceil g/2 \rceil + 1$ equals $\lceil (g+1)/2 \rceil + 1$. So there is a tropical morphism $\tmor\!: \mG_\text b \to (T, \ell_T)$ of degree  $\lceil g/2 \rceil + 1$.  By Lemma~\ref{lemma-loop-bridge}, the bridge and loop added to $\mG$ are above two edges of $T$, and no other point of $\mG_\text b$ is above them. Make the desired morphism by deleting these two edges and everything above. 
	\end{proof}
	
	\section{Constructions} \label{sec-constructions}
	We carry out the constructions necessary to prove deformation. 
	This section is motivated by the following paramount observation.
	The gluing datum $M_2$ of Example~\ref{ex-three-cycles-loop} is the limit at $t_1$ of $M_1$, $\Aq M a _3$, and $\Aq M b _3$ with edge labellings $\lambda_1$, $\Aq \lambda a _3$ and $\Aq \lambda b _3$.  There is an isomorphism from $\Aq M a _3$ to $\Aq M b _3$, but $\Aq \lambda a _3$ does not induce $\Aq \lambda b _3$; moreover $\Aq M a _3$ and $\Aq M b _3$ produce in total two non-isomorphic tropical morphisms realizing the same source $\mH$. 
	Let $A_1$, $\Aq A a_2$, $\Aq A b_3$ be the edge-length matrices of these gluing datums. They satisfy: \[\det A_1 + \det \Aq A a_3 + \det \Aq A b _3 = 0.\] We view this equation as a balancing condition, something ubiquitous in tropical geometry.  
	We generalize this relation with Equation~\eqref{eq-star} in Subsection~\ref{subsection-balancing-condition}. 
	
	We construct several candidates to be in $\cstar {M_0}$,
	for $M_0$ the limit at $t_1$ of some $M$ in $\mathbb {FD}_g$. We use Equation~\eqref{eq-star} to prove that at least in one of these candidates the determinant of the edge-length matrix has the sign opposite to that of $\det A_M$; as needed in Lemma~\ref{lemma-construct-next-cone}.
	
	\subsection{Candidates to being full-dimensional} \label{sec-codim1-limit}
	Let $M$ be in $\mathbb{FD}_g$ and $M_0$ be the limit at~$t_1$. Recall that giving a construction for $\cstar {M_0}$ is hard because having a full-rank edge-length matrix is a global condition. Instead, we pass to a larger class of gluing datums that contains~$\cstar {M_0}$. In return we get conditions that can be checked locally, namely they only depend on the gluing relations above $t_1$ and its ends $u$, $v$.

	\begin{de} \label{def-poss-full-dif} Let $g = 2g'$ and $M$ be a genus-$g$ gluing datum. We say that $M$ is \emph{possibly full-dimensional} if $M$ is change-minimal, satisfies dangling-no-glue, no-return, and $A_M$ is a $(3g-3)\times(3g-3)$-matrix. 
	\end{de}
	
	We denote by $\pcstar {M_0}$ the set of isomorphism classes of possibly full-dimensional gluing datums that contract to a limit isomorphic to $M_0$. By Lemma~\ref{lemma-limit-full-dimensional} we have that $\pcstar {M_0}$ contains $\cstar {M_0}$. We index the elements of $\pcstar{M_0}$ by denoting them $\Aq Mq$, with $q=1,2,$ and so on. We write $\Aq T q$ for the base tree of $\Aq M q$, $\Aq G q$ for the graph that arises from $\Aq M q$, $\Aq \dtmor q$ for $\dtmor_{\Aq M q}$, $\Aq r q$ for $r_{\Aq \dtmor q}$, $\Aq m q$ for $m_{\Aq \dtmor q}$, and $\Aq A q = (\Aq a q _{ij})$ for $A_{\Aq M q}$. We also write $T_0$ for the base tree of $M_0$, $G_0$ for the graph that arises from $M_0$, $\dtmor_0$ for $\dtmor_{M_0}$, $r_0$ for $r_{\dtmor_0}$, $m_0$ for $m_{\dtmor_0}$, and $A_0 = (a_{ij})$ for the matrix we get by inserting the zero column at the beginning of $A_{M_0}$.
	
	Recall that by Equation~\eqref{eq-C} on Page~\pageref{eq-C} we have that $w_0$ is not change-minimal; more precisely, $\ch w_0 + \vale w_0 = 4$. Our strategy to construct the elements of $\pcstar {M_0}$ is to regrow $w_0$ back to an edge in such a way that the resulting gluing datum is change-minimal and satisfies dangling-no-glue and no-return. This is possible because, as advertised, these conditions depend solely on the gluing relations above $t_1$ and its ends, $u$ and $v$. 


	\subsection{A balancing condition} \label{subsection-balancing-condition} On the other direction, note that since $\Aq A q$ is a $(3g-3)\times(3g-3)$-matrix we have that $\Aq M q$ in $\pcstar {M_0}$ is full-dimensional if only if $\det \Aq A q \ne 0$. So we need a relation for the determinants.  This is where the balancing condition suggested by Example~\ref{ex-three-cycles-loop} comes in:
	\begin{lm}[balancing condition] \label{lemma-balancing-condition}
		%
		If $M$ is in $\cstar{M_0}$ and $H(M_0)$ is isomorphic to $H(M)$,
		then there exists a finite sequence of triples $(\Aq M q, \Aq \lambda q, K_q)$, indexed by $q$, where the $\Aq M q$ are in $\pcstar{M_0}$ with edge labellings $\Aq \lambda q$ compatible at $t_1$ and the $K_q$ are positive integers, such that every isomorphism class in $\pcstar{M_0}$ has a representative in the $\Aq M q$, and the following equality holds:
		\begin{align} \label{eq-star} \tag{$\star$}
		\sum_{q = 1}^r K_q \det(\Aq A q) = 0.
		\end{align}
	\end{lm}
	
	It is straightforward to see that the balancing condition of Lemma~\ref{lemma-balancing-condition} implies Lemma~\ref{lemma-construct-next-cone}.
	
	\begin{proof}[Proof of Lemma~\ref{lemma-construct-next-cone}.]
		Since $M$ is in $\cstar{M_0}$, we have that $M$ is in $\pcstar{M_0}$. Then by the balancing condition we have some $\hM$ in $\pcstar{M_0}$ with $\sgn \det A_{\hM} \ne \sgn \det A_M$ and $\det A_{\hM} \ne 0$. The latter equation implies that $\hM$ is in $\cstar{M_0}$, as desired.
	\end{proof}
	
	We prove the balancing condition and construct $\pcstar{M_0}$, simultaneously, with a case-by-case analysis. 
	For verifying Equation~\eqref{eq-star} we rely on formulas that follow from Equations~\eqref{eq-ast} and \eqref{eq-astast} on Page~\pageref{eq-ast}.
	To shorten, we write $\Aq c q$ for $\det \Aq Aq$. 
	Consider the adjugate matrix $\adj{\Aq Aq}$, with columns indexed by edges of $H(\Aq Mq)$ and rows indexed by edges of~$\Aq Tq$.
	Let $c_1, \dots, c_{3g-3}$ be the first row of $\adj{\Aq Aq}$. Recall that the $c_i$ depend only on~$M_0$. For an edge $h$ in $H(\Aq Mq)$ or $H(M_0)$ we write $c_h$ for $c_{\Aq \lambda q_H (h)}$ or $c_{\lambda_{0,H}(h)}$, respectively.  Recall that given a non-dangling $e$ in $\Aq Gq$ or $G_0$ we denote by $h(e)$ the unique edge of $H(\Aq Mq)$ or $H(M_0)$, respectively, that contains $e$. If $e$ is dangling, we let $c_{h(e)}$ be zero. Fix a $j$ in $[3g-3]$. We define:
	\begin{align*} 
	\sigma_0(J,\, j) &=  \sum_{e \in J} \dfrac{c_{h(e)}}{|e|},  
	\text{  for  } J \subseteq \inv{\dtmor_{M_0}}(t_j),\\
	\Aq \sigma q(J,\, j) &=  \sum_{e \in J} \dfrac{c_{h(e)}}{|e|}, 
	\text{  for  } J \subseteq \inv{\dtmor_{\Aq Mq}}(t_j).
	\end{align*}	
	
	For convenience we write $\sigma_0(j)$ for $\sigma_0(\inv{\dtmor_{M_0}}(t_j),\, j)$ and $\Aq \sigma q(j)$ for $\Aq \sigma q(\inv{\dtmor_{\Aq M q}}(t_j),\, j)$. 
	
	\begin{lm} \label{lemma-formula-for-cq}
		The following equalities hold: 
		\begin{align*}
		\Aq \sigma q(j) &= \Aq c q \text{ if } j = 1, \text{ otherwise } \Aq \sigma q(j) = 0, \\
		\sigma_0(J,\, j) &= \Aq \sigma q(J,\, j) \text{ for } j\ne 1 \text{ and any } J.
		\end{align*}
	\end{lm}
	\begin{proof}
		The first equality follows from Equations~\eqref{eq-ast} and \eqref{eq-astast}, and the fact that: 
		\[	\Aq \sigma q(J,\, j) =  \sum_{e \in  \inv{\dtmor_{\Aq Mq}}(t_j) } \dfrac{ c_{h(e)}}{|e|}
		= \sum_{i = 1}^{3g-3} c_i \Aq aq_{ij}.\]
		Note that if $j \ne 1$ the edges in $\inv{\dtmor_{M_0}}(t_j)$ correspond to those in $\inv{\dtmor_{\Aq Mq}}(t_j)$, which together with the facts that the $c_i$ depend only on $M_0$ and $\Aq {a_{ij}} q = a_{ij}$ for $j \ne 1$, gives the second equality.
	\end{proof}
	
	Additionally, note that if $J_1, J_2$ are disjoint, then $\sigma_0(J_1,j) + \sigma_0(J_2,j) = \sigma_0(J_1 \cup J_2,j)$.
	
	\subsection{Cases for constructing \texorpdfstring{$\pcstar{M_0}$}{PStar(M0)}} \label{subsection-cases-for-constructing}
	The plan of action to construct $\pstar{M_0}$ has three steps:
	\begin{enumerate}[(1)]
		\item Finding the possible base trees that contract to $T_0$.
		\item Restricting the possibilities for the gluing relations $\Aq {\sim_u} q$, $\Aq {\sim_v} q$, $\Aq {\sim_{t_1}} q$.
		\item Giving a construction of a gluing datum together with edge labellings compatible at $t_1$ for each of the possibilities found in Step~(2). 
	\end{enumerate}
	
	We deal with Step~(1) on Subsection~\ref{sub-graphs-to-T0}.  Step~(2) is the most delicate one, as it proves that indeed we have representatives for all the elements of $\pcstar{M_0}$. We are able to say some useful generalities in Subsections~\ref{subsection-the-graph-GqA0} and \ref{subsection-combinatorial-type}, but the definitive arguments come from assuming certain features of the gluing relations of $M_0$ above $w_0$ and edges in~$\Neigh {w_0}$. Hence, we split the proof in cases. For Step~(3) we refer the reader to the illustrations that accompany each case.
	
	
	
	
	We have, by Equation~\eqref{eq-C} on Page~\pageref{eq-C}, that $\ch w_0 + \vale w_0 = 4$. Thus, $\vale w_0$ is 1, 2, 3 or 4. In our setting the first option is not possible, because if $\vale w_0$ were 1, then we would have $\vale u =1 $ and $\vale v = 2 $ for all gluing datums in $\pcstar{M_0}$, including the full-dimensional~$M$. So Lemma~\ref{lemma-loop-12} applied to $M$ would imply that there is a loop above $t_1$, so the genus of $H(M_0)$ would drop upon contraction of $t_1$, a contradiction. 
	
	Let $A_0$ be a vertex above $w_0$. If $H(M_0)$ is isomorphic to $H(M)$, then $H(M_0)$ is trivalent. Hence, $\nddeg A_0$ is 0, 2, or 3. Also, since $\ch w_0 \le 2$ we have that $r_0(A_0)$ is 0, 1 or~2.
	
	Our philosophy of division of cases is similar to the one for proving the local properties:
	the different cases for calculating $\pstar{M_0}$ are mainly determined by the values of $\vale w_0$, $r_{0}(A_0)$, and $\nddeg A_0$. There are other minor case specific factors that come into play. This gives a total of 17 cases. See Figure~\ref{fig-diagram-proof} for a map that eases navigating through the 15 main cases. The diagram does not show 2 \emph{auxiliary cases} contained in Case~\{aux-r0\}, which we treat in Subsection~\ref{sub-r0}.

	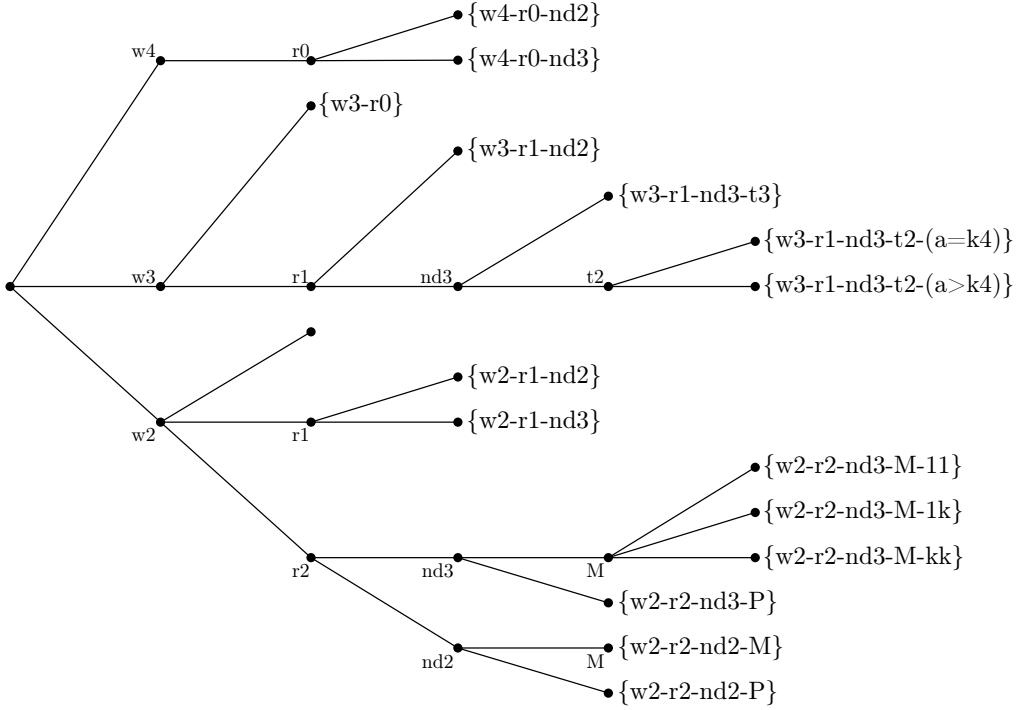
\begin{figure}
		\hspace*{-1.4in}
		\begin{overpic}[scale=0.9]{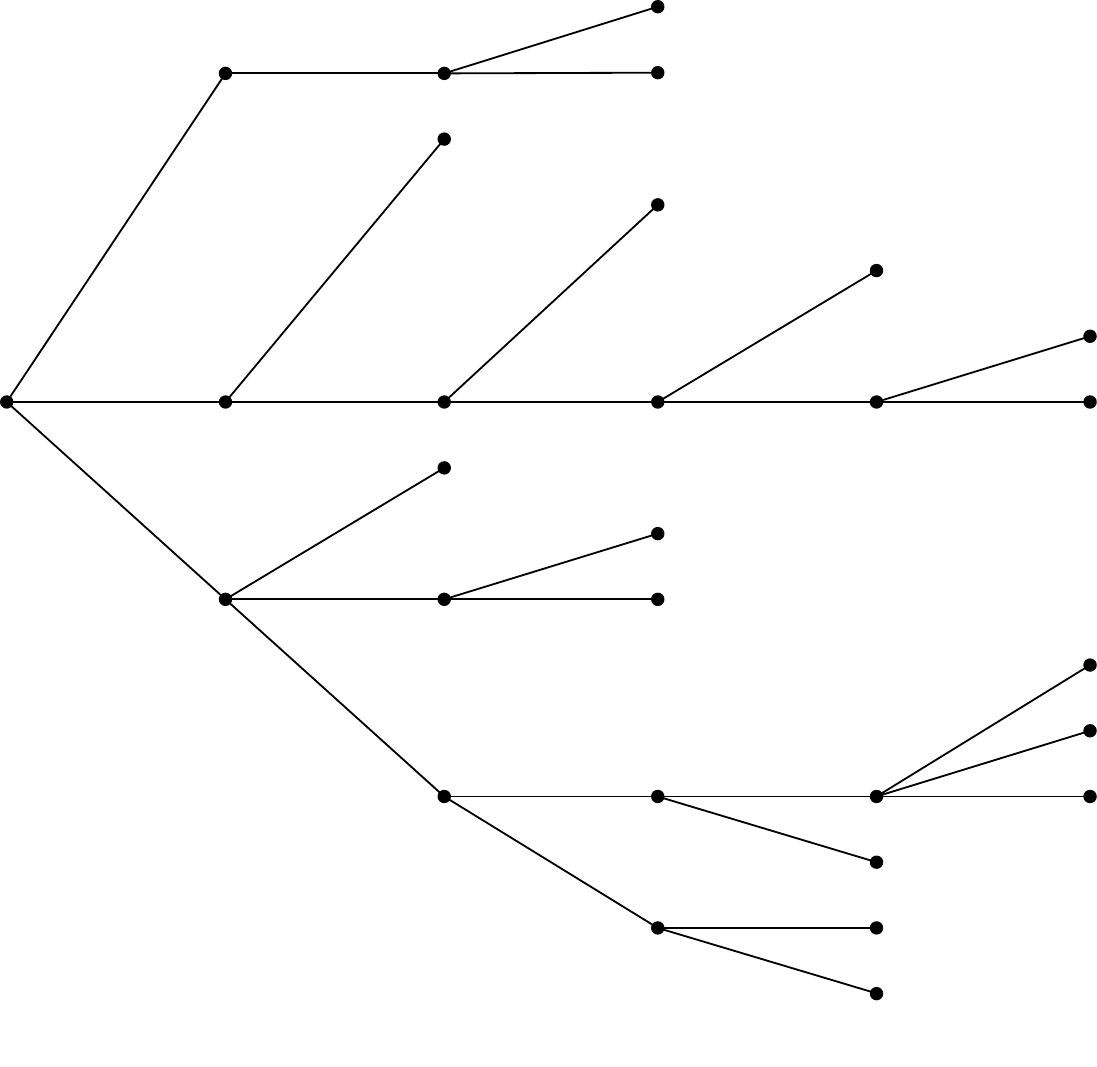}
			
			\put (16.7,91) {\scalebox{0.75}{w4}}
			\put (16.7,61) {\scalebox{0.75}{w3}}
			\put (16.7,40) {\scalebox{0.75}{w2}}
			
			\put (38,91) {\scalebox{0.75}{r0}}
			\put (38,61) {\scalebox{0.75}{r1}}
			\put (38,40) {\scalebox{0.75}{r1}}
			\put (38,22) {\scalebox{0.75}{r2}}
			
			
			\put (55,61) {\scalebox{0.75}{nd3}}
			
			\put (55.2,22) {\scalebox{0.75}{nd3}}
			\put (55.2,10) {\scalebox{0.75}{nd2}}

			\put (77,61) {\scalebox{0.75}{t2}}
			\put (77,22) {\scalebox{0.75}{M}}
			\put (77,10	) {\scalebox{0.75}{M}}
			
			\put (60,95.8) {\scalebox{0.95}{ \{w4-r0-nd2\}}}			
			\put (60,89.8) {\scalebox{0.95}{ \{w4-r0-nd3\}}}		
			\put (40.3,83.8) {\scalebox{0.95}{ \{w3-r0\}}}
			\put (60,77.8) {\scalebox{0.95}{ \{w3-r1-nd2\}}}
			\put (80,71.8) {\scalebox{0.95}{ \{w3-r1-nd3-t3\}}}
			\put (99.3,65.8) {\scalebox{0.95}{ \{w3-r1-nd3-t2-(a=k4)\}}}
			\put (99.3,59.8) {\scalebox{0.95}{ \{w3-r1-nd3-t2-(a>k4)\}}}			
			\put (60,47.8) {\scalebox{0.95}{ \{w2-r1-nd2\}}}
			\put (60,41.8) {\scalebox{0.95}{ \{w2-r1-nd3\}}}
			\put (99.3,35.8) {\scalebox{0.95}{ \{w2-r2-nd3-M-11\}}}
			\put (99.3,29.8) {\scalebox{0.95}{ \{w2-r2-nd3-M-1k\}}}
			\put (99.3,23.8) {\scalebox{0.95}{ \{w2-r2-nd3-M-kk\}}}
			\put (80,17.8) {\scalebox{0.95}{ \{w2-r2-nd3-P\}}}
			\put (80,11.8) {\scalebox{0.95}{ \{w2-r2-nd2-M\}}}
			\put (80,5.8) {\scalebox{0.95}{ \{w2-r2-nd2-P\}}}
		\end{overpic}  
		\caption{\label{fig-diagram-proof} Logical flow of cases to regrow  $w_0$.}
	\end{figure}

	\subsection{Trees contracting to \texorpdfstring{$T_0$}{T0}} \label{sub-graphs-to-T0} We describe the possibilities for $\Aq T q$. Assume that the edges of $T_0$ incident to $w_0$ are labelled with $2, \dots, \deg w_0 + 1$. To construct a tree that contracts to $T_0$ choose disjoint sets $S$, $S'$ whose union is $\{2, \dots, \deg w_0 + 1\}$. At most one of $S$, $S'$ may be empty. Replace $w_0$ with an edge $t_1$ with ends~$u, v$; the edges indexed by $S$ are incident to $u$, and those indexed by $S'$ are incident to~$v$. We denote the resulting tree by~$T_S$. For ease of notation we omit braces; for example we write $T_{2,3}$ for $T_{\{2,3\}}$, and $T_\varnothing$ for the tree corresponding to empty~$S$. Without loss of generality we assume that $|S| \le |S'|$, and in case of equality that $2 \in S$. This gives a one-to-one correspondence between unordered partitions of $\{2, \dots, \deg w_0 + 1\}$ into two parts and trees contracting to~$T_0$. Since $\Aq M q$ is change-minimal, $\vale u \le 3$ and $\vale v \le 3$, that is $|S|, |S'| \le 2$. 
	
	\subsection{The graph \texorpdfstring{$\ndGqAo$}{ndGqA0}} \label{subsection-the-graph-GqA0} 
	Let $\ndGqAo$ be the subgraph of non-dangling elements of $\Aq G q$ that contract to~$A_0$. Since $g(H(M)) = g(H(M_0))$, the graph $\ndGqAo$ is a forest; moreover, by Lemma~\ref{lemma-class-union} it is a tree. By dangling-no-glue, determining the gluing relations $\Aq {\sim_u} q$, $\Aq {\sim_v} q$, $\Aq {\sim_{t_1}} q$ is equivalent to determining $\Neigh A$ for each $A$ in $V(\ndGqAo)$ and the index $\Aq m q(x)$ of each element $x$ in $\ndGqAo$.
	
	
	\begin{lm} \label{lemma-ndval-of-GqA0}
		Let $M$ be a gluing datum, $M_0$ the limit at~$t_1$, and $A_0$ above $w_0$. If $g(H(M)) = g(H(M_0))$, then:
		\[\nddeg{A_0} = \sum_{A \in V(\ndGqAo)} (\nddeg{A} - 2) + 2. \]
	\end{lm}
	
	\begin{proof} 
		Since $\ndGqAo$ is a tree, $|E(\ndGqAo)| = |V(\ndGqAo)|-1$. So we have that:
		\begin{align*}
		\nddeg{A_0} &=  \sum_{A \in V(\ndGqAo)} \nddeg{A}  - 2|E(\ndGqAo)|\\ 
		&= \sum_{A \in V(\ndGqAo)} \nddeg{A} - 2(|V(\ndGqAo)|-1)\\
		&= \sum_{A \in V(\ndGqAo)} (\nddeg{A} - 2) + 2. \qedhere	
		\end{align*}
	\end{proof}
	
	Let $A$ be in $V(\ndGqAo)$. By Lemma~\ref{lemma-ndval-of-GqA0} if $\nddeg A_0 = 2$, then $\nddeg A = 2$. If $\nddeg A_0 = 3$, then either $\nddeg A = 2$ or $\nddeg A = 3$; moreover, the latter case occurs exactly once. So when $\nddeg A$ is 3 we denote by $\Aq A q$ the unique vertex of $\Aq G q$ with non-dangling valency 3 that contracts to $A_0$.
	We make two further observations, one on the edge set and one on the vertex set of $\ndGqAo$:
	
	
	
	\begin{lm}  \label{lemma-edges-in-GqA0}
		Let $r$ be the number of vertices of $\Aq G q$ that contract to $A_0$ and belong either to Case~(r1-nd3) or Case~(r2-nd2) of Proposition~\ref{prop-local} (local properties). Then $|E(\ndGqAo)| \le r+1 \le  r_0(A_0) + 1$.
	\end{lm}
	
	\begin{proof}
		Let $A$ be in $V(\ndGqAo)$. By the local properties and the fact that all the edges of $\ndGqAo$ are above $t_1$, we have that $\vale_{\ndGqAo} A$ is either 1 or 2. Thus, $\ndGqAo$ is in fact a path, so the inner vertices have valency~2. Note that valency~2 implies that $A$ belongs either to Case~(r1-nd3) or Case~(r2-nd2), so the result follows.
	\end{proof}
	
	\begin{lm} \label{lemma-vertices-in-GqA0}
		If $r_0(A_0) \le 1$, then the vertices of $\ndGqAo$ are the ends above $u$, $v$ of the edges $\Aq e q$, where $\Aq e q$ corresponds to $e$ in $\Neighnd {A_0}$.
	\end{lm}
	
	\begin{proof}
		Let $A$ be in $V(\ndGqAo)$. Since $r_0(A_0) \le 1$, we have that $\Aq r q(A) \le 1$. Thus, by the local properties $A$ is not above a leaf. So $A$ satisfies the no-return condition, which means that not all of the edges in $\Neighnd A$ are above $t_1$, so at least one of them corresponds to one of the edges in $\Neighnd {A_0}$.
	\end{proof}

	\subsection{Combinatorial type} \label{subsection-combinatorial-type}  
	We can obtain $H(M_0)$ from $H(M)$ by contracting a set of edges. If the number of edges contracted is zero, then $H(M_0)$ and $H(M)$ are isomorphic. This is equivalent to having a trivalent $H(M_0)$.  This is the setting of Lemma~\ref{lemma-balancing-condition} (balancing condition). Thus, we assume for the rest of this section that $H(M_0)$ and $H(M)$ are isomorphic. Since the edge labelling on $\pcstar{M_0}$ is compatible at $t_1$, the identification is canonical.
	
	The case where $H(M_0)$ and $H(M)$ are not isomorphic is of interest as well, and we deal with it in Part II of this series of articles. It is by studying the $\pcstar{M_0}$ of $M_0$ with non-trivalent $H(M_0)$ that we show how it is possibly to walk between cones $C_H$ with different combinatorial types $H$. We use this to prove that the space $\TM d g$ we alluded to in the sneak peak part of Subsection~\ref{subsection-gonality} is connected. 
	
	\subsection{Case-work} \label{subsection-case-work} Let $A_0$ be above $w_0$. Let $e$ in $\Aq G q$ contract to $A_0$. Let $A_u$, $A_v$ be the ends of $e$, above $u$, $v$ respectively. If $e$ is dangling, then it does not contribute to $\Aq \sigma q (1)$. So, by Lemma~\ref{lemma-dangling-limit} (dangling in the limit) we can assume that $A_0$ is non-dangling. We denote by $J_r$ the set of those $e$ contracting to $A_0$ such that $\Aq rq(A_u) + \Aq rq(A_v) = r$. By the discussion in Subsection~\ref{subsection-cases-for-constructing} we have that $r_0(A_0) \le 2$, hence:
	\[\det \Aq A q = \Aq \sigma q (1) = \Aq \sigma q (J_0, 1) + \Aq \sigma q (J_1, 1) + \Aq \sigma q (J_2, 1). \]
	In each of the following cases we calculate $\Aq \sigma q (J_0, 1)$, $\Aq \sigma q (J_1, 1)$, and $\Aq \sigma q (J_2, 1)$ separately. To avoid having several levels of subindices, we write $c(e)$ for $c_{h(e)}$.
	

	\subsubsection{Case \{aux-r0\}} \label{sub-r0}
	We begin by showing that $\Aq T q = T_{\Aq S q}$ determines $\ndGqAo$ when $r_0(A_0)$ is~0. This is an auxiliary case for calculating $\Aq \sigma q (J_0, 1)$. Since $\vale w_0 \ge 2$, by Lemma~\ref{lem-noreturn-rphi1} (r1 implies no-return) we have that $A_0$ satisfies the no-return condition. Thus, if $\nddeg A_0 = 2$, then $\dtmor_0$ is injective on $\Neighnd {A_0}$. If $\nddeg A_0 = 3$, suppose that $\dtmor_0$ is not injective on $\Neighnd {A_0}$, so $\sum_{e \in \Neighnd{A_0} } |e| \le 2|A_0|$. Hence, $r_0(A_0) = \nddeg A_0 - 2 + 2|A_0| - \sum_{e \in \Neighnd{A_0} } |e| \ge 1$, a contradiction. 
	
	\begin{itemize}[leftmargin=*]
		\item Case \{aux-r0-nd2\}:
		%
		Assume that $\nddeg A_0$ is 2. Let $e_\alpha, e_\beta$ be the edges in $\Neighnd {A_0}$, above $t_\alpha$ and $t_\beta$ respectively. By Lemma~\ref{lemma-vertices-in-GqA0} the vertices of $\ndGqAo$ are the ends of $\Aq {e_\alpha} q$ and $\Aq {e_\beta} q$ above $u$ or $v$. If $\{\alpha, \beta\} \subset \Aq S q$, the ends of $\Aq {e_\alpha} q$, $\Aq {e_\beta} q$ are above~$u$; since $\GqAo$ is connected, they equal one vertex $A$ in~$\Aq G q$, which equals $A_0$ as subsets of $[d]$ by Lemma~\ref{lemma-class-union}. Similarly if $\{\alpha, \beta\} \subset \Aq {S'} q$. Otherwise, one end $A_u$ is above $u$, the other end $A_v$ above $v$, so they are distinct. By connectivity of $\ndGqAo$ there is one edge $e'$ joining $A_u$, $A_v$. By Lemma~\ref{lemma-ndval-of-GqA0} and since $r_0(A_0) = 0$, the vertices $A_u$ and $A_v$ belong to  Case~(r0-nd2) of the local properties. So, as subsets of $[d]$, the classes $\Aq {e_\alpha} q$, $A_u$, $e'$, $A_v$, $\Aq {e_\beta} q$ are equal; and $h(\Aq {e_\alpha} q) = h(e') = h(\Aq {e_\beta} q)$.
		
		\item Case \{aux-r0-nd3\}: Assume that $\nddeg A_0$ is 3. Let $e_\alpha, e_\beta, e_\gamma$ be the edges in $\Neighnd {A_0}$, above $t_\alpha$, $t_\beta$, $t_\gamma$, respectively. Since $|\dtmor_0(\Neighnd {A_0})| = 3$ and $\max(|\Aq S q|,|\Aq {S'} q|) \le 2$, both intersections $\Aq S q \cap \dtmor_0(\Neighnd {A_0})$ and $\Aq {S'} q \cap \dtmor_0(\Neighnd {A_0})$ are non-empty. One of these intersections is a singleton. Assume without loss of generality that the singleton is $\{\alpha \}$. By Lemma~\ref{lemma-edges-in-GqA0} there is at most one edge in $\ndGqAo$, therefore at most two vertices. By Lemma~\ref{lemma-vertices-in-GqA0} the vertices of $\ndGqAo$ are the end $A_2$ of $\Aq {e_\alpha} q$ above $u$, and the vertex $A_3$ that is the end of both $\Aq {e_\beta} q$ and $\Aq {e_\gamma} q$ above~$v$. So $A_2$ and $A_3$ are distinct, joined by a non-dangling edge~$e'$, and $\nddeg A_2 = 2$, $\nddeg A_3 = 3$. By the Case~(r0-nd2) of the local properties, as subsets of $[d]$, we get that $\Aq {e_\alpha} q$, $A_2$, and $e'$ are equal, and $h(\Aq {e_\alpha} q) = h(e')$. Note that they are also a subset of~$A_3$. Finally Lemma~\ref{lemma-class-union} implies that $A_0 = A_3 \cup A_2 = A_3$.
	\end{itemize}
	
	\subsubsection{Case \{w4\}} \label{sub-deg-four} 
	
	Assume that $\vale w_0$ is 4. Then $\ch w_0 = 0$ by Equation~\eqref{eq-C} on Page~\pageref{eq-C}, so $r_0(A_0) = 0$. 
	Thus, the two Cases~\{aux-r0-nd2\} and~\{aux-r0-nd3\} settle entirely this case. 
	Namely, the elements of $\pstar{M_0}$ are $\Aq M 1$, $\Aq M 2$ and $\Aq M 3$, determined by $T_{2,3}$, $T_{2,4}$, and $T_{2,5}$, respectively. 
	
	
	
	
	\begin{itemize}[leftmargin=*]
		\item Case \{w4-r0-nd2\}: Let $\Neighnd{A_0} = \{e_\alpha, e_\beta \}$ with $\dtmor_0(e_\alpha) = t_\alpha$, $\dtmor_0(e_\beta) = t_\beta$. Applying Case~\{aux-r0-nd2\} we get the local parts shown in Figure~\ref{figure-26}, and a formula.
		
		\noindent
		\begin{minipage}{\textwidth}
			\begin{minipage}[t]{.2\textwidth}

				\centering
				\begin{overpic}[scale=0.9]{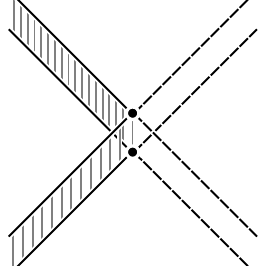} 
					\put (45,26){\scalebox{0.8}{$w_0$}} 
					\put (16,0) {\scalebox{0.8}{$t_\alpha$}}
					\put (3,72) {\scalebox{0.8}{$t_\beta$}}
					\put (90,72){\scalebox{0.8}{$t_\gamma$}} 
					\put (75,0) {\scalebox{0.8}{$t_\delta$}}		 
				\end{overpic}
				
				\vspace{0.5em}
				
				$M_0$
			\end{minipage}\hspace{1em}
			\noindent \begin{minipage}[t]{.22\textwidth} 
				\centering
				\begin{overpic}{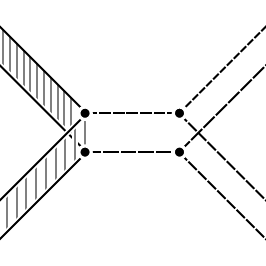} 
					\put (47,34) {\scalebox{0.8}{$t_1$}}
					\put (31,34) {\scalebox{0.8}{$u$}}
					\put (64,34) {\scalebox{0.8}{$v$}}
					\put (9,9) {\scalebox{0.8}{$t_\alpha$}}
					\put (0,62) {\scalebox{0.8}{$t_\beta$}}
					\put (96,62) {\scalebox{0.8}{$t_\gamma$}}
					\put (83,9) {\scalebox{0.8}{$t_\delta$}}		
				\end{overpic} 
				
				\vspace{0.5em}
				
				$T_{\alpha,\beta}$
			\end{minipage}\hspace{1em}
			\begin{minipage}[t]{.22\textwidth}
				\centering 
				\begin{overpic}{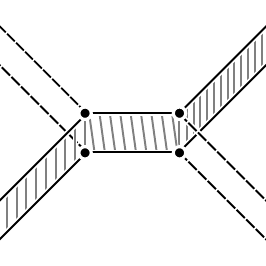} 
					\put (47,34) {\scalebox{0.8}{$t_1$}}
					\put (31,34) {\scalebox{0.8}{$u$}}
					\put (64,34) {\scalebox{0.8}{$v$}}
					\put (9,9) {\scalebox{0.8}{$t_\alpha$}}
					\put (0,62) {\scalebox{0.8}{$t_\gamma$}}
					\put (96,62) {\scalebox{0.8}{$t_\beta$}}
					\put (83,9) {\scalebox{0.8}{$t_\delta$}}
				\end{overpic} 
				
				\vspace{0.5em}
				
				$T_{\alpha,\gamma}$
			\end{minipage}\hspace{1em}
			\begin{minipage}[t]{.22\textwidth} 
				\centering
				\begin{overpic}{\figsdir/251.pdf} 
					\put (47,34) {\scalebox{0.8}{$t_1$}}
					\put (31,34) {\scalebox{0.8}{$u$}}
					\put (64,34) {\scalebox{0.8}{$v$}}
					\put (9,9) {\scalebox{0.8}{$t_\alpha$}}
					\put (0,62) {\scalebox{0.8}{$t_\delta$}}
					\put (96,62) {\scalebox{0.8}{$t_\beta$}}
					\put (83,9) {\scalebox{0.8}{$t_\gamma$}}
				\end{overpic}
				
				\vspace{0.5em}  
				
				$T_{\alpha,\delta}$
			\end{minipage}
			\captionof{figure}{} \label{figure-26}
		\end{minipage}	
		
		\begin{align} \label{eq-w4-nd2} \tag{w4-nd2}
		\sum_{q=1}^3 \sum_{e \in E(\ndGqAo)} \dfrac {c(e)} {|e|} = \dfrac{c(e_\alpha) }{|e_\alpha|} + \dfrac{c(e_\beta)}{|e_\beta|}.
		\end{align}
		
		\item Case \{w4-r0-nd3\}: Let $\Neighnd{A_0} = \{e_\alpha, e_\beta, e_\gamma \}$ with $\dtmor_0(e_\alpha) = t_\alpha$, $\dtmor_0(e_\beta) = t_\beta$, $\dtmor_0(e_\gamma) = t_\gamma$. Applying Case~\{aux-r0-nd3\} we get the local parts shown in Figure~\ref{figure-27}, and a formula.
		
		\vspace{1em}
		
		\noindent
		\begin{minipage}{\textwidth}
			\begin{minipage}[t]{.2\textwidth}

				\centering
				\begin{overpic}[scale=0.9]{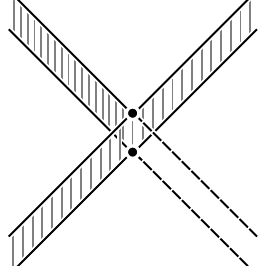} 
					\put (45,26){\scalebox{0.8}{$w_0$}} 
					\put (16,0) {\scalebox{0.8}{$t_\alpha$}}
					\put (3,72) {\scalebox{0.8}{$t_\beta$}}
					\put (90,72){\scalebox{0.8}{$t_\gamma$}} 
					\put (75,0) {\scalebox{0.8}{$t_\delta$}}		 
				\end{overpic}
				
				\vspace{0.5em}
				
				$M_0$
			\end{minipage}\hspace{1em}
			\noindent \begin{minipage}[t]{.22\textwidth} 
				\centering
				\begin{overpic}{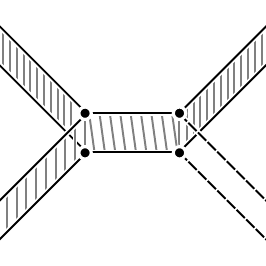} 
					\put (47,34) {\scalebox{0.8}{$t_1$}}
					\put (31,34) {\scalebox{0.8}{$u$}}
					\put (64,34) {\scalebox{0.8}{$v$}}
					\put (9,9) {\scalebox{0.8}{$t_\alpha$}}
					\put (0,62) {\scalebox{0.8}{$t_\beta$}}
					\put (96,62) {\scalebox{0.8}{$t_\gamma$}}
					\put (83,9) {\scalebox{0.8}{$t_\delta$}}		
				\end{overpic} 
				
				\vspace{0.5em}
				
				$T_{\alpha,\beta}$
			\end{minipage}\hspace{1em}
			\begin{minipage}[t]{.22\textwidth}
				\centering 
				\begin{overpic}{\figsdir/254.pdf} 
					\put (47,34) {\scalebox{0.8}{$t_1$}}
					\put (31,34) {\scalebox{0.8}{$u$}}
					\put (64,34) {\scalebox{0.8}{$v$}}
					\put (9,9) {\scalebox{0.8}{$t_\alpha$}}
					\put (0,62) {\scalebox{0.8}{$t_\gamma$}}
					\put (96,62) {\scalebox{0.8}{$t_\beta$}}
					\put (83,9) {\scalebox{0.8}{$t_\delta$}}
				\end{overpic} 
				
				\vspace{0.5em}
				
				$T_{\alpha,\gamma}$
			\end{minipage}\hspace{1em}
			\begin{minipage}[t]{.22\textwidth} 
				\centering
				\begin{overpic}{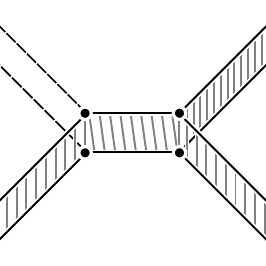} 
					\put (47,34) {\scalebox{0.8}{$t_1$}}
					\put (31,34) {\scalebox{0.8}{$u$}}
					\put (64,34) {\scalebox{0.8}{$v$}}
					\put (9,9) {\scalebox{0.8}{$t_\alpha$}}
					\put (0,62) {\scalebox{0.8}{$t_\delta$}}
					\put (96,62) {\scalebox{0.8}{$t_\beta$}}
					\put (83,9) {\scalebox{0.8}{$t_\delta$}}
				\end{overpic}
				
				\vspace{0.5em}  
				
				$T_{\alpha,\delta}$
			\end{minipage}
			\captionof{figure}{} \label{figure-27}
		\end{minipage}
		
		\begin{align} \label{eq-w4-nd3} \tag{w4-nd3}
		\sum_{q=1}^3 \sum_{e \in E(\ndGqAo)} \dfrac {c(e)} {|e|} = \dfrac{c(e_\alpha) }{|e_\alpha|} + \dfrac{c(e_\beta)}{|e_\beta|} + \dfrac{c(e_\gamma)}{|e_\gamma|}.  
		\end{align}
		
		\item We verify Equation~\eqref{eq-star} for Case~\{w4\} by putting together Lemma~\ref{lemma-formula-for-cq} and Equations \eqref{eq-w4-nd2}, \eqref{eq-w4-nd3}: 
		\begin{align} \label{eq-1}
		\Aq c 1 + \Aq c 2 + \Aq c 3 = \sum_{q=1}^3 \Aq \sigma q (1) &=  \sum_{q=1}^3 \, \sum_{A_0 \in \inv \dtmor_0(A)} \, \sum_{e \in E(\ndGqAo)} \dfrac {c(e)} {|e|} \\
		&= \sum_{j=2}^5 \Aq \sigma q (j) = 0. \nonumber
		\end{align}
	\end{itemize}

	\subsubsection{Case \{w3\}} \label{sub-deg-three} 
	
	Assume that $\vale w_0$ is 3. Then $\ch w_0 = 1$ by Equation~\eqref{eq-C}. There is exactly one vertex of $G_0$ with $r_0$-value 1 above $w_0$, the others have $r_0$-value 0. By Subsection~\ref{sub-graphs-to-T0}, the possible base trees contracting to $T_0$ are $T_{2}$, $T_{3}$ and~$T_{4}$; all with $\vale u = 2$, $\vale v = 3$. 
	Let $\{\alpha\} = \Aq S q$. We analyse the cases where $r_0(A_0)$ is either 0 or~1.
	
	\begin{itemize}[leftmargin=*]
		\item Case \{w3-r0\}: Assume that $r_0(A_0) = 0$, and let $A_u$ in $\Aq {G_{A_0}} q$ be non-dangling above~$u$. Then $\Aq r q(A_u) = 0$. Since $\vale u = 2$, only Case~(r0-nd2) of the local properties is possible, so $\nddeg A = 2$ and as a subset of $[d]$ the class $A_u$ is equal to both of the elements in $\Neighnd{A_u}$. Since any non-dangling edge in $\inv {(\Aq \dtmor q)}(t_\alpha)$ has a non-dangling end above~$u$, we conclude that:
		\[\Aq \sigma q(J_0,1) = \Aq \sigma q (J_0,\alpha) = \sigma_0(J_0,\alpha). \] 
		
		\item Case \{w3-r1\}: Since $H_0$ is trivalent either $\nddeg {A_0}$ is 3 or 2.
		
		\item Case \{w3-r1-nd3\}: Assume that  $\nddeg {A_0}$ is 3. By the no-return condition $\dtmor_0(\Neighnd{A_0})$ has at least two elements. So let $e_2, e_3, e_4$ be the edges in $\Neighnd{A_0}$, labelled such that $|e_4| = \max(|e_2|, |e_3|, |e_4|)$, $\dtmor_0(e_3) = t_3$ and $\dtmor_0(e_4) = t_4$. Let $k_i = |e_i|$. 
		Lemma~\ref{lem-rphi-nd} (nd.~$r_\dtmor$ formula) gives $r_0(A_0) = 1 = 3 + 2|A_0| - 2 - (k_2 + k_3 + k_4)$. That is, $2|A_0| = k_2 + k_3 + k_4$. If $e_2$ were above $t_4$, then $k_2 + k_4 \le |A_0|$. Substituting in $k_3 = 2|A_0| - (k_2 + k_4)$ gives $k_3 \ge |A_0| \ge k_2 + k_4 > k_4$, a contradiction. So $e_2$ is either above $t_2$ or above~$t_3$.
		
		\item Case \{w3-r1-nd3-t2\}: Assume that $e_2$ is above $t_2$.  
		Let $\{\alpha, \beta, \gamma \}=\{2,3,4\}$.
		Recall from Subsection~\ref{subsection-the-graph-GqA0} that $\Aq A q$ is the unique vertex of $\Aq G q$ with $\nddeg \Aq A q = 3$ that contracts to $A_0$.
		We say that $\Aq M q$ is in Position I or II if $\Aq A q$ is above $u$ or $v$, respectively. We prove that $\Aq T q$ and the position of $\Aq A q$ determine~$\Aq M q$.
		
		Position I: Assume that $\Aq A q$ is above $u$. 
		Since $\vale u = 2$ and $\nddeg \Aq A q = 3$ we have that $\Aq A q$ belongs to Case~(r1-nd3) of the local properties (assuming Case~(r0-nd3) contradicts that $\vale u = 2$).
		So there are exactly two edges $e', e''$ of $\Neighnd{\Aq A q}$ above~$t_1$; moreover $|e_\alpha| = |\Aq A q| = |e'| + |e''|$. 
		By Lemma~\ref{lemma-edges-in-GqA0} these are the two edges of $\ndGqAo$. Hence, $\ndGqAo$ has three vertices. By Lemma~\ref{lemma-vertices-in-GqA0} these are: $\Aq A q$ incident to $\Aq {e_\alpha} q$; and the two ends of $\Aq {e_\beta} q$, $\Aq {e_\gamma} q$ above $v$, with non-dangling valency 2 and $\Aq r q$-value equal to~0. By the Case (r0-nd2) of the local properties we conclude that as subsets of $[d]$, the edges $e'$, $\Aq {e_\beta} q$ and their common end are equal, and $h(e') = h(\Aq {e_\beta} q)$; similarly with $e''$ and $\Aq {e_\gamma} q$. Thus, $|e_\alpha| = |e_\beta| + |e_\gamma|$ and $|\Aq A q| = |A_0|$, so this construction is possible only when $\alpha = 4$ and $k_4 = |A_0|$.
		
		Position II: Assume that $\Aq A q$ is above $v$. 
		Since $\vale v = 3$ and $\nddeg \Aq A q = 3$ we have that $\Aq A q$ belongs to Case~(r0-nd3) of the local properties (assuming Case~(r1-nd3) contradicts that $\vale v = 3$). So there is exactly one edge $e'$ of $\Neighnd{\Aq A q}$ above~$t_1$. 
		Let $A'$ be the other end of $e'$, which by Lemma~\ref{lemma-ndval-of-GqA0} has non-dangling valency 2. Lemma~\ref{lemma-edges-in-GqA0} implies that the two vertices of $\ndGqAo$ are $\Aq A q$ and $A'$. Hence, by Lemma~\ref{lemma-vertices-in-GqA0} $A'$ is incident to $\Aq {e_\alpha} q$; and $h(\Aq {e_\alpha} q) = h(e')$. Since $r_0(A_0) = 1$, we have that $A'$ belongs to Case~(r1-nd2) of the local properties and this gives two possible cases for the values of $|A'|$, $|e'|$ and $|\Aq A q|$:
		
		Position II.a: $|A'| = |e'|$ and $|A'| = |\Aq {e_\alpha} q| + 1$. In this case, as subsets of $[d]$, we get that $\Aq {e_\alpha} q \subset A'$, that $A'$ and $e'$ are equal, that $e' \subset \Aq A q$, and therefore, by Lemma~\ref{lemma-class-union}, that $A_0 = \Aq A q \cup A' = \Aq A q$. We also get that $|A_0| \ge |A'| > k_\alpha$.
		
		Position II.b: $|A'| = |e'| + 1$ and $|A'| = |\Aq {e_\alpha} q|$. As subsets of $[d]$ we have that $e' \subset A' \cap \Aq A q$, so $A' \setminus \Aq A q$ is either empty or a singleton. The former would imply that for $i \in A' \setminus e'$, the classes $e'$ and $\class {t_1} i _{\Aq M q}$ are distinct edges of $\Aq G q$ above $t_1$ with equal ends, which gives a cycle over $t_1$, a contradiction. So $|A_0| = |\Aq A q| + |A'| - |\Aq A q \cap A'| = |\Aq A q| + 1$. The ends of $\Aq {e_\beta} q$ and $\Aq {e_\gamma} q$ are above $v$, so $|\Aq A q| \ge k_\beta, k_\gamma$. Hence, $|A_0| > \max(k_\beta, k_\gamma)$, so $k_\alpha \ge 2$.

		Since Position~I implies $|A_0| = k_4$, we treat two cases: $|A_0| = k_4$ and $|A_0| > k_4$.
		
		\item Case \{w3-r1-nd3-t2-(a=k4)\}: Assume that $|A_0| = k_4$. Then $k_2 + k_3 = |A_0|$, so $k_4 > k_2, k_3$. Let $\alpha = 4$, so $|A_0| = k_4$, which precludes Position II.a. Position~I and Position~II.b give $\Aq M 1$ and $\Aq M 2$, respectively. Let $\alpha = 2$, which precludes Position~I. Then it is not true that $|A_0| > \max(k_\beta, k_\gamma) = k_4$, which precludes Position~II.b as well. Position~II.a gives $\Aq M 3$. An analogous argument with $\alpha = 3$ gives $\Aq M 4$. 
		See Figure~\ref{figure-28}.
		
		\vspace{1em}
		
		\noindent
		\begin{minipage}{.3\textwidth}
			\centering
			\begin{overpic}[scale=1]{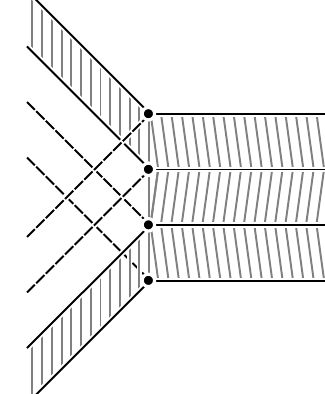}  
				\put (37,74) {\scalebox{1}{$A_0$}}
				\put (60,19) {\scalebox{1}{$t_4$}}
				\put (26,84) {\scalebox{1}{$t_2$}}
				\put (26,8) {\scalebox{1}{$t_3$}}
				\put (35,19) {\scalebox{1}{$w_0$}}
				
				\put (-2,93) {\scalebox{1}{$e_2$}}
				\put (-2,3) {\scalebox{1}{$e_3$}} 
				\put (83,48) {\scalebox{1}{$e_4$}}  
			\end{overpic}
			
			$M_0$
		\end{minipage}\hspace{1em} 
		\begin{minipage}{.25\textwidth}
			\centering
			{ \small
				\begin{align*}
				\dfrac {c(e_4)} {k_2 + k_3} + \sigma_0(J_0, 4) &= 0\\
				\dfrac {c(e_2)} {k_2} + \sigma_0(J_0, 2) &= 0\\
				\dfrac {c(e_3)} {k_3} + \sigma_0(J_0, 3) &= 0
				\end{align*}
			}
		\end{minipage}\hspace{0.5em}
		\begin{minipage}{.33\textwidth}
			\centering
			\begin{overpic}[scale=1]{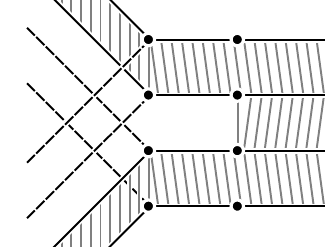}  
				\put (43,4) {\scalebox{1}{$v$}}
				\put (70,4) {\scalebox{1}{$u$}}
				\put (56,4) {\scalebox{1}{$t_1$}}
				\put (70,67) {\scalebox{1}{$A^{(1)}$}}  
			\end{overpic}
			
			$M^{(1)}$
			{\small
				\begin{align*}
				&|e'| = k_2 \quad \quad |e''| = k_3
				\\
				&\Aq \sigma q(J_0, 1) = \sigma_0(J_0, 4)\\
				&\Aq c1 = \dfrac {c(e_2)} {k_2} + \dfrac {c(e_3)} {k_3} + \Aq \sigma q(J_0, 1) 
				\end{align*}
			}
		\end{minipage}

		\noindent
		\begin{minipage}{\textwidth}
			\noindent	\begin{minipage}[t]{.3\textwidth}
				\centering
				\begin{overpic}[scale=1]{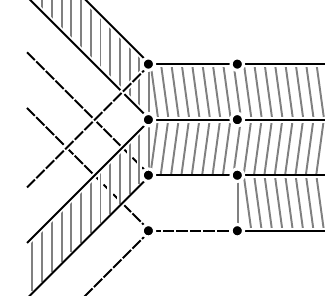} 
					\put (44,12) {\scalebox{1}{$v$}}
					\put (70,12) {\scalebox{1}{$u$}}
					\put (56,12) {\scalebox{1}{$t_1$}}
					\put (43,75) {\scalebox{1}{$A^{(2)}$}}
					
					\put (14,93) {\scalebox{1}{$e_2$}}
					\put (-2,6) {\scalebox{1}{$e_3$}} 
					\put (101,44) {\scalebox{1}{$e_4$}}   
				\end{overpic}
				
				$M^{(2)}$ \\ \vspace{0.5em}
				{\small $
					\begin{aligned}
					&|e'| = k_4 - 1 = k_2 + k_3 - 1\\ 
					&\Aq \sigma 2(J_0, 1) = \sigma_0(J_0, 4)\\
					&c^{(2)} = \dfrac {c(e_4)} {k_2 + k_3 - 1} + \Aq \sigma 2(J_0, 1)
					\end{aligned} $
				}
			\end{minipage} \hspace{1em}
			\begin{minipage}[t]{.3\textwidth}
				\centering
				\begin{overpic}[scale=1]{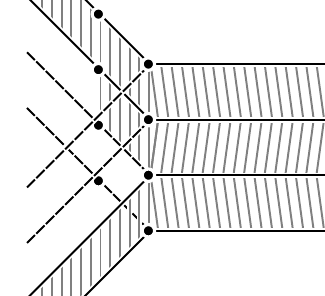}
					\put (30,89) {\scalebox{1}{$u$}}
					\put (46,73) {\scalebox{1}{$v$}}
					\put (38,81) {\scalebox{1}{$t_1$}}
					\put (41,10) {\scalebox{1}{$A^{(3)}$}} 
				\end{overpic}
				
				$M^{(3)}$ \\ \vspace{0.5em}
				{\small $
					\begin{aligned}
					&|e'| = k_2 + 1
					\\
					&\Aq \sigma 3(J_0, 1) = \sigma_0(J_0, 2)\\
					&c^{(3)} = \dfrac {c(e_2)} {k_2 + 1} + \Aq \sigma 3(J_0, 1)
					\end{aligned} $
				}
			\end{minipage}\hspace{1em}
			\begin{minipage}[t]{.3\textwidth}
				\centering
				\begin{overpic}[scale=1]{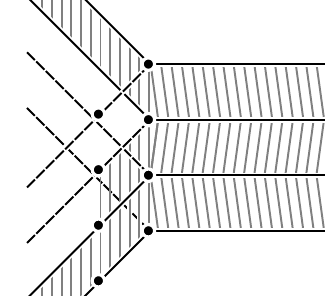} 
					\put (31,0) {\scalebox{1}{$u$}}
					\put (45,12) {\scalebox{1}{$v$}}
					\put (39,6) {\scalebox{1}{$t_1$}}
					\put (45,74) {\scalebox{1}{$A^{(4)}$}} 
				\end{overpic}
				
				$M^{(4)}$ \\ \vspace{0.5em}
				{\small $
					\begin{aligned}
					&|e'| = k_3 + 1\\
					&\Aq \sigma 4(J_0, 1) = \sigma_0(J_0, 3)\\
					&c^{(4)} = \dfrac {c(e_3)} {k_3 + 1} + \Aq \sigma 4(J_0, 1) 
					\end{aligned} $
				}
			\end{minipage} 
			\captionof{figure}{} \label{figure-28}
		\end{minipage}
		
		\noindent To verify Equation~\eqref{eq-star} we compute: 
		\begin{align} \label{eq-2}
		&c^{(1)} + (k_2 + k_3 - 1)c^{(2)} +(k_2+1) c^{(3)} + (k_3+1) c^{(4)}\\
		\nonumber 
		&= \left( \tfrac{(k_2 + 1)}{k_2}c(e_2) + (k_2 + 1)\sigma_0(J_0, 2) \right) +
		\left(\tfrac{(k_3 + 1)}{k_3}c(e_3) + (k_3 + 1)\sigma_0(J_0, 3) \right)\\ \nonumber  	
		&+ \left(c(e_4) + (k_2 + k_3)\sigma_0(J_0, 4)  \right)
		= 0 + 0 + 0 = 0.
		\end{align}
		
		\vspace{1em}
		
		\item Case \{w3-r1-nd3-t2-(a>k4)\}: Assume that $|A_0| > k_4$, which precludes Position~I. If $\min(k_2, k_3, k_4)$ were 1, then $k_2 + k_3 + k_4 \le 2|A_0| - 1$, contradicting $0 = 2|A| - k_2 - k_3 - k_4$. Hence, $\min(k_2, k_3, k_4) \ge 2$, so both cases of Position II are possible. It turns out that they balance each other for each possible~$\alpha$. The case $\alpha = 4$ is shown in Figure~\ref{figure-29}. 
		\vspace{1em} 
		
		\noindent
		\begin{minipage}{\textwidth}
			\begin{minipage}{.3\textwidth}
				\centering
				\begin{overpic}[scale=1]{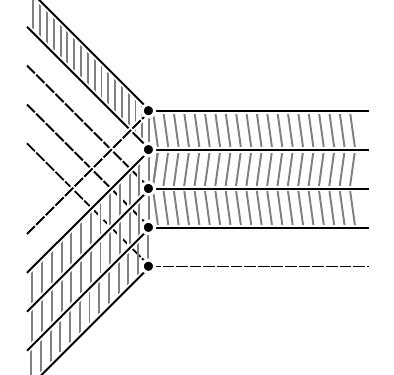} 
					\put (60,18) {\scalebox{1}{$t_4$}}
					\put (26,82) {\scalebox{1}{$t_2$}}
					\put (26,8) {\scalebox{1}{$t_3$}}
					\put (34,18) {\scalebox{1}{$w_0$}}
					\put (36,70) {\scalebox{1}{$A_0$}}
					
					\put (-2,79) {\scalebox{1}{$e_2$}}
					\put (-2,10) {\scalebox{1}{$e_3$}} 
					\put (92,51) {\scalebox{1}{$e_4$}}  
				\end{overpic}
				
				$M_0$
			\end{minipage}\hspace{0.1em}
			\begin{minipage}{.3\textwidth}
				\centering
				\begin{overpic}[scale=1]{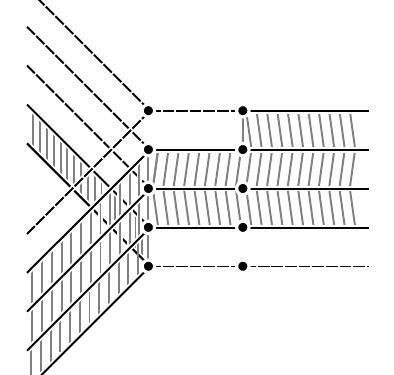} 
					\put (47,18) {\scalebox{1}{$t_1$}}
					\put (60,18) {\scalebox{1}{$u$}}
					\put (36,18) {\scalebox{1}{$v$}}
					\put (36,70) {\scalebox{1}{$A^{(1)}$}} 
				\end{overpic}
				
				$M^{(1)}$
			\end{minipage}\hspace{0.5em}
			\begin{minipage}{.3\textwidth}
				\centering
				\begin{overpic}[scale=1]{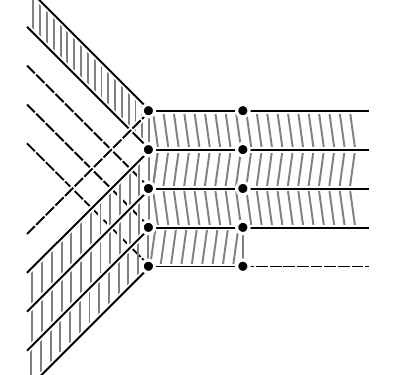} 
					\put (47,18) {\scalebox{1}{$t_1$}}
					\put (60,18) {\scalebox{1}{$u$}}
					\put (36,18) {\scalebox{1}{$v$}}
					\put (36,70) {\scalebox{1}{$A^{(2)}$}}  
				\end{overpic}
				
				$M^{(2)}$
			\end{minipage}\\
			\begin{minipage}{.3\textwidth}
				\centering
				{\small $
					\begin{aligned}
					&\dfrac {c(e_4)} {k_4} + \sigma_0(J_0, 4) = 0
					\end{aligned}$
				}
			\end{minipage}\hspace{0.5em}
			\begin{minipage}{.3\textwidth}
				\centering
				{\small $
					\begin{aligned}
					&|e'| = k_4 - 1\\
					&\Aq \sigma 1(J_0, 1) = \sigma_0(J_0, 4)\\
					&c^{(1)} = \dfrac {c(e_4)} {k_4-1} + \Aq \sigma 1(J_0, 1) 
					\end{aligned} $
				}
			\end{minipage}\hspace{0.5em}
			\begin{minipage}{.3\textwidth}
				\centering
				{\small $
					\begin{aligned}
					&|e'| = k_4 + 1\\
					&\Aq \sigma 2(J_0, 1) = \sigma_0(J_0, 4)\\
					&c^{(2)} = \dfrac {c(e_4)} {k_4+1} + \Aq \sigma 2(J_0, 1) 
					\end{aligned} $
				}
			\end{minipage}
			\captionof{figure}{} \label{figure-29}
		\end{minipage}
		
		\noindent We conclude that:
		\begin{align*} 
		&(k_4-1)c^{(1)} + (k_4+1)c^{(2)} 
		= 2c(e_4) + 2k_4\sigma_0(J_0, 4) = 0.
		\end{align*}

		Analogous constructions exist for $\alpha = 2$, which gives $\Aq M 3$, $\Aq M 4$; and $\alpha = 3$, which gives $\Aq M 5$, $\Aq M 6$. 
		We verify Equation~\eqref{eq-star}:
		\begin{align} \label{eq-3}
		& \left( (k_4-1)c^{(1)} + (k_4+1)c^{(2)} \right) + 
		\left( (k_2-1)c^{(3)} + (k_2+1)c^{(4)} \right) + \\ \nonumber
		& \left( (k_3-1)c^{(5)} + (k_3+1)c^{(6)} \right) = 0 + 0 + 0 = 0.
		\end{align}
		
		\vspace{1em}
		
		\item Case \{w3-r1-nd3-t3\}: Assume that $e_2$ is above $t_3$. Then $k_2 + k_3 \le |A_0|$. Recall that $2|A_0| = k_2 + k_3 + k_4$ and $|A_0| \ge k_i$, so $k_4 = |A_0|$, $k_2 + k_3 = |A_0|$ and $|A_0| > k_2, k_3$. We argue that $\Aq T q$ determines $\ndGqAo$. If $\alpha$ were 2, then by Lemma~\ref{lemma-vertices-in-GqA0} all the vertices of $\ndGqAo$ would be above $v$, so $\ndGqAo$ would have a single vertex with $\Aq r q$-value 1 above trivalent $v$, contradicting that $\Aq M q$ is change-minimal. For $\alpha = 3,4$, we get that $\ndGqAo$ has vertices above $u$ and $v$. The vertices above $v$ belong to the Case~(r0-nd2) of the local properties, since Case~(r0-nd3) would produce a non-dangling element above~$t_2$. So if $\alpha = 3$, we get $\Aq M 1$ where $\Aq {e_2} q$, $\Aq {e_3} q$ have ends above $v$, and by Case~(r0-nd2) these determine the two edges $e'$, $e''$ of $\ndGqAo$; similarly, if $\alpha = 4$ we get $\Aq M 2$ where $\Aq {e_4} q$ has an end above $v$ and this determines the edge $e'$ of $\ndGqAo$. So the vertex of $\ndGqAo$ above $u$ is $\Aq A q$, and it belongs to Case~(r1-nd3). So $\Aq A q$ equals $e_4$, equals $A_0$ as subsets of $[d]$. 
		%
		%
		%
		%
		%
		This gives $\Aq M 1$ and $\Aq M 2$ for $\alpha = 3,4$, respectively. See Figure~\ref{figure-30}.
		\vspace{1em}
		
		\noindent		
		\begin{minipage}{\textwidth}
			\begin{minipage}[t]{.3\textwidth}
				\centering 
				\begin{overpic}[scale=1]{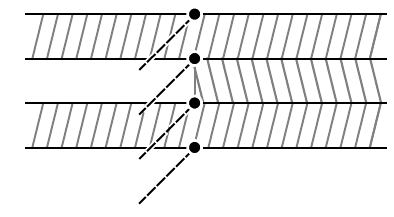}  
					\put (25,9) {\scalebox{1}{$t_3$}}
					\put (62,9) {\scalebox{1}{$t_4$}}
					\put (46,9) {\scalebox{1}{$w_0$}}
					\put (-5,44) {\scalebox{1}{$e_2$}}	
					\put (-5,21) {\scalebox{1}{$e_3$}}	
					\put (98,34) {\scalebox{1}{$e_4$}}	 
					\put (46,55) {\scalebox{1}{$A_0$}}
				\end{overpic}

				$M_0$ \\ \vspace{0.5em}
				\centering
				{\small
					$\begin{aligned}
					&\dfrac {c(e_2)} {k_2}  + \dfrac {c(e_3)} {k_3} + \sigma_0(J_0, 3) = 0\\
					&\dfrac {c(e_4)} {k_2 + k_3}  + \sigma_0(J_0, 4) = 0
					\end{aligned}$ } 
			\end{minipage}\hspace{0.5em}
			\begin{minipage}[t]{.3\textwidth} 
				\centering
				\begin{overpic}[scale=1]{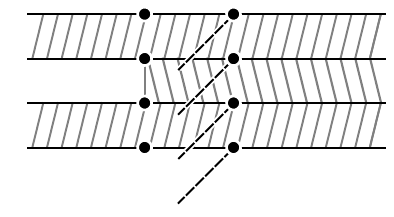}   
					\put (42,10) {\scalebox{1}{$t_1$}}
					\put (33,10) {\scalebox{1}{$u$}}
					\put (57,10) {\scalebox{1}{$v$}}
					\put (34,55) {\scalebox{1}{$A^{(1)}$}}
				\end{overpic}  
				
				$\Aq M 1$ \\ \vspace{0.5em} {\small
					$\begin{aligned}
					&|e'| = k_4 = k_2 + k_3\\
					&\Aq \sigma 1(J_0, 1) = \sigma_0(J_0, 3)\\
					&\Aq c 1 = \dfrac {c(e_4)} {k_2 + k_3} + \Aq \sigma 1(J_0, 1)
					\end{aligned}$ }
			\end{minipage}\hspace{0.5em}
			\begin{minipage}[t]{.3\textwidth} 
				\centering
				\begin{overpic}[scale=1]{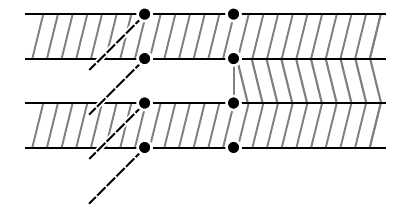}  
					\put (45,10) {\scalebox{1}{$t_1$}}
					\put (35,10) {\scalebox{1}{$v$}}
					\put (57,10) {\scalebox{1}{$u$}}
					\put (56,55) {\scalebox{1}{$A^{(2)}$}}
				\end{overpic}
				
				$\Aq M 2$ \\ \vspace{0.5em} {\small
					$\begin{aligned}
					&|e'| = k_2 \quad \quad |e''| = k_3\\
					&\Aq \sigma 2(J_0, 1) = \sigma_0(J_0, 4)\\
					&\Aq c2 = \dfrac {c(e_2)} {k_2}  + \dfrac {c(e_3)} {k_3} \\&\qquad+  \Aq \sigma 2(J_0, 1)
					\end{aligned}$ } 
			\end{minipage}
			\captionof{figure}{} \label{figure-30}
		\end{minipage}

		\noindent We verify Equation~(\ref{eq-star}):
		\begin{align} \label{eq-4}
		&c^{(1)} + c^{(2)} = \\
		&\left( \tfrac 1 {k_2 + k_3 } c(e_4) + \sigma_0(J_0, 4) \right) +
		\left( \tfrac 1 {k_2} c(e_2) + \tfrac 1 {k_3} c(e_3) + \sigma_0(J_0, 3) \right)  = 0 + 0 = 0. \nonumber 
		\end{align}
		
		\vspace{1em}
		
		\item Case \{w3-r1-nd2\}: Assume that $\nddeg A_0$ is 2. Let $e, e'$ be in $\Neighnd{A_0}$. They are in the same edge $h$ of $H_0$. Assume that $|e| \le |e'|$. By no-return we may assume that $e$ and $e'$ are above $t_3$ and $t_4$ respectively. Lemma~\ref{lem-rphi-nd} (nd.~$r_\dtmor$ formula) gives $r_\varphi(A_0) = 1 = 2|A_0| - |e'| - |e|$. Thus, if $k=|e|$, then $|A_0| = |e'| = k + 1$. From here, follow a  reasoning analogous to case \{w3-nd3-t2\} to show that $\Aq T q$ determines $\Aq M q$, and that the vertices of $\ndGqAo$ above $v$ belong to Case~(r0-nd2) of the local properties, and the ones above $u$ to Case~(r1-nd2). See Figure~\ref{figure-31}.
		\vspace{1em}
		
		\noindent
		\begin{minipage}{\textwidth}
			\begin{minipage}[t]{.3\textwidth} 
				\centering 
				\begin{overpic}[scale=1]{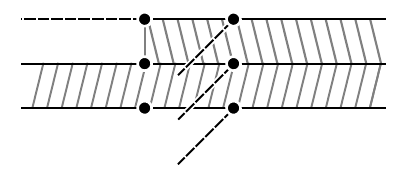} 
					\put (25,9) {\scalebox{1}{$t_3$}}
					\put (62,9) {\scalebox{1}{$t_4$}}
					\put (46,9) {\scalebox{1}{$w_0$}}
					\put (-2,20) {\scalebox{1}{$e$}}
					\put (98, 27) {\scalebox{1}{$e'$}}	
					\put (46,43) {\scalebox{1}{$A_0$}}
				\end{overpic}
				
				$M_0$
				\begin{align*}
				&\dfrac{c_h} {k } + \sigma_0(J_0, 3) = 0\\
				&\dfrac{c_h}{k+1} + \sigma_0(J_0, 4) = 0
				\end{align*}
			\end{minipage}\hspace{0.5em}
			\begin{minipage}[t]{.3\textwidth}
				\centering
				\begin{overpic}[scale=1]{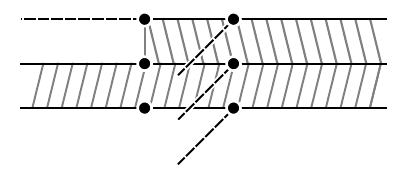}  
					\put (45,10) {\scalebox{1}{$t_1$}}
					\put (35,10) {\scalebox{1}{$u$}}
					\put (57,10) {\scalebox{1}{$v$}}
				\end{overpic}
				
				$M^{(1)}$
				\begin{align*}
				&\Aq \sigma 1(J_0, 1) = \sigma_0(J_0, 3)\\
				&\Aq c1 = \dfrac {c_h} {k+1} + \Aq \sigma 1(J_0, 1)
				\end{align*}
			\end{minipage}\hspace{0.5em}
			\begin{minipage}[t]{.30\textwidth} 
				\centering
				\begin{overpic}[scale=1]{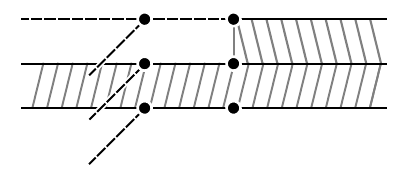} 
					\put (45,10) {\scalebox{1}{$t_1$}}
					\put (35,10) {\scalebox{1}{$v$}}
					\put (57,10) {\scalebox{1}{$u$}}
				\end{overpic}
				
				$M^{(2)}$
				\begin{align*} 
				&\Aq \sigma 2(J_0, 1) = \sigma_0(J_0, 4)\\
				&\Aq c 2 = \dfrac {c_h} {k} +  \Aq \sigma 2(J_0, 1)
				\end{align*} 
			\end{minipage}
			\captionof{figure}{} \label{figure-31}
		\end{minipage}
		
		\noindent We verify Equation~(\ref{eq-star}):
		\begin{align} \label{eq-5}
		c^{(1)} + c^{(2)} = 
		\left (  \tfrac {1} {k} c_h + \sigma_0(J_0, 3)\right) +
		\left (\tfrac {1} {k+1} c_h + \sigma_0(J_0, 4)\right)  = 0 + 0 = 0. 
		\end{align}

	\end{itemize}
	
	\subsubsection{Case \{w2\}} \label{sub-deg2nd3} 
	
	Assume that $\vale w_0$ is 2. Then $r_0(A_0) = \vale A_0 - 2$. So if $r_0(A_0) = 0$, then $\vale A_0 = 2$. Hence, $\sigma_0 (J_0, 2) = \sigma_0 (J_0, 3)$, and we denote this quantity by $s$. Also,  $\ch w_0 = 2$ by Equation~\eqref{eq-C}. 
	The trees contracting to $T_0$ are $T_{\varnothing}$ with $\vale u = 1$, $\vale v = 3$; and $T_2$ with $\vale u = \vale v = 2$. If $\Aq T q = T_\varnothing$, then Remark~\ref{rem-leaves-min-change} implies that $\Aq \sigma q(J_0, 1) = \Aq \sigma q (J_1, 1) = 0$, so $\det \Aq A q = \Aq \sigma q (J_2, 1)$. If $\Aq T q = T_2$, then a reasoning analogous to Case~\{w3-r0\} gives
	\begin{align*}
	\Aq \sigma q(J_0, 1) &= \Aq \sigma q (J_0, 2) = s, \\
	\Aq \sigma q(J_0, 1) &= \Aq \sigma q (J_0, 3) = s.
	\end{align*} 
	So we are left with the cases $r_0(A_0) = 2$ and $r_0(A_0) = 1$.
	
	\begin{itemize}[leftmargin=*] 
		\item Case \{w2-r2\}: Assume that $r_0(A_0) = 2$. Then $\vale A_0 = 4$. Since $A_0$ satisfies no-return, let $e_2$, $e_3$ be non-dangling edges in $\Neigh {A_0}$, above $t_2$, $t_3$, respectively. Let $e_1$, $e_4$ be the remaining two edges of $\Neigh {A_0}$. We may assume without loss of generality that $e_1$ is above $t_2$. By the refinement property:
		\[\sum_{\substack{e \in \Neigh{A_0} \\ \dtmor_0(e) = t_2}} |e| = \sum_{\substack{e \in \Neigh{A_0} \\ \dtmor_0(e) = t_3}} |e| = |A_0|.\]
		So in particular $\sum_{i=1}^4 k_i = 2|A_0|$. We explore how $\Aq T q$ affects $\ndGqAo$:
		
		Base I: assume that $\Aq T q$ is $T_{\varnothing}$. Then, by Remark~\ref{rem-leaves-min-change} (change-minimal leaves) we have that $\ndGqAo$ has a vertex $A_u$ above $u$ with $|A_u| = 2$, $\Aq r q(A_u) = 2$, and two incident edges $e'$, $e''$ with $|e'| = |e''| = 1$.  These are all the edges of $\ndGqAo$ by Lemma~\ref{lemma-edges-in-GqA0} and the fact that $r_0(A_0) = 2$. Thus, $A_u$ and the ends $A'$, $A''$ of $e'$, $e''$, respectively, are the vertices of $\ndGqAo$.  We may assume that $\Aq {e_2} q$ and $A'$ are incident. There are two possibilities for $\Neigh {A'}$ and $\Neigh {A''}$. Either $\Aq {e_3} q$ is incident to $A'$ or $A''$.  We call the first case Base~I.a. Since $\Aq r q (A') = \Aq r q(A'') = 0$, we have that $k_2 = |A'| = k_3$, $k_1 = |A''| = k_4$. We call the second case Base~I.b; it implies that $k_2$ equals the cardinality of the edge above $t_3$ incident to $A'$. So either $k_2 = k_4$ and $k_3 = k_1$; or $k_2 = k_1$ and $k_3 = k_4$. In particular, here the two edges of $\Neigh {A_0}$ are above $t_2$, and the other two above $t_3$. See Case~\{w2-r2-nd3-M-11\} for figures of both possibilities.
		
		Base II: assume that $\Aq T q$ is $T_2$. Then the vertices of $\ndGqAo$ belong to Cases~(r0-nd2), (r1-nd2), or (r1-nd3) of the local properties. By Lemma~\ref{lemma-ndval-of-GqA0} at most one vertex of $\ndGqAo$ belongs to Case~(r1-nd3). So by Lemma~\ref{lemma-edges-in-GqA0} we get that $\ndGqAo$ has either one edge $e'$; or two edges $e'$, $e''$. We call them Base~II.1, Base~II.2, respectively.

		\item Case \{w2-r2-nd3\}: Assume that $\nddeg A_0$ is 3. We may assume without loss of generality that $e_4$ is dangling. So $k_4 = 1$ and $2|A_0| - 1= \sum_{i=1}^3 k_i$. If $e_4$ is above $t_3$, then $k_3 + 1 = |A_0| = k_1 + k_2$. If $e_4$ is above $t_2$, then $k_3 = |A_0| = k_1 + k_2 + 1$. 
		%
		%
		We call these two possibilities Cardinality~M and Cardinality~P.
		
		For Base~II we have that exactly one vertex of $\ndGqAo$ belongs to Case~(r1-nd3), so the other vertex of $\ndGqAo$ with $\Aq r q$-value 1 belongs to Case~(r1-nd2). 
		
		In Base~II.1, Lemma~\ref{lemma-vertices-in-GqA0} implies that the ends of $\Aq {e_1} q$, $\Aq {e_2} q$ above $u$ are the same vertex of $\ndGqAo$; by no-return this vertex has non-dangling valency 3, so it is $\Aq A q$. By Case~(r1-nd3) we have that $e_1 \cup e_2$, $\Aq A q$, and $e'$ are equal as subsets of $[d]$; therefore, $|e'|=|\Aq A q| = k_1 + k_2$. Hence, the other end $A'$ of $e'$ (which is above $v$, incident to $\Aq {e_3} q$) belongs to Case~(r1-nd2). Either $|e'| = |A'| = k_3 + 1$, or $|e'| = |A'|-1 = k_3 - 1$. The first possibility implies $k_1 + k_2 = k_3 + 1$, namely Cardinality~M; the second possibility implies $k_1 + k_2 = k_3 - 1$, namely Cardinality~P. We call these possibilities Base~II.1.M and Base~II.1.P, respectively.
		
		In Base~II.2, $\ndGqAo$ has three vertices. By Lemma~\ref{lemma-vertices-in-GqA0} these are the ends $A'$, $A''$ of $\Aq {e_1} q$, $\Aq {e_2} q$, respectively, above $u$,  and the end $B$ of $\Aq {e_3} q$ above~$v$. By no-return we may assume that $e'$, $e''$ are incident to $A'$, $A''$, respectively. Since these are all the edges of $\ndGqAo$, $\nddeg A' = \nddeg A'' = 2$, and $\nddeg B = 3$. So $B$ is $\Aq A q$ and $h(\Aq {e_1} q) = h(e')$, $h(\Aq {e_2} q) = h(e'')$. 
		So Case~(r1-nd3) gives that $k_3 = |\Aq A q| = |e'| + |e''|$. One of $A'$, $A''$ belongs to Case~(r1-nd2) and the other to Case~(r0-nd2). If $A'$ belongs to Case~(r1-nd2) then either $|e'| = |A'|-1 = k_1 - 1$, so $k_1 > 1$ and $(k_1-1) + k_2 = k_3$, that is Cardinality~M; or $|e'| = |A'|+1 = k_1 + 1$, which gives $(k_1+1) + k_2 = k_3$, that is Cardinality~P. We call these possibilities Base~II.2.1.M and Base~II.2.1.P, respectively. Analogously, if $A''$ belongs to Case~(r1-nd2) we get Base~II.2.2.M and Base~II.2.2.P, where the former gives $k_2 > 1$.
		
		\item Case \{w2-r2-nd3-M\}: Assume that $e_4$ is above~$t_3$, namely Cardinality~M. The possibilities are Base~I.a, Base~I.b, Base~II.1.M, Base~II.2.1.M, and Base~II.2.2.M.
		
		\item Case \{w2-r2-nd3-M-11\}: \label{case-w2-r2-nd3-M-11} Assume that $k_1 = 1$ and $k_2 = 1$. Base~II.2.1.M, Base~II.2.2.M are precluded since $k_1 \not>1$, $k_2\not>1$, respectively. Base~I.a, Base~I.b, Base~II.1.M determine $\Aq M 1$, $\Aq M 2$ and $\Aq M 3$, respectively. See Figure~\ref{figure-32}.
		
		\begin{minipage}{.45\textwidth}
			\centering {\small
				\begin{overpic}[scale=1.3]{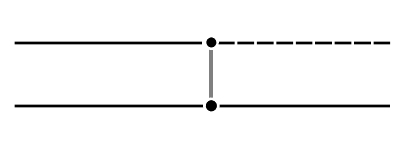} 
					\put (51,29) {\scalebox{1}{$A_0$}}
					\put (35,4) {\scalebox{1}{$t_2$}}
					\put (67,4) {\scalebox{1}{$t_3$}}
					\put (51,4) {\scalebox{1}{$w_0$}}
					\put (-3,25) {\scalebox{1}{$e_1$}}	
					
					\put (-3,10) {\scalebox{1}{$e_2$}}	
					\put (100,10) {\scalebox{1}{$e_3$}}	 
			\end{overpic} }
			
			$M_0$
		\end{minipage}%
		\begin{minipage}{0.45\textwidth}
			\centering
			\begin{align*}
			&c(e_1) + c(e_2) + s= 0\\
			&c(e_1) + s = 0
			\end{align*}
		\end{minipage}
		
		\vspace{1em} 
		\noindent
		\begin{minipage}{\textwidth}
			\begin{minipage}[t]{.30\textwidth}
				\centering 
				\begin{overpic}[scale=1]{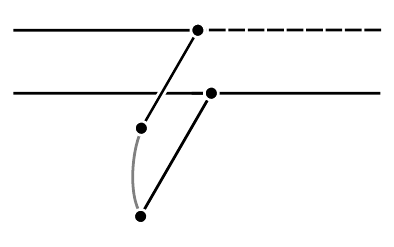}  
					\put (47,16) {\scalebox{1}{$t_1$}}
					\put (52,42) {\scalebox{1}{$A^{(1)}$}} 
					\put (39,4) {\scalebox{1}{$u$}} 
					\put (53,31) {\scalebox{1}{$v$}} 
					\put (-5,53) {\scalebox{1}{$e_1$}}	
					\put (-5,38) {\scalebox{1}{$e_2$}}	
					\put (98,38) {\scalebox{1}{$e_3$}}   
				\end{overpic}
				
				$M^{(1)}$ \\ \vspace{0.5em}
				{\small $
					\begin{aligned}
					&\Aq \sigma 1(J_0, 1) = \Aq \sigma 1(J_1, 1) = 0\\
					&\Aq c1 = 2c(e_1) 
					\end{aligned}$
				}
			\end{minipage}\hspace{0.8em}
			\begin{minipage}[t]{.30\textwidth} 
				\centering
				\begin{overpic}[scale=1]{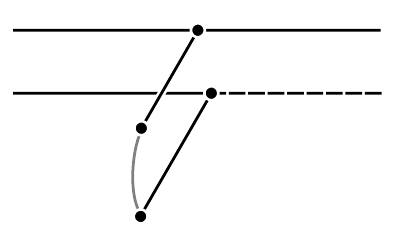}     
					\put (47,16) {\scalebox{1}{$t_1$}}
					\put (48,58) {\scalebox{1}{$A^{(2)}$}}
					\put (39,4) {\scalebox{1}{$u$}} 
					\put (53,31) {\scalebox{1}{$v$}} 
				\end{overpic} 
				
				$M^{(2)}$ \\ \vspace{0.5em}
				\small{
					$\begin{aligned}
					&\Aq \sigma 2(J_0, 1) = \Aq \sigma 2(J_1, 1) = 0\\
					&\Aq c2 = 2c(e_2) 
					\end{aligned}$
				}
			\end{minipage}\hspace{0.8em}
			\begin{minipage}[t]{.30\textwidth} 
				\centering
				\begin{overpic}[scale=1]{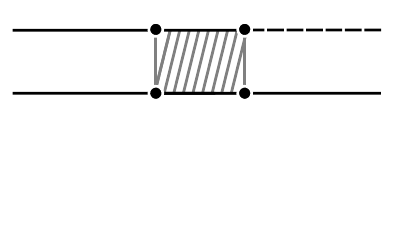}       
					\put (48,31) {\scalebox{1}{$t_1$}}
					\put (37,31) {\scalebox{1}{$u$}} 
					\put (59,31) {\scalebox{1}{$v$}} 
					\put (36,58) {\scalebox{1}{$A^{(3)}$}}
				\end{overpic}
				
				$M^{(3)}$ \\ \vspace{0.5em}
				{\small
					$\begin{aligned}
					&|e| = 2\\ 
					&\Aq \sigma 3(J_0, 1) = s\\
					&c^{(3)} = \dfrac {c(e_3)} 2 +  s
					\end{aligned}$ 
				}
			\end{minipage}
			\captionof{figure}{} \label{figure-32}
		\end{minipage}
		
		\noindent We verify Equation~\eqref{eq-star}:
		\begin{align} \label{eq-6}  c^{(1)} + c^{(2)} + 4 c^{(3)} = 2(c(e_1) + c(e_2) + s) + 2(c(e_3) + s) = 2 \cdot 0 + 2\cdot0 = 0. \end{align}
		
		\vspace{1em}
		
		\item Case \{w2-r2-nd3-M-1k\}: Assume that $k_1 = 1$ and $k_2 \ge 2$. Let $k = k_2$, so $|A_0| = k + 1$ and $|e_3|=k$. Base~I.b and Base~II.2.1.M are precluded since $|e_2|\ne|e_4|$ and $k_1\not>1$, respectively. Base~I.a, Base~II.2.2.M, Base~II.1.M determine $\Aq M 1$, $\Aq M 2$ and $\Aq M 3$, respectively. See Figure~\ref{figure-33}.
		
		\vspace{1em}
		
		\noindent\begin{minipage}{.5\textwidth}
			\centering
			\begin{overpic}[scale=1.3]{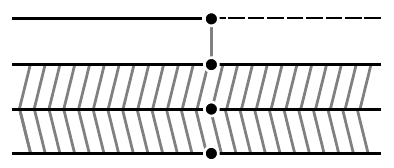}   
				\put (33,-3) {\scalebox{1}{$t_2$}}
				\put (67,-3) {\scalebox{1}{$t_3$}}
				\put (48,-3) {\scalebox{1}{$w_0$}}
				\put (-3,37) {\scalebox{1}{$e_1$}}	
				\put (-3,13) {\scalebox{1}{$e_2$}}	
				\put (98,13) {\scalebox{1}{$e_3$}}	 
				\put (50,42) {\scalebox{1}{$A_0$}}
			\end{overpic} 
			\vspace{1em}
			
			$M_0$
		\end{minipage}%
		\begin{minipage}{0.45\textwidth}
			\centering
			\begin{align*}
			&c(e_1) + \dfrac {c(e_2)} k + s = 0\\
			&\dfrac {c(e_3)} k + s = 0
			\end{align*}
		\end{minipage}  
		
		\vspace{1em} 
		
		\noindent	
		\begin{minipage}{\textwidth}
			\begin{minipage}[t]{.30\textwidth}
				\centering 
				\begin{overpic}[scale=1]{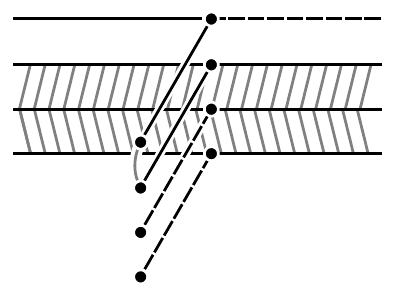} 
					\put (47,15) {\scalebox{1}{$t_1$}}
					\put (52,61) {\scalebox{1}{$A^{(1)}$}} 
					\put (39,2) {\scalebox{1}{$u$}} 
					\put (53,29) {\scalebox{1}{$v$}} 
					\put (-5,69) {\scalebox{1}{$e_1$}}	
					\put (-5,45) {\scalebox{1}{$e_2$}}	
					\put (98,45) {\scalebox{1}{$e_3$}}  
				\end{overpic}
				
				$M^{(1)}$ \\ \vspace{0.5em}
				{\small
					$\begin{aligned}
					&\Aq \sigma 1(J_0, 1) = \Aq \sigma 1(J_1, 1) = 0\\
					&\Aq c 1 = 2c(e_1) 
					\end{aligned}$ } 
			\end{minipage} \hspace{0.5em}
			\begin{minipage}[t]{.3\textwidth} 
				\centering
				\begin{overpic}[scale=1]{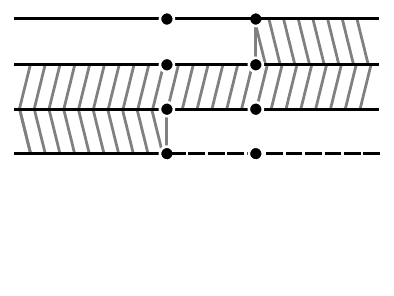}    
					\put (51,28) {\scalebox{1}{$t_1$}}
					\put (40,28) {\scalebox{1}{$u$}}
					\put (63,28) {\scalebox{1}{$v$}}
					\put (62,73) {\scalebox{1}{$A^{(2)}$}}
				\end{overpic}  
				
				$M^{(2)}$ \\ \vspace{0.5em} {\small
					$\begin{aligned}
					&|e'| = 1 \quad |e''| = k - 1\\
					&\Aq \sigma 2(J_0, 1) = s\\
					&\Aq c 2 = c(e_1) + \dfrac {c(e_2)} {k-1} + s
					\end{aligned}$ }
			\end{minipage}\hspace{0.5em}
			\begin{minipage}[t]{.3\textwidth} 
				\centering
				\begin{overpic}[scale=1]{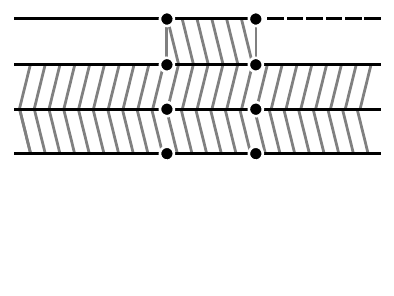}  
					\put (51,28) {\scalebox{1}{$t_1$}}
					\put (40,28) {\scalebox{1}{$u$}}
					\put (63,28) {\scalebox{1}{$v$}}
					\put (39,73) {\scalebox{1}{$A^{(3)}$}}
				\end{overpic}
				
				$M^{(3)}$ \\ \vspace{0.5em} {\small
					$\begin{aligned}
					&|e'| = k + 1\\  
					&\Aq \sigma 3(J_0, 1) = s\\
					&\Aq c 3 = \dfrac {c(e_3)} {k+1} + s
					\end{aligned}$ } 
			\end{minipage}
			\captionof{figure}{} \label{figure-33}
		\end{minipage}
		
		\noindent We verify Equation~(\ref{eq-star}):
		\begin{align} \label{eq-7}
		&c^{(1)} + 2(k-1) c^{(2)} + 2(k+1) c^{(3)} = \\ &2(kc(e_1) + c(e_2) + ks)+ (c(e_3) + ks) = 2 \cdot 0 + 0 = 0. \nonumber
		\end{align}
		
		\vspace{1em}
		
		\item Case \{w2-r2-nd3-M-kk\}: Assume that $k_1 \ge 2$ and $k_2 \ge 2$. Then $k_3 > k_1, k_2$, so Base~I is precluded since $|e_2|\not=|e_3|$ nor $|e_2|\not=|e_4|$. Base~II.2.1.M, Base~II.2.2.M, and Base~II.1.M determine $\Aq M 1$, $\Aq M 2$ and $\Aq M 3$, respectively. See Figure~\ref{figure-34}.
		
		\vspace{1em}
		
		\noindent \begin{minipage}{.5\textwidth}
			\centering
			\vspace{1em}
			\begin{overpic}[scale=1.1]{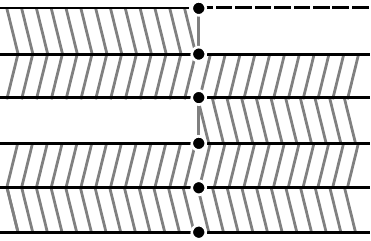}    
				\put (22,-6) {\scalebox{1}{$t_2$}}
				\put (75,-6) {\scalebox{1}{$t_3$}}
				\put (50,-6) {\scalebox{1}{$w_0$}}
				
				\put (-8,49) {\scalebox{1}{$e_1$}}	
				\put (-8,12) {\scalebox{1}{$e_2$}}	
				\put (102,23) {\scalebox{1}{$e_3$}}
				
				\put (50,67) {\scalebox{1}{$A_0$}}
			\end{overpic} 
			\vspace{1em}
			
			$M_0$
		\end{minipage}%
		\begin{minipage}{0.45\textwidth}
			\centering
			\begin{align*}
			&\dfrac {c(e_1)} {k_1} + \dfrac {c(e_2)} {k_2} + s = 0\\
			&\dfrac {c(e_3)} {k_1 + k_2 - 1} + s = 0
			\end{align*}
		\end{minipage}
		
		\vspace{1em}
		
		\begin{minipage}{\textwidth}  
			\noindent	\begin{minipage}[t]{.3\textwidth}
				\centering 
				\begin{overpic}[scale=1]{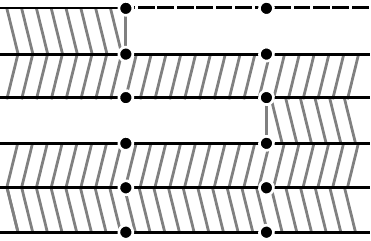}  
					\put (51,-6) {\scalebox{1}{$t_1$}}
					\put (32,-6) {\scalebox{1}{$u$}}
					\put (70,-6) {\scalebox{1}{$v$}}
					\put (62,52) {\scalebox{1}{$\Aq A 1$}} 
					
					\put (-10,49) {\scalebox{1}{$e_1$}}	
					\put (-10,13) {\scalebox{1}{$e_2$}}	
					\put (102,25) {\scalebox{1}{$e_3$}}
				\end{overpic}
				
				\vspace{1em}
				
				$M^{(1)}$ {\small
					\begin{align*}
					&|e'| = k_1 - 1 \quad |e''| = k_2\\
					&\Aq \sigma 1(J_0, 1) = s\\
					&c^{(1)} = \dfrac {c(e_1)} {k_1-1} + \dfrac {c(e_2)} {k_2} + s
					\end{align*} }
			\end{minipage}\hspace{0.5em}
			\begin{minipage}[t]{.3\textwidth} 
				\centering
				\begin{overpic}[scale=1]{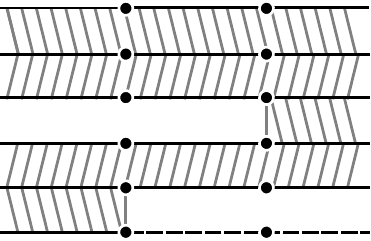}   
					\put (51,-6) {\scalebox{1}{$t_1$}}
					\put (32,-6) {\scalebox{1}{$u$}}
					\put (70,-6) {\scalebox{1}{$v$}}
					\put (65,66) {\scalebox{1}{$\Aq A 2$}}
				\end{overpic}  
				
				\vspace{1em}
				
				$M^{(2)}$
				\begin{align*}
				&|e'| = k_1  \quad |e''| = k_2 - 1\\
				&\Aq \sigma 2(J_0, 1) = s\\
				&\Aq c 2 = \dfrac {c(e_1)} {k_1} + \dfrac {c(e_2)} {k_2-1} +  s
				\end{align*}
			\end{minipage}\hspace{0.5em}
			\begin{minipage}[t]{.3\textwidth} 
				\centering
				\begin{overpic}[scale=1]{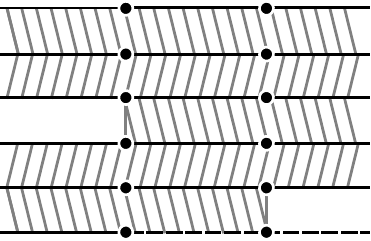} 
					\put (51,-6) {\scalebox{1}{$t_1$}}
					\put (32,-6) {\scalebox{1}{$u$}}
					\put (70,-6) {\scalebox{1}{$v$}}
					\put (30,66) {\scalebox{1}{$\Aq A 3$}}
				\end{overpic}
				
				\vspace{1em}
				
				$M^{(3)}$
				\begin{align*}
				&|e'| = k_1  + k_2\\
				&\Aq \sigma 3(J_0, 1) = s\\
				&\Aq c 3 = \dfrac {c(e_3)} {k_1 + k_2} +  s
				\end{align*}
			\end{minipage}
			\captionof{figure}{} \label{figure-34}
		\end{minipage}
		
		\noindent We verify Equation~(\ref{eq-star}):
		\begin{align}  \label{eq-8}
		&(k_1 - 1) c^{(1)} + (k_2-1) c^{(2)} + (k_1 + k_2) c^{(3)} = \\
		\nonumber
		& \left (\tfrac{k_1 + k_2 - 1}{k_1}{c(e_1)} + \tfrac{k_1 + k_2 - 1}{k_2}{c(e_2)} + (k_1 + k_2 - 1)s \right) + \left ({c(e_3)} + (k_1 + k_2 - 1) s \right )= 0.
		\end{align}
		
		\vspace{1em}
		
		\item Case \{w2-r2-nd3-P\}: Assume that $e_4$ is above~$t_2$, namely Cardinality~P. Base~I is precluded. Base~II.2.1.P, Base~II.2.2.P, Base~II.1.P determine $\Aq M 1$, $\Aq M 2$ and $\Aq M 3$, respectively. See Figure~\ref{figure-35}.  \vspace{0.5em}
		
		\noindent  \begin{minipage}{.5\textwidth}
			\centering
			\begin{overpic}[scale=1.3]{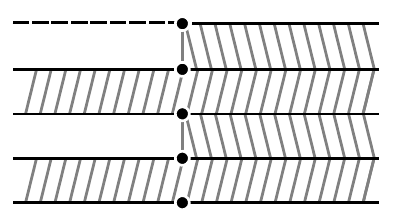}  
				\put (27,-4) {\scalebox{1}{$t_2$}}
				\put (62,-4) {\scalebox{1}{$t_3$}}
				\put (43,-4) {\scalebox{1}{$w_0$}}
				\put (-1,30) {\scalebox{1}{$e_1$}}	
				\put (-1,8) {\scalebox{1}{$e_2$}}	
				\put (98,25) {\scalebox{1}{$e_3$}}	 
				\put (43,52) {\scalebox{1}{$A_0$}}
			\end{overpic} 
			\vspace{1em}
			
			$M_0$
		\end{minipage}%
		\begin{minipage}{0.45\textwidth}
			\centering
			\begin{align*}
			&\dfrac {c(e_1)} {k_1}  + \dfrac {c(e_2)} {k_2}  + s = 0 \\
			&\dfrac {c(e_3)} {k_1 + k_2 + 1}  + s = 0
			\end{align*}
		\end{minipage}
		
		\vspace{1.5em} 
		
		\begin{minipage}{\textwidth}
			\noindent	\begin{minipage}[t]{.3\textwidth}
				\centering 
				\begin{overpic}[scale=1]{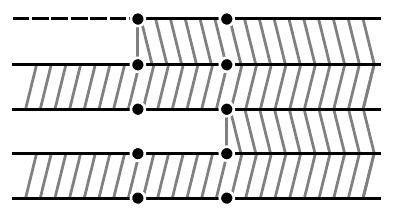}  
					\put (44,-3) {\scalebox{1}{$t_1$}}
					\put (33,-3) {\scalebox{1}{$u$}}
					\put (55,-3) {\scalebox{1}{$v$}}
					\put (54,54) {\scalebox{1}{$A^{(1)}$}} 
					
					\put (-3,30) {\scalebox{1}{$e_1$}}	
					\put (-3,9) {\scalebox{1}{$e_2$}}	
					\put (98,25) {\scalebox{1}{$e_3$}}	
				\end{overpic}\vspace{1em}
				
				$M^{(1)}$ \\ \vspace{0.5em}
				{\small
					$\begin{aligned} 
					&|e'| = k_1 + 1 \quad |e''| = k_2\\
					&\Aq \sigma q(J_0, 1) = s\\
					&c^{(1)} = \dfrac {c(e_1)} {k_1+1} + \dfrac {c(e_2)} {k_2} +  s
					\end{aligned}$}
			\end{minipage}\hspace{1em}
			\begin{minipage}[t]{.3\textwidth} 
				\centering
				\begin{overpic}[scale=1]{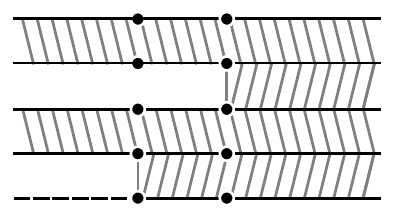}   
					\put (44,-3) {\scalebox{1}{$t_1$}}
					\put (33,-3) {\scalebox{1}{$u$}}
					\put (55,-3) {\scalebox{1}{$v$}}
					\put (32,54) {\scalebox{1}{$A^{(2)}$}}
				\end{overpic}  
				
				\vspace{1em}
				
				$M^{(2)}$ \\ \vspace{0.5em}
				{\small
					$\begin{aligned}
					&|e'| = k_1  \quad |e''| = k_2 + 1\\
					&\Aq \sigma 2(J_0, 1) = s\\
					&c^{(2)} = \dfrac {c(e_1)} {k_1}  + \dfrac {c(e_2)} {k_2+1} +  s
					\end{aligned}$ }
			\end{minipage}\hspace{1em}
			\begin{minipage}[t]{.3\textwidth} 
				\centering
				\begin{overpic}[scale=1]{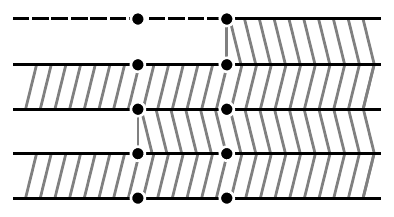} 
					\put (44,-3) {\scalebox{1}{$t_1$}}
					\put (33,-3) {\scalebox{1}{$u$}}
					\put (55,-3) {\scalebox{1}{$v$}}
					\put (54,54) {\scalebox{1}{$A^{(3)}$}}
				\end{overpic}
				
				\vspace{1em}
				
				$M^{(3)}$ \\ \vspace{0.5em}
				{\small
					$\begin{aligned}
					&|e'| = k_1 + k_2\\
					&\Aq \sigma 3(J_0, 1) = s\\
					&c^{(3)} = \dfrac {c(e_3)} {k_1 + k_2} + s
					\end{aligned}$ }
			\end{minipage} 
			\captionof{figure}{} \label{figure-35}
		\end{minipage}
		
		\noindent We verify Equation~(\ref{eq-star}):
		\begin{align} \label{eq-9} &(k_1 + 1) c^{(1)} + (k_2+1) c^{(2)} + (k_1 + k_2) c^{(3)} =\\ 
		\nonumber
		& (k_1 + k_2 + 1)(\tfrac 1 {k_1} c(e_1) + \tfrac 1 {k_2} c(e_2) + s) + (k_1 + k_2 + 1)(\tfrac 1 {k_1 + k_2 + 1} c(e_3) + s) = 0 + 0 = 0.
		\end{align}
		
		\vspace{1em}
		
		\item Case \{w2-r2-nd2\}: Assume that $\nddeg A_0$ is 2. Then both $e_1$, $e_4$ are dangling. Let $h$ be equal to $h(e_2)$ equal to $h(e_3)$.

		\item Case \{w2-r2-nd2-M\}: Recall that $e_1$ is above~$t_2$. Assume that $e_4$ is above~$t_3$. Then $k_2 + 1 = |A_0| = k_3 + 1$. Since $|e_2| = |e_3|$, the columns corresponding to $t_2$ and $t_3$ in $A_0$ are equal. So $M_0$ is not full-rank, contradicting Lemma~\ref{lemma-limit-full-rank}.

		\item Case \{w2-r2-nd2-P\}: Assume that $e_4$ is above $t_2$. Then $k_3 = |A_0| = k_2+2$, and Base~I is precluded. In Base~II all vertices of $\ndGqAo$ belong to Case~(r1-nd2). Thus, $\ndGqAo$ has a single edge $|e'|$. The Case~(r1-nd2) on the end of $e_3$ above $v$ implies that $|e'| = k_3 - 1 = k_2 + 1$. This is the only possibility, shown in Figure~\ref{figure-36}.
		\begin{minipage}{\textwidth}
			\noindent\begin{minipage}[t]{.45\textwidth}
				\centering
				\begin{overpic}[scale=1]{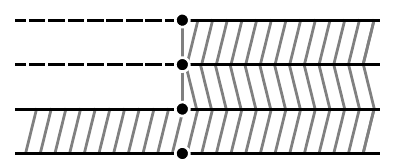} 
					\put (43,41) {\scalebox{1}{$A_0$}}
					\put (28,-4) {\scalebox{1}{$t_2$}}
					\put (61,-4) {\scalebox{1}{$t_3$}}
					\put (43,-4) {\scalebox{1}{$w_0$}}	
					\put (-3,7) {\scalebox{1}{$e_2$}}	
					\put (97,18) {\scalebox{1}{$e_3$}}	 
				\end{overpic}
				
				\vspace{1em}
				
				$M_0$ \\ \vspace{0.5em}
				$\begin{aligned}
				&s = \sigma_0(J_0, 2) = \sigma_0(J_0, 3)\\
				&\dfrac{c_h}{k_2} + s = 0, \quad
				\dfrac{c_h}{k_2+2} + s = 0\\
				\end{aligned} $
				
			\end{minipage}\hspace{0.8em}
			\begin{minipage}[t]{.45\textwidth} 
				\centering
				\begin{overpic}[scale=1]{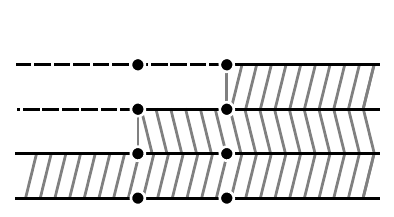}  
					\put (45,-4) {\scalebox{1}{$t_1$}}
					\put (32,-4) {\scalebox{1}{$u$}} 
					\put (56,-4) {\scalebox{1}{$v$}} 
				\end{overpic} 
				
				\vspace{1em}
				
				$M^{(1)}$ \\ \vspace{0.5em}
				$\begin{aligned}
				&|e'| = k_2 + 1\\ 
				&\Aq \sigma 1(J_0, 1) = s\\
				&\Aq c 1 = \dfrac{c_h}{k_2+1} + s\\
				\end{aligned}$
				
			\end{minipage}
			\captionof{figure}{} \label{figure-36}
		\end{minipage}
		\vspace{1em}
		
		From $c_h + k_2s = 0$ and $c_h + (k_2+2)s = 0$ it follows that $\Aq c 1 = c_h + (k_2+1)s = 0$, hence $\Aq M 1$ is not full-dimensional.   
		
		\item Case \{w2-r1\}: Assume that $r_0(A_0) = 1$. Then $\vale A_0 = 3$. Since $A_0$ satisfies no-return, let $e_2$, $e_3$ be non-dangling edges in $\Neigh {A_0}$, above $t_2$, $t_3$, respectively. Let $e_1$ be the remaining edge of $\Neigh {A_0}$. We may assume without loss of generality that~$e_1$ is above~$t_2$. By the refinement property $k_1 + k_2 = |A_0| = k_3$. Let $\tilde A$ be the unique vertex of $\ndGqAo$ with $\Aq r q(\tilde A) = 1$. If $\tilde A$ is above $u$ (resp. $v$), then the vertices of $\ndGqAo$ above $v$ (resp. $u$) belong to the Case~(r0-nd2) of the local properties (Case~(r0-nd3) would contradict that $v$ (resp. $u$) is divalent). These facts and Lemma~\ref{lemma-vertices-in-GqA0} determine the classes. 
		
		\item Case \{w2-r1-nd3\}: Assume that $\nddeg A_0$ is 3. We get two elements in $\pcstar {M_0}$, shown in Figure~\ref{figure-37}.
		\vspace{1em}
		
		\noindent    
		\begin{minipage}{\textwidth}
			\begin{minipage}[t]{.30\textwidth}
				\centering 
				\begin{overpic}[scale=1]{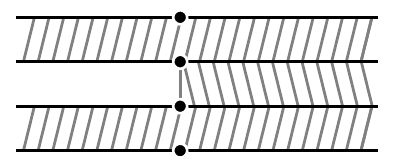}  
					\put (27,-5) {\scalebox{1}{$t_2$}}
					\put (62,-5) {\scalebox{1}{$t_3$}}
					\put (42,-5) {\scalebox{1}{$w_0$}}
					\put (-4,31) {\scalebox{1}{$e_1$}}	
					\put (-4,7) {\scalebox{1}{$e_2$}}	
					\put (96,20) {\scalebox{1}{$e_3$}}	 
					\put (43,42) {\scalebox{1}{$A_0$}}
				\end{overpic}
				
				\vspace{1em}
				
				local part around $A_0$
				\begin{align*}
				&\sigma_0(J_{A_0},2) = \frac {c(e_1)} {k_1}  + \frac {c(e_2)} {k_2}  \\
				&\sigma_0(J_{A_0},3) = \frac {c(e_3)} {k_1 + k_2 } 
				\end{align*}
			\end{minipage}\hspace{0.5em}
			\begin{minipage}[t]{.30\textwidth} 
				\centering
				\begin{overpic}[scale=1]{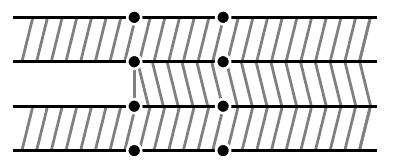} 
					\put (43,-4) {\scalebox{1}{$t_1$}}
					\put (32,-4) {\scalebox{1}{$u$}}
					\put (55,-4) {\scalebox{1}{$v$}}
					\put (31,42) {\scalebox{1}{$A^{(q)}$}}
				\end{overpic}

				\vspace{1em}
				
				$\Aq \dtmor 1 (\tilde A) = u$
				\begin{align*}
				&|e'| = k_3 \\ 
				&\Aq \sigma 1(J_{A_0},1) = \frac {c(e_3)} {k_1 + k_2} 
				\end{align*}
			\end{minipage}\hspace{0.5em}
			\begin{minipage}[t]{.30\textwidth} 
				\centering
				\begin{overpic}[scale=1]{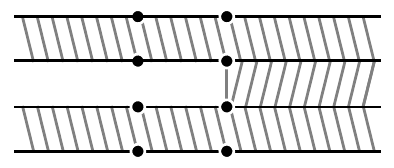} 
					\put (43,-4) {\scalebox{1}{$t_1$}}
					\put (32,-4) {\scalebox{1}{$u$}}
					\put (55,-4) {\scalebox{1}{$v$}}
					\put (54,42) {\scalebox{1}{$A^{(q)}$}}
				\end{overpic}  
				
				\vspace{1em}
				
				$\Aq \dtmor 2 (\tilde A) = v$
				\begin{align*}
				&|e'| = k_1  \quad |e''| = k_2\\
				&\Aq \sigma 2(J_{A_0},1) = \dfrac {c(e_1)} {k_1} + \dfrac {c(e_2)} {k_2}  
				\end{align*}
			\end{minipage}
			\captionof{figure}{} \label{figure-37}
		\end{minipage}
		
		\vspace{1em} 
		
		Thus, \[\Aq \sigma 1(J_{A_0},1) + \Aq \sigma 2(J_{A_0},1) =  \sigma_0(J_{A_0}, 2) + \sigma_0(J_{A_0}, 3).\]
		
		\vspace{1em}
		
		\item Case \{w2-r1-nd2\}: Assume that $\nddeg A_0$ is 2. We may assume that $e_1$ is dangling. So $k_3 = k_2 + 1$. Let $h = h(e_2) = h(e_3)$.  We get two elements in $\pcstar {M_0}$, shown in Figure~\ref{figure-38}.
		
		\vspace{1em}
		
		\noindent
		\begin{minipage}{\textwidth}
			\begin{minipage}[t]{.30\textwidth}
				\centering 
				\begin{overpic}[scale=1]{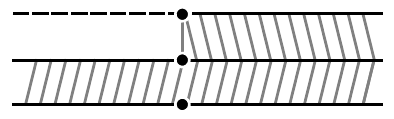}  
					\put (25,-5) {\scalebox{1}{$t_2$}}
					\put (63,-5) {\scalebox{1}{$t_3$}}
					\put (42,-5) {\scalebox{1}{$w_0$}}
					\put (-1,7) {\scalebox{1}{$e_2$}}	
					\put (97,13) {\scalebox{1}{$e_3$}}	 
					\put (43,30) {\scalebox{1}{$A_0$}}
				\end{overpic}
				
				\vspace{1em}
				
				local part around $A_0$
				\begin{align*}
				&\sigma_0(J_{A_0},2) = \dfrac {c_h} {k_1}\\
				&\sigma_0(J_{A_0},3) = \dfrac {c_h} {k_1 + 1} 
				\end{align*}
			\end{minipage}\hspace{0.5em}
			\begin{minipage}[t]{.3\textwidth} 
				\centering
				\begin{overpic}[scale=1]{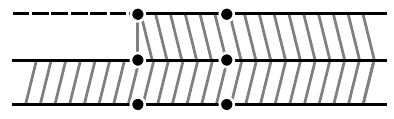} 
					\put (43,-4) {\scalebox{1}{$t_1$}}
					\put (32,-4) {\scalebox{1}{$u$}}
					\put (55,-4) {\scalebox{1}{$v$}}
					\put (31,30) {\scalebox{1}{$A^{(q)}$}}
				\end{overpic}
				
				\vspace{1em}
				
				$\Aq \dtmor 1 (\tilde A) = u$
				\begin{align*}
				&|e'| = k_2 + 1\\ 
				&\Aq \sigma 1(J_{A_0},1)= \dfrac {c_h} {k_2 + 1}  
				\end{align*} 
			\end{minipage}\hspace{0.5em} 
			\begin{minipage}[t]{.3\textwidth} 
				\centering
				\begin{overpic}[scale=1]{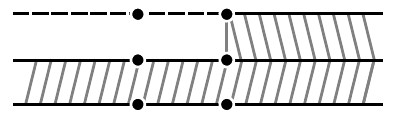}  
					\put (43,-4) {\scalebox{1}{$t_1$}}
					\put (32,-4) {\scalebox{1}{$u$}}
					\put (55,-4) {\scalebox{1}{$v$}}
					\put (54,30) {\scalebox{1}{$A^{(q)}$}}
				\end{overpic}  
				
				\vspace{1em}
				
				$\Aq \dtmor 2 (\tilde A) = v$
				\begin{align*}
				&|e'| = k_2\\
				&\Aq \sigma 2(J_{A_0},1) = \dfrac {c_h} {k_2} 
				\end{align*}
			\end{minipage}
			\captionof{figure}{} \label{figure-38}
		\end{minipage}
		
		\vspace{1em} 
		
		Thus, \[\Aq \sigma 1(J_{A_0},1) + \Aq \sigma 2(J_{A_0},1) = \sigma_0(J_{A_0}, 2) + \sigma_0(J_{A_0}, 3).\]
		
		\vspace{1em}
		
		\item Proof of \eqref{eq-star} for case \{w2-r1\}: There is another vertex $B_0$ above $w_0$ with $r_0(B_0) = 1$. The previous analysis holds for $B_0$, with notation entirely analogous. Note that in $\Aq M q$, $\Aq \dtmor q (\tilde A) \ne \Aq \dtmor q (\tilde B)$ because $\ch u = \ch v = 1$. So $\Aq \dtmor q (\tilde A)$ determines the gluing datum, and it still holds that \[\Aq \sigma 1(J_{B_0},1) + \Aq \sigma 2(J_{B_0},1) =  \sigma_0(J_{B_0}, 2) + \sigma_0(J_{B_0}, 3).\] This gives the following calculation, which verifies Equation~(\ref{eq-star}): 
		\begin{align*}
		\Aq c1 &= \Aq \sigma 1(J_{A_0},1) + \Aq \sigma 1(J_{B_0},1) + s,
		&
		c^{(2)} &= \Aq \sigma 2(J_{A_0},1) + \Aq \sigma 2(J_{B_0},1) + s,
		\end{align*}
		\begin{align} \label{eq-10}
		c^{(1)} + c^{(2)} &= \sigma_0(2) + \sigma_0(3) = 0 + 0 = 0.
		\end{align}
	\end{itemize} 
	
	\subsection*{Conclusion} It is remarkable how diverse the arguments of Section~\ref{sec-constructions} are. By no means do we stand in front of a construction that has been repeated with subtle variations. The richness and diversity of the behaviour of possibly full-dimensional gluing datums defied many attempts of further consolidation into fewer cases. The end result is exhaustive, so Lemma~\ref{lemma-balancing-condition} is verified. This finishes the proof of Theorem~\ref{theorem-gonality}. The method is effective as well; see \cite{dra} for code.
	
	\subsection*{Preview part II} To summarize the past sneak peeks, let $\Aq \dtmor 1$, $\Aq \dtmor 2$ be
	\DTmors that share a limit $\dtmor_0$. The cone $C_{\dtmor_0}$
	may be viewed as a common face to the cones $C_{\Aq \dtmor 1}$
	and~$C_{\Aq \dtmor 2}$. Hence, it is natural to glue at
	$C_{\dtmor_0}$ and mod out by isomorphisms.  The possibilities
	for $C_{\dtmor_0}$ include the cases that we skipped where
	the source of $\dtmor_0$ is not trivalent. The moduli space $\TM d g$ 
	that emerges by gluing all genus-$g$ full-dimensional cones, as mentioned in Subsection~\ref{subsection-gonality}, is the main object in Part~II of
	this series. We study the connectivity of this space and its projection to $\MTrop g$ to get a
	stronger version of Theorem~\ref{theorem-gonality}, in the sense that we obtain a count of realizations. Also, the result becomes independent from the work in \Cite{cd18}.
	
	\sloppy
	\printbibliography

\end{document}